\documentclass[a4paper,11pt]{article}
\usepackage{amsmath,amsthm,amssymb}
\usepackage{enumerate}

\usepackage[utf8]{inputenc}
\usepackage{amsfonts}

\usepackage[english]{babel}
\usepackage[utf8]{inputenc}
\usepackage[a4paper]{geometry}
\usepackage{latexsym}
\usepackage{amscd}
\usepackage{graphics}
\usepackage{color}
\usepackage{array}
\usepackage{mathrsfs}
\usepackage{graphicx}
\usepackage{stmaryrd}
\usepackage{mathabx}
\usepackage{dsfont}
\usepackage{comment}
\usepackage{calligra}
\usepackage{varioref}
\usepackage{prettyref}
\usepackage{tikz}
\usepackage{tikz-cd}

\def\subjclass#1{{\renewcommand{\thefootnote}{}%
\footnote{\emph{Mathematics Subject Classification (2020):} #1}}}

\def\keywords#1{{\renewcommand{\thefootnote}{}%
\footnote{\emph{Keywords:} #1}}}

\def\ackn#1{{\renewcommand{\thefootnote}{}%
\footnote{#1}}}

\newtheorem{thm}{Theorem}[section]
\newtheorem{cor}[thm]{Corollary}
\newtheorem{lem}[thm]{Lemma}

\newtheorem{prop}[thm]{Proposition}
\newtheorem{conj}[thm]{Conjecture}

\newtheorem{question}[thm]{Question}
\newtheorem{claim}[thm]{Claim}

\newrefformat{sec}{Section \ref{#1}}
\newrefformat{subsec}{Subsection \ref{#1}}
\newrefformat{thm}{Theorem \ref{#1}}
\newrefformat{cor}{Corollary \ref{#1}}
\newrefformat{lem}{Lemma \ref{#1}}
\newrefformat{prob}{Problem \ref{#1}}
\newrefformat{prop}{Proposition \ref{#1}}
\newrefformat{conj}{Conjecture \ref{#1}}
\newrefformat{fact}{Fact \ref{#1}}
\newrefformat{question}{Question \ref{#1}}
\newrefformat{defin}{Definition \ref{#1}}
\newrefformat{claim}{Claim \ref{#1}}
\newrefformat{exa}{Example \ref{#1}}
\newrefformat{exs}{Examples \ref{#1}}
\newrefformat{rem}{Remark \ref{#1}}

\theoremstyle{definition}
\newtheorem{defin}[thm]{Definition}
\newtheorem{rem}[thm]{Remark}
\newtheorem{exa}[thm]{Example}
\newtheorem{exs}[thm]{Examples}

\numberwithin{equation}{section}

\newcommand{\N}{\mathbb{N}}
\newcommand{\R}{\mathbb{R}}
\newcommand{\U}{\mathcal{U}}
\newcommand{\X}{\mathcal{X}}
\renewcommand{\H}{\mathcal{H}}

\newcommand{\G}{\mathcal{G}}
\newcommand{\I}{\textbf{I}}
\newcommand{\II}{\textbf{II}}

\renewcommand{\P}{\mathbb{P}}
\newcommand{\Ros}{\mathcal{R}}
\newcommand{\HHP}{$d_2$-HI{}}
\newcommand{\MNH}{$d_2$-minimal{}}
\newcommand{\TNH}{$d_2$-tight{}}
\newcommand{\concat}{%
  \mathbin{\raisebox{1ex}{\scalebox{.7}{$\frown$}}}%
}

\DeclareMathOperator{\supp}{supp}
\DeclareMathOperator{\spa}{span}
\DeclareMathOperator{\Seq}{Seq}
\DeclareMathOperator{\Id}{Id}

\DeclareMathOperator{\Sub}{Sub}
\DeclareMathOperator{\Ban}{Ban}
\DeclareMathOperator{\AP}{AP}
\DeclareMathOperator{\bs}{bs}
\DeclareMathOperator{\Ball}{B}

\newcommand{\Subinf}{\Sub^\infty}
\newcommand{\Subfin}{\Sub^{<\infty}}
\newcommand{\ttau}{\widetilde{\tau}}

\newcommand{\elinf}{[\N]^\infty}
\newcommand{\Ezero}{\mathbf{E}_0}

\newcommand{\Baninf}{\Ban^\infty}
\newcommand{\Banfin}{\Ban^{<\infty}}

\makeatletter
\newcommand{\alaligne}{~\vspace*{\topsep}\nobreak\@afterheading}
\makeatother

\usepackage{glossaries}

\newcommand{\cL}{\mathcal{L}}
\newcommand{\Part}{\mathcal P}

\begin{document}

\title{Local Banach-space dichotomies and  ergodic spaces}

\author{W. Cuellar Carrera \and N. de Rancourt \and V. Ferenczi}

\date{\today}

\maketitle

\begin{abstract}
We prove a local version of Gowers' Ramsey-type theorem \cite{GowersRamsey}, as well as local versions both  of the  Banach space first dichotomy
(the ``unconditional/HI'' dichotomy) of Gowers \cite{GowersRamsey} and of the third dichotomy  (the ``minimal/tight'' dichotomy) due to
Ferenczi--Rosendal \cite{FerencziRosendalMinimal}. This means that we obtain versions of these dichotomies restricted to  certain
families of subspaces called D-families, of which several concrete examples are given. As a main example,
 non-Hilbertian spaces form D-families; therefore versions of the above properties for non-Hilbertian spaces appear in new Banach space dichotomies.
 As a consequence we obtain new information on the number of subspaces of non-Hilbertian Banach spaces, making some progress towards
 the ``ergodic'' conjecture of Ferenczi--Rosendal and towards a question of Johnson.

\end{abstract}


\ackn{Cuellar was supported by FAPESP grants (2018/18593-1) and (2019/23669-0). De Rancourt was supported by FWF Grant P29999 and by the joint FWF–GA\v{C}R Grant No. 17-33849L. Ferenczi was supported by CNPq grants (303034/2015-7) and (303721/2019-2). Cuellar and Ferenczi were supported by FAPESP grant (2016/25574-8). De Rancourt and Ferenczi were supported by USP-COFECUB grant 31466UC.}
\tableofcontents





\

\setcounter{footnote}{0}

\section{Introduction and background}

In this paper, we will only consider real Banach spaces; however, all of our results transpose to the complex case. Unless otherwise specified, when writing about a \textit{Banach space} (or simply a \textit{space}), we shall mean an infinite-dimensional Banach space, by \textit{subspace} of a Banach space, we shall mean infinite-dimensional, closed vector subspace, and by \textit{direct sum}, we shall mean topological direct sum. By \textit{operator}, we shall always mean bounded linear operator. By \textit{Hilbertian space} we mean a space which is linearly isomorphic (but not necessarily isometric) to a Hilbert space.
For all other unexplained notation, see the end of this introduction.

\bigskip

\subsection{Ergodic Banach spaces}

A Banach space is said to be \textit{homogeneous}  if it is isomorphic to all of its (closed, infinite-dimensional) subspaces. A famous problem due to Banach, and known as the \textit{homogeneous space problem}, asked whether, up to isomorphism, $\ell_2$ is the only homogeneous Banach space. The answer turned out to be positive; this problem was eventually solved in the 1990's by a combination of results by Gowers--Maurey \cite{GowersMaurey}, Komorowski--Tomczak-Jaegermann \cite{KomorowskiTomczak}, and Gowers \cite{GowersRamsey}.

\smallskip

The homogeneous space characterization of the Hilbert space shows that, as soon as a separable Banach space $X$ is non-Hilbertian, it should have at least two non-isomorphic subspaces. Thus, the following general question was asked by Godefroy:

\begin{question}[Godefroy]\label{question:Godefroy}
 How many different subspaces, up to isomorphism, can a separable, non-Hilbertian Banach space have? 
\end{question}

This question seems to be very difficult in general, although good lower bounds for several particular classes of spaces are now known. A seemingly simplest particular case of Godefroy's question was formulated by Johnson:

\begin{question}[Johnson]\label{question:Johnson}

Does there exist a separable Banach space having exactly two different subspaces, up to isomorphism?

\end{question}

Even this question is still open. More generally, it is not known whether there exist a separable, non-Hilbertian Banach space with at most countably many different subspaces, up to isomorphism; anticipating on considerations following below, let us note that a positive answer to  Conjecture \ref{conj:Ergodic} would imply that such a space does not exist.
 In the rest of this paper, a separable Banach space having exactly two different subspaces, up to isomorphism, will be called a \textit{Johnson space}.

\smallskip

It turns out that the right setting to study Godefroy's question is the theory of the \textit{classification of definable equivalence relations}. This theory studies equivalence relations $E$ on nonempty standard Borel spaces $X$ which, when seen as subsets of $X^2$, have a sufficiently low descriptive complexity (in general, Borel or analytic). Recall that a \textit{Polish space} is a separable and completely metrizable topological space. A \textit{standard Borel space} is a set $X$ equipped with a $\sigma$-algebra $\mathcal{B}$ such that $\mathcal{B}$ is the Borel $\sigma$-algebra associated to some Polish topology on $X$. When $X$ is a standard Borel space, the $X^n$'s, for $n \geqslant 1$, will always be endowed with the product $\sigma$-algebras; this makes them standard Borel spaces as well. A subset $A$ of a standard Borel space $(X, \mathcal{B})$ is said to be \textit{Borel} if it is an element of $\mathcal{B}$, \textit{analytic} if it is the projection of a Borel subset of $X^2$, and \textit{coanalytic} if its complement is analytic. A \textit{Borel mapping} between two standard Borel spaces is a mapping for which the preimage of every Borel set is Borel, and an \textit{isomorphism} is a Borel bijection (it automatically follows that its inverse is Borel). It is a classical fact in descriptive set theory that all uncountable standard Borel spaces are isomorphic, and that a Borel subset of a standard Borel space is itself a standard Borel space when equipped with the induced $\sigma$-algebra. For proofs of all the forementioned facts, see \cite{Kechris}. 
The central notion of the theory is  \textit{Borel-reducibility}.

\begin{defin}

Let $X, Y$ be nonempty standard Borel spaces, and $E, F$ be equivalence relations on $X$ and $Y$ respectively.

\begin{itemize}

    \item it is said that $E$ \textit{Borel-reduces} to $F$, denoted by $(X, E) \leqslant_B (Y, F)$ (or simply $E \leqslant_B F$) if there is a Borel mapping $f : X \longrightarrow Y$ (called a \textit{reduction}) such that for every $x, y \in X$, we have $x \, E \, y \Leftrightarrow f(x) \, F \, f(y)$.
    
    \item it is said $E$ and $F$ are Borel-equivalent, denoted by $E \equiv_B F$, if $E \leqslant_B F$ and $F \leqslant_B E$.
    
    \item we denote by $E <_B F$ the fact that $E \leqslant_B F$ and $E \nequiv_B F$.
    
\end{itemize}

\end{defin}

\noindent The Borel-reducibility relation defines a hierarchy of complexities on the class of all equivalence relations on standard Borel spaces, the complexity classes being the equivalence classes of $\equiv_B$.

\smallskip

Observe that a reduction $f$ from $(X, E)$ to $(Y, F)$ induces a one-to-one mapping $X/E \rightarrow Y/F$, and in particular, if $E \leqslant_B F$, then $|X/E| \leqslant |Y/F|$. Thus, classes of complexity can be seen as Borel cardinalities: studying the complexity of an equivalence relation gives us at least as much information than counting its classes. If $E$ is analytic and has at most countably many classes, then $E$ is actually Borel and $E \leqslant_B F \Leftrightarrow |X/E| \leqslant |Y/F|$. Thus, for such an $E$, the complexity of $E$ and the number of its classes agree. However, for relations with uncountably many classes, it turns out that the complexity of the relation gives strictly more information than the number of its classes.
The classification of relations with exactly continuum many classes is extremely complex and is actually the main focus of the theory.
\smallskip

We now define a particular equivalence relation that will be important in the rest of this paper. Denote by $\Delta$ the Cantor space, that is, $\{0, 1\}^\N$ with the product topology and the associated standard Borel structure.

\begin{defin}

The equivalence relation $\Ezero$ on $\Delta$ is defined as follows: two sequences $(x_n)_{n \in \N}$ and $(y_n)_{n \in \N}$ are $\Ezero$-equivalent if and only if $x_n = y_n$ eventually.

\end{defin}

It can easily be shown that $(\Delta, =) <_B (\Delta, \Ezero)$; in particular, $\Ezero$ and the equality on the Cantor space are examples of two inequivalent equivalence relations both having continuum-many classes. 
%
%
It follows from important dichotomies by Silver \cite{SilverEqRel} and by Harrington--Kechris--Louveau \cite{hkl} that the following family of equivalence relations:
$$(1, =) <_B (2, =) <_B (3, =) <_B \ldots <_B (\N, =) <_B (\Delta, =) <_B (\Delta, \mathbf{E_0})$$ is an exhaustive initial segment of the whole hierarchy of Borel equivalence relations, in the sense that every Borel equivalence relation $E$ is either Borel-equivalent to some element of this hierarchy, or is strictly above $\Ezero$. Note that this is not true anymore when $E$ is only supposed analytic 
(an analytic equivalence relation $E$ which is strictly above $(\N, =)$ and incomparable to $(\R, =)$ is constructed in \cite{SilverEqRel}).
For a complete presentation of the theory of the classification of definable equivalence relations, see \cite{Kanovei}; note for example that $\Ezero$ is still quite low in the whole hierarchy.

\smallskip

One of the main applications of this theory is the study of the complexity of classification problems in mathematics. When one wants to classify a class $\mathscr{C}$ of objects up to isomorphism, it is often possible to equip $\mathscr{C}$  with a natural Borel structure, for which the isomorphism relation is, in general, analytic. Knowing the complexity of this isomorphism relation gives an indication on the difficulty of the associated classification problem.
For instance, such a class of structures can be classified by real invariants if and only if the isomorphism relation on this class is reducible to $(\R, =)$ (or equivalently, to $(\Delta, =)$). Conversely, if $(\Delta, \Ezero)$ is reducible to the isomorphism relation on this class, this implies that the associated classification problem is quite complex.

\smallskip

One can, in particular, study the classification problem for closed vector-subspaces of a given separable Banach space $X$. To do this, we first need to put a standard Borel structure on $\Sub(X)$; this was first done by Bossard \cite{Bossard}. We refer to his paper for more details and proofs. The set $\Sub(X)$ is endowed with the \textit{Effros Borel structure}, that is, the $\sigma$-algebra generated by sets of the form $\{Y \in \Sub(X) \mid Y \cap U \neq \varnothing\}$, where $U$ ranges over all open subsets of $X$. This makes it a standard Borel space, on which the isomorphism relation is analytic. It is clear from the definition that this Borel structure on $\Sub(X)$ only depends on the isomorphic structure of $X$; in particular, if $T \colon X \to Y$ is an isomorphism between two separable Banach spaces, then $T$ induces a Borel isomorphism between $\Sub(X)$ and $\Sub(Y)$. It is also easy to see that if $Y$ is a subspace of $X$, then the Effros Borel structure on $\Sub(Y)$ coincides with the trace on $\Sub(Y)$ of the Effros Borel structure on $\Sub(X)$. We also mention the following lemma, which will be useful in applications. Here, $\Part(\N)$ is identified with the Cantor space and  if $(x_i)_{i \in I}$ is a family of elements of a Banach space $X$, we will let $[x_i \mid i \in I] = \overline{\spa(x_i \mid i \in I)}$.

\begin{lem}\label{lem:PNtoSub}

Let $X$ be a separable Banach space and let $(x_n)_{n \in \N}$ be a sequence of elements of $X$. Then the mapping $j \colon \Part(\N) \to \Sub(X)$ defined by $j(A) = [x_n \mid n \in A]$ is Borel.

\end{lem}

\begin{proof}

Let $U$ be an open subset of $X$; we prove that $\mathcal{V}:=\{A \in \Part(N) \mid j(A) \cap U \neq \varnothing\}$ is an open subset of $\Part(\N)$, which is enough to conclude. Let $A \in \mathcal{V}$. Then since $U$ is open, $U$ contains a finite linear combination of the $x_n$'s, $n \in A$, so there is a finite $s \subseteq A$ such that $[x_n \mid n \in s] \cap U \neq \varnothing$. In particular, the open neighborhood $\{B \in \Part(\N) \mid s \subseteq B\}$ of $A$ is entirely contained in $\mathcal{V}$.

\end{proof}

Let us mention that the Effros Borel structure can also be used to study the isomorphism relation on the class of \textit{all} finite- and infinite-dimensional separable Banach spaces. Indeed, using the fact that the separable Banach space $\mathcal{C}(\Delta)$ is isometrically universal for this class,
we can identify the class of all finite- and infinite-dimensional separable Banach spaces with $\Sub(\mathcal{C}(\Delta))$. Using this coding, it has been shown by Ferenczi, Louveau and Rosendal \cite{felouro} that the isomorphism relation on the class of all finite- and infinite-dimensional separable Banach spaces is \textit{analytic-complete}, that is, is maximum for $\leqslant_B$ among all analytic equivalence relations on standard Borel spaces. This gives a formal proof of the heuristic fact that there is no reasonable classification of separable Banach spaces, up to isomorphism.

\smallskip

We can also simply study the complexity of the isomorphism relation on $\Sub(X)$ for any separable Banach space $X$; this complexity gives strictly more information that the number of different subspaces of $X$, up to isomorphism, \textit{including the finite-dimensional ones}. So Godefroy's question can be generalized by asking, for spaces $X$ with infinitely many different subspaces up to isomorphism, what is the complexity of the isomorphism relation of $\Sub(X)$. In their investigation on this question, Ferenczi and Rosendal defined the following class of separable Banach spaces in \cite{FerencziRosendalErgodic}:

\begin{defin}

A separable Banach space $X$ is said to be \textit{ergodic} if $\Ezero$ is Borel-reducible to the isomorphism relation on $\Sub(X)$.

\end{defin}

In particular, ergodic Banach spaces have continuum many pairwise non-isomorphic subpaces, and their subspaces cannot be classified by real numbers, up to isomorphism. Immediate consequences of this definition are that $\ell_2$ is non-ergodic, that a subspace of a non-ergodic space is itself non-ergodic, and that the notion of ergodicity is invariant under isomorphism. Ergodic Banach spaces are quite complex and on the contrary,  non-ergodic spaces are expected to be regular in some sense. Ferenczi and Rosendal have shown several regularity properties for non-ergodic spaces. For instance:

\begin{thm}[Ferenczi--Rosendal, \cite{FerencziRosendalOnTheNumber, RosendalCaracErgodic}]\label{square}

Let $X$ be a non-ergodic Banach space with an unconditional basis. Then $X$ is isomorphic to $X \oplus Y$ for every subspace $Y$ spanned by a (finite or infinite) subsequence of the basis. In particular, $X$ is isomorphic to its square and to its hyperplanes.

\end{thm}

\begin{thm}[Ferenczi--Rosendal, \cite{FerencziRosendalErgodic}]

Let $X$ be a non-ergodic separable Banach space. Then $X$ has a subspace $Y$ with an unconditional basis, such that $Y$ is isomorphic to $Y \oplus Z$ for every block-subspace $Z$ of $Y$.

\end{thm}

All these results led them to the following conjecture:

\begin{conj}[Ferenczi--Rosendal]\label{conj:Ergodic}

Every separable non-Hilbertian Banach space is ergodic.

\end{conj}

This conjecture is still open. 
We quote below some of the most relevant partial results supporting the conjecture.

\begin{defin}

A Banach space $X$ is said to be \textit{minimal} if it embeds isomorphically into all of its subspaces.

\end{defin}

The notion of minimality was based on the classical examples of the $\ell_p$'s, $1 \leqslant p < \infty$, and $c_0$ (and their subspaces). Later on the dual of Tsirelson's space, and then Schlumprecht's arbitrarily distortable space were added to the list, see \cite{tsirelson74, casazzashura, schlumprecht}, as well as \cite{CasazzaKaltonKutzarovaMastylo} for variants on Schlumprecht's example. 

\begin{thm}[Ferenczi, \cite{FerencziErgodicMinimal}]\label{thm:FerencziMinimal}

Every non-ergodic separable Banach space contains a minimal subspace.

\end{thm}

It is a consequence of Kwapien's theorem \cite{kwapien} that a space is Hilbertian if and only if there exists a constant $K$ such all its finite-dimensional subspaces are $K$-isomorphic to a Euclidean space. This property may be relaxed as follows:

\begin{defin}\label{defin:AsympHilbert}

A Banach space $X$ is said to be \textit{asymptotically Hilbertian} if there exists a constant $K$ such that for every $n \in \N$, there exists a finite-codimensional subspace $Y$ of $X$ all of whose $n$-dimensional subspaces are $K$-isomorphic to $\ell_2^n$.

\end{defin}

\begin{thm}[Anisca, \cite{Anisca}]\label{thm:AniscaAsympHilbert}

Every asymptotically Hilbertian, non-Hilbertian separable Banach space is ergodic.

\end{thm}

A generalization of the last result will be proved in this paper (see \prettyref{thm:AniscaGeneralization}), using a different method than Anisca's original one.

\begin{thm}[Cuellar Carrera, \cite{Cuellar}]\label{thm:CuellarNearHilbert}

Every non-ergodic separable Banach space has type $p$ and cotype $q$ for every $p < 2 < q$.

\end{thm}

This last result is particularly significant since it shows that counterexamples to \prettyref{conj:Ergodic} should be geometrically very close to be Hilbertian. In particular, the $\ell_p$'s, $1 \leqslant p \neq 2 < \infty$ and $c_0$ are ergodic (this had already been shown by Ferenczi and Galego \cite{FerencziGalego} for the $\ell_p$'s, $1 \leqslant p < 2$, and $c_0$). A consequence of this, combined with James' theorem, is that non-ergodic spaces having an unconditional basis should be reflexive.

\smallskip

We refer to the survey \cite{FerencziRosendalSurvey} as well as to Ferenczi's Th\`ese d'Habilitation (in French) \cite{FerencziHabilitation} for more details. These references list, in particular, better estimates on the complexity of the isomorphism relation between subspaces for several classical Banach spaces.

\smallskip



On the path to possible answers to Johnson's \prettyref{question:Johnson} and of Ferenczi--Rosendal's \prettyref{conj:Ergodic} we identify two weaker conjectures to be studied in the present paper.

\begin{conj}\label{conj:WeakJohnson}
Every Johnson space has an unconditional basis.
\end{conj}

\begin{conj}\label{conj:WeakErgodic}
Every non-ergodic non-Hilbertian separable Banach space contains a non-Hilbertian subspace having an unconditional basis.
\end{conj}


 Conjectures \ref{conj:WeakJohnson} and \ref{conj:WeakErgodic} are important because they allow us to reduce Johnson's and Ferenczi--Rosendal conjectures to the case of spaces having an unconditional basis, for which, as we saw above, we already know many properties. We shall not solve these conjectures, but we make significant progress on them as will appear in \prettyref{sec:Hilbert}.

\bigskip

\subsection{Gowers' classification program}

In order to motivate our forthcoming definitions, we first present the main steps of the solution of the homogeneous space problem. We start with a definition.

\begin{defin}[Gowers--Maurey, \cite{GowersMaurey}]

A Banach space $X$ is \textit{hereditarily indecomposable (HI)} if it contains no direct sum of two subspaces.

\end{defin}

HI spaces exist; they were first built by Gowers and Maurey \cite{GowersMaurey}, as a solution to the unconditional basic sequence problem: they were the first spaces known to contain no subspace with an unconditional basis. Independently of the existence of HI spaces, the combination of the following three results solves positively the homogeneous space problem:

\begin{thm}[Gowers--Maurey, \cite{GowersMaurey}]\label{thm:GowersMaurey}

An HI space is isomorphic to no proper subspace of itself.

\end{thm}

\begin{thm}[Komorowski--Tomczak-Jaegermann, \cite{KomorowskiTomczak}]\label{thm:ktj}

Every Banach space either contains a subspace without unconditional basis, or an isomorphic copy of $\ell_2$.

\end{thm}

\begin{thm}[Gowers' first dichotomy, \cite{GowersRamsey}]\label{thm:Gowers1stDichoto}

Every Banach space either contains a subspace with an unconditional basis, or an HI subspace.

\end{thm}

We refer to \cite{KomorowskiTomczak} for a more precise statement of Theorem \ref{thm:ktj}.
Gowers' first dichotomy is especially important, since it allows to restrict the homogeneous space problem to two special cases, the case of spaces with an unconditional basis and the case of HI spaces. In both of these radically opposite cases, we dispose of specific tools allowing us to solve the problem more efficiently. Based on this remark, 
Gowers suggested in \cite{GowersRamsey} a classification program for separable Banach spaces ``up to subspace''. The goal is to build a list of classes of separable Banach spaces, as fine as possible, satisfying the following requirements:
\begin{enumerate}
    \item the classes are hereditary: if $X$ belongs to a class $\mathcal{C}$ then all subspaces of $X$ also belong to $\mathcal{C}$ (or, in the case of classes defined by properties of bases, all block-subspaces of $X$ belong to $\mathcal{C}$);
    \item the classes are pairwise disjoint;
    \item knowing that a space belongs to a class gives much information about the structure of this space;
    \item every Banach space contains a subspace belonging to one of the classes.
\end{enumerate}
Such a list is in general called a \textit{Gowers list}. The most difficult property to prove among the above is in general 4.; Gowers' first dichotomy proves this property for the two classes of spaces with an unconditional basis and HI spaces, thus showing that these two classes form a Gowers list. In the same paper \cite{GowersRamsey}, Gowers suggests that this list could be refined by proving new dichotomies in the same spirit, and himself proves a second dichotomy. Three other dichotomies were then proved by Ferenczi and Rosendal \cite{FerencziRosendalMinimal}, leading to a Gowers list with $6$ classes (all of whose are now known to be nonempty) and $19$ possible subclasses.

\smallskip

All of these dichotomies draw a border between a class of ``regular'' spaces (spaces sharing many properties with classical spaces such as the $\ell_p$'s, $1 \leqslant p < \infty$, or $c_0$), and a class of ``pathological'' or ``exotic'' spaces. These dichotomies are often important in the study of the problem of the complexity of the isomorphism relation between subspaces of a space $X$; when $X$ is on the ``pathological'' side, we expect this relation to be rather complex. We present below the most important of the dichotomies by Ferenczi and Rosendal (called ``third dichotomy'' in \cite{FerencziRosendalMinimal}), which will be particularly relevant in this paper.

\begin{defin}[Ferenczi--Rosendal]\label{defin:Tight}

\alaligne

\begin{enumerate}

\item Let $(e_n)_{n \in \N}$ be a basis of some Banach space. A Banach space $X$ is \textit{tight in} the basis $(e_n)$ if there is an infinite sequence of nonempty intervals $I_0 < I_1 < \ldots$ of integers such that for every infinite $A \subseteq \N$, we have $X \nsqsubseteq \left[e_n \left| n \notin \bigcup_{i \in A} I_i\right.\right]$.

\item A basis $(e_n)_{n \in \N}$ is said to be \textit{tight} if every Banach space is tight in it. A Banach space $X$ is \textit{tight} if it has a tight basis.

\end{enumerate}

\end{defin}

In the case of reflexive Banach spaces it is known that all bases are tight if one of them is tight \cite[Corollary 3.5]{FerencziRosendalMinimal}. Note that there is a more intuitive characterization of  tightness, see \cite{FerencziGodefroy}. Namely $X$ is tight in $(e_n)$ exactly when the set of $A \subseteq \N$ such that $X$ embeds into $[e_n \mid n\in A]$ is meager (in the natural topology on $\Part(\N)$ obtained by identifying it with the Cantor space). However the definition with the intervals $I_i$ is more operative, allowing for example to distinguish forms of tightness according to the dependence between $X$ and the associated sequence of intervals $(I_i)$.

\begin{thm}[Ferenczi--Rosendal]\label{thm:MinimalTight}

Every Banach space either has a minimal subspace, or has a tight subspace.

\end{thm}

This dichotomy will be referred as \textit{the minimal/tight dichotomy} in the rest of this paper. Here, the ``regular'' class is the class of minimal spaces, and the ``pathological'' class is the class of tight spaces: these spaces are isomorphic to very few of their own subspaces. An example of a tight space is Tsirelson's space (see \cite{FerencziRosendalMinimal}). The minimal/tight dichotomy is a generalization of \prettyref{thm:FerencziMinimal} (which itself improved the main result of \cite{Pelczar}): indeed, it can be shown quite easily that tight spaces are ergodic, which, combined with the dichotomy, shows that non-ergodic separable spaces should have a minimal subspace.

\

Ferenczi--Rosendal's definition of tightness is restricted to Schauder bases. This was not a relevant loss of generality for Theorem \ref{thm:MinimalTight}. For our local versions of this dichotomy, however, it will be important to extend the notion to FDD's. To give a concrete example of our need to use FDD's, note that one may force a space to be non-Hilbertian just by imposing restrictions on the summands of an FDD, without condition on the way they ``add up''; this would of course not be possible with bases. 
The  definition is straightforward, and properties of tight bases extend without harm to tight FDD's:

\begin{defin}[Tight FDD's]\label{defin:TightFDD}

\alaligne

\begin{enumerate}

\item Let $(F_n)_{n \in \N}$ be an FDD of some Banach space. A Banach space $X$ is \textit{tight in}  $(F_n)$ if there is an infinite sequence of nonempty intervals $I_0 < I_1 < \ldots$ of integers such that for every infinite $A \subseteq \N$, we have $X \nsqsubseteq \left[F_n \left| n \notin \bigcup_{i \in A} I_i\right.\right]$.

\item An FDD $(F_n)_{n \in \N}$ is said to be \textit{tight} if every Banach space is tight in it.
\end{enumerate}

It is clear from the definition that if a space is spanned by a tight FDD, then it has a tight subspace.

\end{defin}

\bigskip

\subsection{Local Ramsey theory}\label{subsec:LocalRamsey}

Dichotomies such as Gowers' or Ferenczi--Rosendal's present drawbacks if one wants to deal with problems related to ergodicity. Indeed, $\ell_2$ always belongs to the ``regular'' class defined by those dichotomies, which makes them useless to apply to spaces containing an isomorphic copy of $\ell_2$. Typically, if a space is $\ell_2$-saturated, but non-Hilbertian, then these dichotomies do not provide information on the structure of the space itself.

For this reason, it would be interesting to have dichotomies similar to Gowers' or Ferenczi--Rosendal's, but which avoid $\ell_2$, that are, dichotomies of the form ``every non-Hilbertian Banach space $X$ contains a non-Hilbertian subspace either in $\mathcal{R}$, or in $\mathcal{P}$'', where $\mathcal{R}$ is a class of ``regular'' spaces, and $\mathcal{P}$ is a class of ``pathological'' spaces. Proving such dichotomies is the main goal of this paper.

\smallskip

Gowers' and Ferenczi--Rosendal's dichotomies are proved using combinatorial methods, and especially Ramsey theory. 
 Here, for an infinite $M \subseteq \N$, we denote by $[M]^\infty$ the set of infinite subsets of $M$; we see $\elinf$ as a subset of the Cantor space, endowed with the induced topology.

\begin{thm}[Silver, \cite{SilverRamsey}]\label{thm:Silver}

Let $\mathcal{X} \subseteq \elinf$ be analytic. Then there exists an infinite $M \subseteq \N$ such that either $[M]^\infty \subseteq \X$, or $[M]^\infty \subseteq \X^c$.

\end{thm}

A topological proof of Silver's theorem was obtained by E. Ellentuck \cite{ellentuck74}. Similar topologies to those introduced by Ellentuck will be considered in section \ref{EllentuckTop}.

\smallskip
The proofs of both Gowers' dichotomies in \cite{GowersRamsey} are based on a version of \prettyref{thm:Silver} in the context of Banach spaces, known as \textit{Gowers' Ramsey-type theorem for Banach spaces}.  Here, $\N$ is replaced with a separable Banach space $X$, the set $\X$ becomes a set of normalized sequences in $X$, and the monochromatic set $M$ becomes a subspace of $X$. In this context, a result exactly similar to \prettyref{thm:Silver} does not hold, and the conclusion has to be weakened, using a game-theoretic framework. The exact statement of Gowers' Ramsey-type theorem is a bit technical and will be given in \prettyref{sec:GowersSpaces} (\prettyref{thm:GowersThm}); a more comprehensive presentation of this theory can be found in \cite{todorcevicbleu}, Part B, Chapter IV. The proofs of the dichotomies of Ferenczi and Rosendal in \cite{FerencziRosendalMinimal} use either Gowers' Ramsey-type theorem, or similar methods based on Ramsey theory and games.

\smallskip

If one wants to prove Banach-space dichotomies where the outcome space lies in some prescribed family of subspaces (for instance, non-Hilbertian subspaces), one needs adapted Ramsey-theoretic results. Fortunately, such results exist in classical Ramsey theory; they form a topic usually called \textit{local Ramsey theory}. Here, the word \textit{local} refers to the fact that we want to find a monochromatic subset \textit{locally}; meaning, in a prescribed family of subsets. We present below the local version of Silver's \prettyref{thm:Silver}, due to Mathias \cite{MathiasLocal}. A complete presentation of local Ramsey theory can be found in \cite{todorcevicorange}, Chapter 7.

\begin{defin}

\alaligne

\begin{enumerate}
    
    \item A \textit{coideal} on $\N$ is a nonempty subset $\H \subseteq \elinf$ satisfying, for all $A, B \in \Part(\N)$:
    \begin{enumerate}
        \item if $A \in \H$ and $A \subseteq B$, then $B \in \H$;
        \item if $A \cup B \in \H$, then either $A \in \H$ or $B \in \H$.
    \end{enumerate}
    
    \item The coideal $\H$ is said to be \textit{$P^+$} if for every decreasing sequence $(A_n)_{n \in \N}$ of elements of $\H$, there exists $A_\infty \in \H$ such that for every $n \in \N$, $A_\infty \subseteq^* A_n$ (meaning, here, that $A_\infty \setminus A_n$ is finite).
    
    \item The coideal $\H$ is said to be \textit{selective} if for every decreasing sequence $(A_n)_{n \in \N}$ of elements of $\H$, there exists $A_\infty \in \H$ such that for every $n \in \N$, $A_\infty \setminus \llbracket 0, n \rrbracket \subseteq A_n$.
    
\end{enumerate}
\end{defin}

\begin{thm}[Mathias, \cite{MathiasLocal}]

Let $\H$ be a selective coideal on $\N$, and let $\mathcal{X} \subseteq \elinf$ be analytic. Then there exists $M \in \H$ such that either $[M]^\infty \subseteq \X$, or $[M]^\infty \subseteq \X^c$.

\end{thm}

A local Ramsey theory in Banach spaces has already been developed by Smythe in \cite{Smythe}. There, he proves an analogue of Gowers' Ramsey-type theorem where the outcome space is ensured to lie in some prescribed family $\H$ of subspaces of the space $X$ in which we work. The conditions on the family $\H$ are similar to those in the definition of a selective coideal. However, in the context of Banach spaces, these conditions become quite restrictive and it is not clear that they are met by ``natural'' families in a Banach-space-theoretic sense. Smythe's theory seems to be more adapted to dealing with problems of genericity, as illustrated in \cite{Smythe}.

\smallskip

In this paper, we shall prove a local version of Gowers' Ramsey-type theorem for families $\H$ satisfying weaker conditions, which are closer to the definition of $P^+$-coideals (\prettyref{thm:LocalGowersThm}). This theorem has a weaker conclusion than Smythe's theorem; however, the range of families $\H$ to which it applies is much broader and includes ``natural'' families in a Banach-space-theoretic sense, for instance the family of non-Hilbertian subspaces of a given space. These families, called \textit{D-families}, will be defined and studied in \prettyref{sec:DFamilies}. In order to motivate their definition, we state below a sufficient condition for being a $P^+$-coideal which is well-known to set-theoreticians. This fact is folklore; it is, for instance, an easy consequence of Lemma 1.2 in \cite{Mazur}.

\begin{lem}\label{ppp}

Let $\H$ be a coideal on $\N$. If $\H$ is $G_\delta$ when seen as a subset of the Cantor space, then $\H$ is $P^+$.

\end{lem}

\bigskip

\subsection{Organization of the paper}

After the introductory Section 1, Section 2 is still mainly a background section, presenting the formalism of Gowers spaces, as well as their approximate versions, developed by de Rancourt \cite{RancourtRamseyI} as a generalization of Gowers Ramsey-type theory in Banach spaces, and necessary to prove local dichotomies.

\smallskip

In Section 3 we define and study the notion of D-family, \prettyref{defin:D-family}. In similarity to Lemma \ref{ppp}, a set of subspaces of a Banach space $X$ will be called a D-family if it is closed under finite-dimensional modifications and is $G_\delta$ for 
a certain rather fine
 topology  on the set of subspaces of $X$. This will ensure on one hand that such families have a diagonalization property similar to the $P^+$-property, and on the other hand that they have a good behavior relative to FDD's, so that ``local'' Ramsey theorems, i.e. restricted to subspaces in the D-family, may be hoped for.
The concrete examples of D-families are associated to the important notion of degree $d$,  \prettyref{defin:degree}, which allow us to formalize quantitative estimates relating the finite-dimensional subspaces $F$ of a space $X$, and $X$ itself, by assigning a positive real number $d(X,F)$.  A subspace $Y$ of $X$ is $d$-small,  \prettyref{defin:d-small},
when the degrees $d(Y,F)$ are uniformly bounded for $F \subseteq Y$, and $d$-large otherwise;  the conditions on the definition of degree imply that the family of $d$-large subspaces of $X$ is a D-family,  \prettyref{prop:DegDFamily}. Several classical properties of Banach spaces are equivalent to being $d$-small for a well-chosen degree $d$, for instance being Hilbertian, having a certain fixed type or cotype, or having Gordon-Lewis local unconditional structure \cite{GordonLewis}, \prettyref{exs:Degree}.

\smallskip

 In Section 4, we concentrate on the formalism of approximate Gowers spaces to prove our local version of Gowers' Ramsey-type theorem (\prettyref{thm:LocalGowersThm}) for analytic games. Then we deduce from it a local version of Gowers' first dichotomy (\prettyref{thm:FirstDichoto}). This ``first dichotomy'' in the case of a D-family induced by a degree $d$ may be stated as follows:
 
 \begin{thm}[see \prettyref{thm:FirstDichotofordegrees}]
 
 Let $X$ be a $d$-large Banach space. Then $X$ has a $d$-large subspace $Y$ such that:
 
 \begin{enumerate}
     \item either $Y$ is spanned by a UFDD;
     \item or $Y$ contains no direct sum of two $d$-large subspaces.
 \end{enumerate}
 
 \end{thm}
 
The first alternative is stronger than containing an unconditional basic sequence, and the second one, a ``pathological'' property, is weaker than the HI property.
 
 \smallskip

 In Section 5, we will then prove a local version of the minimal/tight dichotomy, \prettyref{thm:2ndDichoto}. In the case of a degree $d$, this dichotomy may be stated as follows: 
 
 \begin{thm}[see \prettyref{thm:2ndDichotoForDegrees}]\label{thm:2ndDichotoIntro}
 
 Let $X$ be a $d$-large Banach space. Then $X$ has a $d$-large subspace $Y$ such that:
 
 \begin{enumerate}
     \item either $Y$ isomorphically embeds into all of its $d$-large subspaces;
     \item or $Y$ is spanned by an FDD in which every $d$-large Banach space is tight.
 \end{enumerate}
 
 \end{thm}
 
The property satisfied by $Y$ in the first alternative will be called \emph{$d$-minimality}. Note that the word ``minimal" here refers to a minimal element, among $d$-large spaces, for the relation of embedding between subspaces.   So a $d$-minimal space should not be thought of as small is this context; it is
a $d$-large space. The proof of \prettyref{thm:2ndDichotoIntro} is more delicate than for the first local dichotomy; it is inspired by a proof by Rosendal of a variant of the classical minimal/tight dichotomy \cite{RosendalAlphaMinimal} and relies on the formalism of Gowers spaces.
Quite importantly towards the questions of Godefroy and Johnson, we prove that the relation between tightness and ergodicity still holds in the local version. Our precise result, in the case of a degree $d$, is the following:

\begin{thm}[see \prettyref{thm:TNHImpliesErgodic}]

Let $X$ be a $d$-large Banach space spanned by an FDD in which every $d$-large Banach space is tight. Then $X$ is ergodic.

\end{thm}

Consequently a $d$-large and non-ergodic separable space must contain a $d$-minimal subspace, \prettyref{cor:HMinErgodic}, and the study of $d$-minimal spaces turns out to be quite relevant. We end the section with additional observations about the $d$-minimality property, and consequences. For example we generalize the result by Anisca that non-Hilbertian spaces which are asymptotically
Hilbertian must be ergodic (\prettyref{thm:AniscaAsympHilbert}), to the case of $d$-large spaces which are ``asymptotically $d$-small'', \prettyref{thm:AniscaGeneralization}.

\smallskip

Finally in Section 6, we consider the Hilbertian degree $d_2(F)$, defined as the Banach-Mazur distance of $F$ to the euclidean space of the same dimension,
and for which the class of $d$-small spaces is exactly the class of Hilbertian spaces. 
In this case the two dichotomies immediately translate as
(to avoid confusion let us insist on the fact that each of 1. and 2. states a dichotomy, but 1. versus 2. is not):

\begin{thm}
Every non-Hilbertian Banach space contains a non-Hilbertian subspace which:
\begin{enumerate}
 \item either is spanned by a UFDD, or does not contain any direct sum of non-Hilbertian subspaces,
 \item either isomorphically embeds into all its non-Hilbertian subspaces, or has an FDD in which every non-Hilbertian space is tight.
\end{enumerate}

\end{thm}

We therefore give some applications of the theory developed 
in the previous sections for the study of ergodicity and Johnson's question, applying these new dichotomies using only non-Hilbertian subspaces.
We reproduce two of our results below as an illustration:

\begin{thm}[see \prettyref{cor:Johnsonbasis}]
Let $X$ be a Johnson space. Then $X$ has a Schauder basis; moreover, $X$ has an unconditional basis if and only if it is isomorphic to its square.
\end{thm}

\begin{thm}[see Theorems \ref{thm:ErgodicMNH} and \ref{thm:MNHHHP}]\label{thm:MNHHHPIntro}
Let $X$ be a separable, non-Hilbertian, non-ergodic Banach space. Then $X$ has a non-Hilbertian subspace $Y$ which isomorphically embeds
into all of its non-Hilbertian subspaces, and which moreover satisfies one of the following two properties:

\begin{enumerate}
    \item $Y$ has an unconditional basis;
    \item $Y$ contains no direct sum of two non-Hilbertian subspaces.
\end{enumerate}
\end{thm}

We moreover conjecture that the second alternative in \prettyref{thm:MNHHHPIntro} cannot actually happen, \prettyref{conj:MNHHHP}.
We end the section by identifying non trivial examples of spaces which do not contain direct sums
of non-Hilbertian subspaces, \prettyref{exa:HIplusl2} and \prettyref{exa:ArgyrosRaikofstalis}, and giving a list of open problems.

\bigskip

\subsection{Definitions and notation}\label{subsec:DefNot}

This subsection lists the main classical definitions and notation that will be needed in this work. We denote by $\N$ the set of nonnegative integers, and by $\R_+$ the set of nonnegative real numbers. We denote by $\Baninf$ the class of all (infinite-dimensional) Banach spaces, by $\Banfin$ the class of finite-dimensional normed spaces, and we let $\Ban = \Baninf \cup \Banfin$. Given a Banach space $X$, we denote by $\Subinf(X)$ the set of (infinite-dimensional, closed) subspaces of $X$, by $\Subfin(X)$ the set of finite-dimensional subspaces of $X$, and we let $\Sub(X) = \Subfin(X) \cup \Subinf(X)$. For $Y, Z \in \Sub(X)$, we will say that $Y$ is \emph{almost contained} in $Z$, and write $Y \subseteq^* Z$, when a finite-codimensional subspace of $Y$ is contained in $Z$.

\smallskip

When writing about a Banach space $X$, we will in general assume that it comes with a fixed norm, that we will usually denote by $\|\cdot\|$. The unit sphere of $X$ for this norm will be denoted by $S_X$, and if necessary we will denote by $\delta_{\|\cdot\|}$ the distance induced by this norm. For $x\in X$ and $r \geqslant 0$, we denote by $\Ball(x, r)$ the open ball centered at $x$ with radius $r$ (which is just the empty set when $r=0$).

\smallskip

Given two finite- or infinite-dimensional Banach spaces $X$ and $Y$, the space of  continuous linear operators  from $X$ to $Y$ will be denoted by $\cL(X, Y)$, or simply by $\cL(X)$ when $X = Y$. It will be equipped by the operator norm coming from the norms of $X$ and $Y$, and this norm will also be usually denoted by $\|\cdot\|$. For $C \geqslant 1$, a \textit{$C$-isomorphism} between $X$ and $Y$ is an isomorphism $T \colon X \to Y$ such that $\|T\|\cdot\|T^{-1}\| \leqslant C$. The \textit{Banach-Mazur distance} between $X$ and $Y$, denoted by $d_{BM}(X, Y)$ is the infimum of the $C \geqslant 1$ such that there exists a $C$-isomorphism between $X$ and $Y$ (if $X$ and $Y$ are not isomorphic, then $d_{BM}(X, Y) = \infty$). A space will be called \emph{Hilbertian} if it is at finite Banach-Mazur distance to a Hilbert space, and \textit{$\ell_2$-saturated} if every subspace of $X$ has a Hilbertian subspace. A \emph{$C$-isomorphic embedding} from $X$ into $Y$ is an embedding which is a $C$-isomorphism onto its image. We write $X \sqsubseteq Y$ if $X$ isomorphically embeds into $Y$, and $X \sqsubseteq_C Y$ if $X$ $C$-isomorphically embeds into $Y$.

\smallskip

Two families $(x_i)_{i \in I}$ and $(y_i)_{i \in I}$ of elements of a Banach space $X$ are said to be \textit{$C$-equivalent}, for $C \geqslant 1$, if for every family $(a_i)_{i \in I}$ of reals numbers with finite support, we have: $$\frac{1}{C}\cdot\left\|\sum_{i \in I}a_i y_i\right\| \leqslant \left\|\sum_{i \in I}a_i x_i\right\| \leqslant C\cdot\left\|\sum_{i \in I}a_i y_i\right\|.$$ In this case, there is a unique $C^2$-isomorphism $T \colon \overline{\spa(x_i \mid i \in I)} \to \overline{\spa(y_i \mid i \in I)}$ such that for every $i$, we have $T(x_i) = y_i$. The families $(x_i)$ and $(y_i)$ are simply said to be \textit{equivalent} if they are $C$-equivalent for some $C \geqslant 1$.

\smallskip

In this paper, we will often use the notion of \emph{finite-dimensional decomposition (FDD)}. Recall that an FDD of a space $X$ is a sequence $(F_n)_{n \in \N}$ of nonzero finite-dimensional subspaces of $X$ such that every $x \in X$ can be written in a unique way as $\sum_{n = 0}^\infty x_n$, where $\forall n \in \N \; x_n \in F_n$. In this case there exists a constant $C$ such that for all $x \in X$ and all $n \in \N$, we have $\left\|\sum_{i < n} x_i\right\| \leqslant C \|x\|$. The smallest such $C$ is called the \emph{constant} of the FDD $(F_n)$. A sequence of finite-dimensional subspaces which is an FDD of the closed subspace it generates will simply be called an \emph{FDD}, without more precision. An \emph{unconditional finite-dimensional decomposition (UFDD)} is an FDD $(F_n)_{n \in \N}$ such that for every sequence $(x_n)_{n \in \N}$ with $x_n \in F_n$ for all $n$, if $\sum_{n=0}^\infty x_n$ converges, then this convergence is unconditional. In this case, there is a constant $K$ such that for all such sequences $(x_n)$, and for every sequence of signs $(\varepsilon_n)_{n \in \N} \in \{-1, 1\}^\N$, we have $\left\|\sum_{i < n} x_n\right\| \leqslant K \left\|\sum_{i < n} \varepsilon_n x_n\right\|$. The smallest such $K$ is called the \emph{unconditional constant} of the UFDD $(F_n)$.

\smallskip

Fix $(F_n)_{n \in \N}$ an FDD of a space $X$. For $x = \sum_{n = 0}^\infty x_n \in X$, the \emph{support} of $x$ is $\supp(x) = \{n \in \N \mid x_n \neq 0\}$. For $A \subseteq X$, we let $\supp(A) = \bigcup_{x \in A} \supp(x)$. A \emph{blocking} of $(F_n)$ is a sequence $(G_n)_{n \in \N}$ of finite-dimensional subspaces of $X$ for which there exists a partition of $\N$ into nonempty successive intervals $I_0 < I_1 < \dots$ such that for every $n$, $G_n = \bigoplus_{i \in I_n} F_i$. A \textit{block-FDD} of $(F_n)$ is a sequence $(G_n)_{n \in \N}$ of nonzero finite-dimensional subspaces of $X$ such that $\supp(G_0) < \supp(G_1) < \ldots$ (here, for two nonempty sets of integers $A$ and $B$, we write $A < B$ for $\forall i \in A \; \forall j \in B \; i < j$). A blocking is a particular case of block-FDD. A block-FDD of $(F_n)$ is itself an FDD, and its constant is less than or equal to the constant of $(F_n)$; moreover, if $(F_n)$ is a UFDD, then a block-FDD of $(F_n)$ is also a UFDD, and its unconditional constant is less than or equal to this of $(F_n)$. A \textit{block-sequence} of $(F_n)$ is a sequence $(x_n)_{n \in \N}$ of vectors of $X$ such that $(\R x_n)_{n \in \N}$ is a block-FDD of $(F_n)$. Such a sequence is a basic sequence, with constant less than or equal to the constant of the FDD $(F_n)$.

\smallskip

If $(F_i)_{i \in I}$ is a family of finite-dimensional subspaces of a Banach space $X$, we will let $[F_i \mid i \in I] = \overline{\sum_{i \in I}F_i}$. This notation will often (but not only) be used in the case where $(F_i)$ is a (finite or infinite) subsequence of an FDD. 

\smallskip

For $C \geqslant 1$, a \textit{$C$-bounded minimal system} in a Banach space $X$ is a family $(x_i)_{i \in I}$ of nonzero elements of $X$ such that for every family $(a_i)_{i \in I}$ of real numbers with finite support and for every $i_0 \in I$, we have $\|a_{i_0}x_{i_0}\| \leqslant C \cdot \left\|\sum_{i \in I}a_i x_i\right\|$. Every separable Banach space contains a countable bounded minimal system whose closed span is the whole space; several more precise results by Terenzi show that such a system can be chosen to have properties that are very close to those of a Schauder basis (see for example \cite{Terenzi98}).
A normalized, $1$-bounded minimal system is called an \textit{Auerbach system}; by Auerbach's lemma (\cite{Biorthogonal}, Theorem 1.16), every finite-dimensional normed space has an Auerbach basis (that is, a basis which is an Auerbach system). A basic sequence with constant $\leqslant C$ is a $2C$-bounded minimal system. But there are other interesting examples. For instance, let $(F_n)_{n \in \N}$ be an FDD with constant $C$. Let, for $n \in \N$, $d_n = \sum_{m < n} \dim(F_m)$, and let $(e_i)_{d_n \leqslant i < d_{n + 1}}$ be an Auerbach basis of $F_n$. Then the sequence $(e_i)_{i \in \N}$ is a $2C$-bounded minimal system.

\smallskip

Given two families $(x_i)_{i \in I}$ and $(y_i)_{i \in I}$ that are $K$-equivalent, if $(x_i)$ is a $C$-bounded minimal system, then $(y_i)$ is a $CK^2$-bounded minimal system. We will also often use the following small perturbation principle for bounded minimal systems:

\begin{lem}

For every $C \geqslant 1$ and every $\varepsilon > 0$, there exists $\delta > 0$ satisfying the following property: if $(x_i)_{i \in I}$ is a $C$-bounded minimal system in a Banach space $X$, if $(y_i)_{i \in I}$ is a family of elements of the same space, and if: $$\sum_{i \in I}\frac{\|x_i - y_i\|}{\|x_i\|} \leqslant \delta,$$ then $(x_i)$ and $(y_i)$ are $(1 + \varepsilon)$-equivalent.

\end{lem}

The proof is routine. This is a classical result for basic sequences,  see for example \cite{kalton}, Theorem 1.3.9, 
 and the proof is exactly the same for bounded minimal systems.

\bigskip\bigskip

\section{Gowers spaces}\label{sec:GowersSpaces}

In this section, we present the formalism of Gowers spaces. This formalism will be our main tool to prove dichotomies. It has been developed by de Rancourt in \cite{RancourtRamseyI}, as a generalisation of Gowers' Ramsey-type theory in Banach spaces developed in \cite{GowersRamsey}. The proofs of all the results presented in this section can be found in \cite{RancourtRamseyI}.

\bigskip

\subsection{Gowers spaces}

For a set $\Pi$, denote by $\Pi^{< \N}$ the set of all finite sequences of elements of $\Pi$. A sequence of length $n$ will usually be denoted by $s = (s_0, \ldots, s_{n-1})$, and the unique sequence of length $0$ will be denoted by $\varnothing$. Let $\Seq(\Pi) = \Pi^{< \N} \setminus \{\varnothing\}$. For $s \in \Pi^{< \N}$ and $x \in \Pi$, the concatenation of $s$ and $x$ will be denoted by $s \concat x$.


\begin{defin}\label{DefGowersSpaces}

A \emph{Gowers space} is a quintuple $\G = (\P, \Pi, \leqslant, \leqslant^*, \vartriangleleft\nolinebreak)$, where $\P$ is a nonempty set (the set of \emph{subspaces}), $\Pi$ is an at most countable nonempty set (the set of \emph{points}), $\leqslant$ and $\leqslant^*$ are two quasiorders on $\P$ (i.e. reflexive and transitive binary relations), and $\vartriangleleft \; \subseteq \Seq(\Pi) \times \P$ is a binary relation, satisfying the following properties:

\begin{enumerate}

\item for every $p, q \in \P$, if $p \leqslant q$, then $p \leqslant^* q$;

\item for every $p, q \in \P$, if $p \leqslant^* q$, then there exists $r \in \P$ such that $r \leqslant p$, $r\leqslant q$ and $p \leqslant^* r$;

\item for every $\leqslant$-decreasing sequence $(p_i)_{i \in \N}$ of elements of $\P$, there exists $p^* \in \P$ such that for all $i \in \N$, we have $p^* \leqslant^* p_i$;

\item for every $p \in \P$ and $s \in\Pi^{<\N}$, there exists $x \in \Pi$ such that $s \concat x \vartriangleleft p$;

\item for every $s \in \Seq(\Pi)$ and every $p, q \in \P$, if $s \vartriangleleft p$ and $p\leqslant q$, then $s \vartriangleleft q$.

\end{enumerate}

We say that $p, q \in \P$ are \emph{compatible} if there exists $r \in \P$ such that $r \leqslant p$ and $r \leqslant q$. To save writing, we will often write $p \lessapprox q$ when $p \leqslant q$ and $q \leqslant^* p$.

\end{defin}

The prototypical example of a Gowers space is the following. Let $K$ be an at most countable field. The \emph{Rosendal space} over $K$ is $\Ros_K = (\P, \Pi, \subseteq, \subseteq^*\nolinebreak, \vartriangleleft)$, where:
\begin{itemize}
    \item $\Pi$ is a countably infinite-dimensional $K$-vector space;
    \item $\P$ is the set of all infinite-dimensional subspaces of $\Pi$;
    \item $\subseteq$ is the usual inclusion relation between subspaces;
    \item $\subseteq^*$ is the almost inclusion, defined by $Y \subseteq^* Z$ iff $Z$ contains a finite-codimensional subspace of $Y$;
    \item $(x_0, \ldots, x_n) \vartriangleleft Y$ iff $x_n \in Y$.
\end{itemize}
Here, we have that $Z \lessapprox Y$ iff $Z$ is a finite-codimensional subspace of $Y$, and $Y$ and $Z$ are compatible iff $Y \cap Z$ is infinite-dimensional.

\smallskip

In the case of the Rosendal space, the fact that $s \vartriangleleft p$ actually only depends on $p$ and on the last term of $s$. This is the case in most usual examples of Gowers spaces; spaces satisfying this property will be called \emph{forgetful Gowers spaces}. In these spaces, we will allow ourselves to view $\vartriangleleft$ as a binary relation on $\Pi \times \P$. However, in the proof of \prettyref{thm:2ndDichoto}, we will use a Gowers space which is not forgetful.

\smallskip

In the rest of this subsection, we fix a Gowers space $\G = (\P, \Pi, \leqslant, \leqslant^*, \vartriangleleft)$. To every $p \in \P$, we associate the four following games:

\begin{defin}\label{defin:GowersGames}

Let $p \in \P$.

\begin{enumerate}

\item \emph{Gowers' game below $p$}, denoted by $G_p$, is defined in the following way:

\smallskip

\begin{tabular}{ccccccc}
\textbf{I} & $p_0$ & & $p_1$ & & $\hdots$ & \\
\textbf{II} & & $x_0$ & & $x_1$ & & $\hdots$ 
\end{tabular}

\smallskip

\noindent where the $x_i$'s are elements of $\Pi$, and the $p_i$'s are elements of $\P$. The rules are the following:

\begin{itemize}

\item for \textbf{I}: for all $i \in \N$, $p_i \leqslant p$;

\item for \textbf{II}: for all $i \in \N$, $(x_0, \ldots, x_i) \vartriangleleft p_i$.

\end{itemize}

\noindent The outcome of the game is the sequence $(x_i)_{i \in \N} \in \Pi^\N$.

\item The \emph{asymptotic game below $p$}, denoted by $F_p$, is defined in the same way as $G_p$, except that this time we moreover require that $p_i \lessapprox p$.

\item The \textit{adversarial Gowers' games} below $p$, denoted by $A_p$ and $B_p$, are obtained by mixing Gowers' game and the asymptotic game. The game $A_p$ is defined in the following way:

\smallskip

\begin{tabular}{cccccccc}
\textbf{I} & & $x_0,  q_0$ & & $x_1,  q_1$ & & $\hdots$ \\
\textbf{II} & $p_0$ & & $y_0,  p_1$ & & $y_1,  p_2 $ & & $\hdots$ 
\end{tabular}

\smallskip

\noindent where the $x_i$'s and the $y_i$'s are elements of $\Pi$, and the $p_i$'s and the $q_i$'s are elements of $\P$. The rules are the following:

\begin{itemize}

\item for \textbf{I}: for all $i \in \N$, $(x_0, \ldots, x_i) \vartriangleleft p_i$ and $q_i \lessapprox p$;

\item for \textbf{II}: for all $i \in \N$, $(y_0 \ldots, y_i) \vartriangleleft q_i$ and $p_i \leqslant p$.

\end{itemize}

\noindent The outcome of the game is the pair of sequences $((x_i)_{i \in \N}, (y_i)_{i \in \N}) \in \left(\Pi^\N\right)^2$.

\item The game $B_p$ is defined in the same way as $A_p$, except that this time we require $p_i \lessapprox p$, whereas we only require $q_i \leqslant p$.

\end{enumerate}

\end{defin}

In this paper, when dealing with games, we shall use a convention introduced by Rosendal: we associate an \emph{outcome} to the game, and define a winning condition in terms of the outcome belonging or not to a determined set. 
For example, saying that player \II{} has a strategy to reach a set $\X \subseteq \Pi^\N$ in the game $G_p$ means that she has a winning strategy in the game whose rules are those of $G_p$ and whose winning condition is the fact that the outcome belongs to $\X$.

\smallskip

We endow the set $\Pi$ with the discrete topology and the set $\Pi^\N$ with the product topology. The two main results about Gowers spaces, proved by de Rancourt in \cite{RancourtRamseyI}, are the following:

\begin{thm}[Abstract Rosendal's theorem, \cite{RancourtRamseyI}]\label{thm:AbstractRosendal}

Let $\X \subseteq \Pi^\N$ be analytic, and let $p\in \P$. Then there exists $q \leqslant p$ such that:
\begin{itemize}
    \item either player \I{} has a strategy to reach $\X^c$ in $F_q$;
    \item or player \II{} has a strategy to reach $\X$ in $G_q$.
\end{itemize}

\end{thm}

\begin{thm}[Adversarial Ramsey principle, \cite{RancourtRamseyI}]\label{thm:AdvRamsey}

Let $\X \subseteq \left(\Pi^\N\right)^2$ be Borel, and let $p\in \P$. Then there exists $q \leqslant p$ such that:
\begin{itemize}
    \item either player \I{} has a strategy to reach $\X$ in $A_q$;
    \item or player \II{} has a strategy to reach $\X^c$ in $B_q$.
\end{itemize}

\end{thm}

\begin{rem}

The definition of the games $A_p$ and $B_p$ we give here is slightly different than the original definition given in \cite{RancourtRamseyI}. This is to save notation in the rest of the paper, and in particular in the proof of \prettyref{thm:2ndDichoto}, which will be quite technical. The version of \prettyref{thm:AdvRamsey} we state above is thus slightly weaker than the original one.

\end{rem}

\smallskip

\prettyref{thm:AbstractRosendal} has been stated and proved by Rosendal in \cite{RosendalGowers} in the special case of the Rosendal space, as a discrete version of Gowers' Ramsey-type theorem in Banach spaces. \prettyref{thm:AdvRamsey} has been proved by Rosendal for $\mathbf{\Sigma}_0^3$ and $\mathbf{\Pi}_0^3$ subsets, in the case of the Rosendal space, in \cite{Rosendaladverse}, where he also conjectured the result for Borel sets, which has been proved by de Rancourt in \cite{RancourtRamseyI}.

\bigskip

\subsection{Approximate Gowers spaces}\label{subsec:ApproxGowersSpaces}

Approximate Gowers spaces are a version of Gowers spaces where the set of points is not anymore a countable set, but a Polish metric space. This formalism is more convenient to obtain approximate Ramsey-type theorems in Banach spaces, for example.

\smallskip

In this section and in the rest of this paper, we use the following notation: if $(\Pi, \delta)$ is a metric space, if $\mathcal{X} \subseteq \Pi^\N$ and if $\Delta = (\Delta_n)_{n \in \N}$ is a sequence of positive real numbers, then we let $(\X)_{\Delta} = \{(x_n)_{n \in \N} \in \Pi^\N \mid \exists (y_n)_{n \in \N} \in \X \; \forall n \in \N \; \delta(x_n, y_n) \leqslant \Delta_n\}$.

\begin{defin}

An \emph{approximate Gowers space} is a sextuple $\G = (\P, \Pi, \delta, \leqslant, \leqslant^*, \vartriangleleft)$, where $\P$ is a nonempty set, $\Pi$ is a nonempty Polish space, $\delta$ is a compatible distance on $\Pi$, $\leqslant$ and $\leqslant^*$ are two quasiorders on $\P$, and $\vartriangleleft \;\subseteq \Pi \times \P$ is a binary relation, satisfying the same axioms 1. -- 3. as in the definition of a Gowers' space and satisfying moreover the two following axioms:

\begin{enumerate}

\setcounter{enumi}{3}

\item for every $p \in \P$, there exists $x \in \Pi$ such that $x \vartriangleleft p$;

\item for every $x \in \Pi$ and every $p, q \in \P$, if $x \vartriangleleft p$ and $p\leqslant q$, then $x \vartriangleleft q$.

\end{enumerate}

\noindent The relation $\lessapprox$ and the compatibility relation on $\P$ are defined in the same way as for a Gowers space.

\end{defin}

With this definition, approximate Gowers spaces are always forgetful, that is, the relation $\vartriangleleft$ is defined as a subset of $\Pi \times \P$ and not as a subset of $\Seq(\Pi) \times \P$ (this technical restriction seems to be needed to prove approximate versions of Theorems \ref{thm:AbstractRosendal} and \ref{thm:AdvRamsey}). In all cases we will encounter in this paper, $\vartriangleleft$ will actually be the membership relation.

\smallskip

The prototypical example of an approximate Gowers space is the following. Let $X$ be a separable Banach space. The \emph{canonical approximate Gowers space over $X$} is\linebreak $\G_E = (\P, S_X, \delta_{\|\cdot\|}, \subseteq, \subseteq^*, \in)$, where:
\begin{itemize}
    \item $\P$ is the set of all (infinite-dimensional) subspaces of $X$;
    \item $S_X$ is the unit sphere of $X$;
    \item $\delta_{\|\cdot\|}$ is the distance on $S_X$ induced by the norm of $X$;
    \item $\subseteq$ is the usual inclusion relation between subspaces;
    \item $\subseteq^*$ is the almost inclusion between subspaces, as defined in \prettyref{subsec:DefNot};
    \item $\in$ is the membership relation between points and subspaces.
\end{itemize}
Here, we have that $Z \lessapprox Y$ iff $Z$ is a finite-codimensional subspace of $Y$, and $Y$ and $Z$ are compatible iff $Y \cap Z$ is infinite-dimensional.

\smallskip

In the context of approximate Gowers spaces, de Rancourt proved in \cite{RancourtRamseyI} an approximate version of \prettyref{thm:AdvRamsey}, but we will not use it in this paper. However, we will introduce an approximate version of \prettyref{thm:AbstractRosendal}. In the rest of this subsection, we fix an approximate Gowers space $\G = (\P, \Pi, \delta, \leqslant, \leqslant^*, \vartriangleleft)$. In this space, Gowers' game $G_p$ is defined in the same way as in Gowers spaces (\prettyref{defin:GowersGames}), apart from the fact that the rule $(x_0, \ldots, x_i) \vartriangleleft p_i$ is obviously replaced with $x_i \vartriangleleft p_i$. We will also define a strengthening of the asymptotic game. Recall that a subset of $\Pi$ is said to be \emph{relatively compact} if its closure in $\Pi$ is compact. In what follows, for $K \subseteq \Pi$ and $p \in \P$, we abusively write $K \vartriangleleft p$ to say that the set $\{x \in K \mid x \vartriangleleft p\}$ is dense in $K$.

\begin{defin}

A \emph{system of relatively compact sets} for the approximate Gowers space $\G$ is a set $\mathcal{K}$ of relatively compact subsets of $\Pi$, equipped with an associative binary operation $\oplus$, satisfying the following property: for every $p \in \P$, and for every $K, L \in \mathcal{K}$, if $K \vartriangleleft p$ and $L \vartriangleleft p$, then $K \oplus L \vartriangleleft p$.

\smallskip

If $(\mathcal{K}, \oplus)$ is a system of relatively compact sets for $\G$ and if $(K_n)_{n \in \N}$ is a sequence of elements of $\mathcal{K}$, then:

\begin{itemize}

\item for $A \subseteq \N$ finite, denote by $\bigoplus_{n \in A} K_n$ the sum $K_{n_1} \oplus \ldots \oplus K_{n_k}$, where $n_1, \ldots, n_k$ are the elements of $A$ taken in increasing order;

\item a \emph{block-sequence} of $(K_n)$ is,  by definition, a sequence $(x_i)_{i \in \N} \in \Pi^\N$ for which there exists an increasing sequence of nonempty sets of integers $A_0 < A_1 < A_2 < \ldots$ such that for every $i \in \N$, we have $x_i \in \bigoplus_{n \in A_i} K_n$.

\end{itemize}

\noindent Denote by $\bs((K_n)_{n \in \N})$ the set of all block-sequences of $(K_n)$.

\end{defin}

In the canonical approximate Gowers space $\G_X$ over a separable Banach space $X$, we can define a natural system of relatively compact sets, $(\mathcal{K}_X, \oplus_X)$, as follows: the elements of $\mathcal{K}_X$ are the unit spheres of finite-dimensional subspaces of $X$ and the operation $\oplus_X$ on $\mathcal{K}_X$ is defined by $S_F \oplus_X S_G = S_{F + G}$. Observe that, given $(F_n)_{n \in \N}$ an FDD of a subspace of $X$, the block-sequences of $(S_{F_n})_{n \in \N}$ in the sense given by the latter definition are exactly the normalized block-sequences of $(F_n)$ in the Banach-theoretic sense.

\begin{defin}

Let $(\mathcal{K}, \oplus)$ be a system of relatively compact sets for $\mathcal{G}$, and $p \in \P$. The \emph{strong asymptotic game below $p$}, denoted by $SF_p$, is defined as follows:

\smallskip

\begin{tabular}{ccccccc}
\textbf{I} & $p_0$ & & $p_1$ & & $\hdots$ & \\
\textbf{II} & & $K_0$ & & $K_1$ & & $\hdots$ 
\end{tabular}

\smallskip

\noindent where the $K_n$'s are elements of $\mathcal{K}$, and the $p_n$'s are elements of $\P$. The rules are the following:

\begin{itemize}

\item for \textbf{I}: for all $n \in \N$, $p_n \lessapprox p$;

\item for \textbf{II}: for all $n \in \N$, $K_n \vartriangleleft p_n$.

\end{itemize}

\noindent The outcome of the game is the sequence $(K_n)_{n \in \N} \in \mathcal{K}^\N$.

\end{defin}

We endow $\Pi^\N$ with the product topology. The following result, proved by de Rancourt in \cite{RancourtRamseyI}, is the approximate version of \prettyref{thm:AbstractRosendal}.

\begin{thm}[Abstract Gowers' theorem, \cite{RancourtRamseyI}]\label{thm:AbstractGowers}

Let $(\mathcal{K}, \oplus)$ be a system of relatively compact sets for $\G$. Let $\X \subseteq \Pi^\N$ be analytic, let $p \in \P$ and let $\Delta$ be a sequence of positive real numbers. Then there exists $q \leqslant p$ such that:

\begin{itemize}

\item either player \I{} has a strategy in $SF_q$ to build a sequence $(K_n)_{n \in \N}$ such that \linebreak $\bs((K_n)_{n \in \N}) \subseteq \X^c$;

\item or player \II{} has a strategy in $G_q$ to reach $(\X)_\Delta$.

\end{itemize}

\end{thm}

From this abstract result, we can easily recover the original Ramsey-type theorem proved by Gowers in \cite{GowersRamsey}, and used in the same article to deduce his first dichotomy (Theorem \ref{thm:Gowers1stDichoto}) along with another dichotomy:

\begin{thm}[Gowers]\label{thm:GowersThm}

Let $X$ be a separable Banach space, $\X \subseteq \left(S_X\right)^\N$ be analytic and $\Delta$ be a sequence of positive real numbers. Then there exists a subspace $Y$ of $X$ such that:

\begin{itemize}

\item either $Y$ has a basis $(y_n)_{n \in \N}$ such that all normalized block-sequences of $(y_n)$ belong to $\X^c$;

\item or player \II{} has a strategy in $G_Y$ to reach $(\X)_\Delta$.

\end{itemize}

\end{thm}

In the statement of this theorem, $G_Y$ denotes the Gowers' game relative to the canonical approximate Gowers space $\G_X$. The original statement proved by Gowers is a bit different in its formulation, however both are easily equivalent. As an illustration of the formalism of approximate Gowers spaces, we now prove \prettyref{thm:GowersThm}.

\begin{proof}[Proof of \prettyref{thm:GowersThm}]

Work in the canonical approximate Gowers space $\G_X$, with the system of relatively compact sets $(\mathcal{K}_X, \oplus_X)$ defined above. Apply \prettyref{thm:AbstractGowers} to $\X$, $p = X$, and $\Delta$. Then either we get a subspace $Y \subseteq X$ such that player \II{} has a strategy in $G_Y$ to reach $(\X)_\Delta$, and we are done, or we get a subspace $Y \subseteq X$ such that player \I{} has a strategy $\tau$ in $SF_Y$ to build a sequence $(K_n)_{n \in \N}$ with $\bs((K_n)_{n \in \N}) \subseteq \X^c$. We can assume that the strategy $\tau$ is such that for every run of the game $SF_Y$:

\smallskip

\begin{tabular}{ccccccc}
\textbf{I} & $Y_0$ & & $Y_1$ & & $\hdots$ & \\
\textbf{II} & & $S_{F_0}$ & & $S_{F_1}$ & & $\hdots$ 
\end{tabular}

\smallskip

\noindent played according to $\tau$, the natural projection $[F_i \mid i < n] \oplus Y_n \to [F_i \mid i < n]$ has norm at most 2. Now consider any run of the game where \I{} plays according to $\tau$ and \II{} plays unit spheres of subspaces of dimension $1$: $S_{\R y_0}, S_{\R y_1}, \ldots$. Then by construction, $(y_n)_{n \in \N}$ is a basic sequence with constant at most $2$, and because \I{} played according to $\tau$, all normalized block-sequences of $(y_n)$ belong to $\X^c$.

\end{proof}

The main goal of next section is to investigate conditions on families $\H$ of subspaces of $X$ for which a local version of \prettyref{thm:GowersThm} can be proved, that is, a version of \prettyref{thm:GowersThm} where we can ensure that the subspace $Y$ given by the theorem is in $\H$. Such a result will be proved in \prettyref{sec:FirstDichotomy}.

\bigskip\bigskip

\section{D-families: definition and examples}\label{sec:DFamilies}

In this section, we introduce the notion of a D-family: these families will be those for which we will be able to prove local Banach-space dichotomies. The ``D'' in the name of D-families both refers to the possibility of proving such dichotomies, and to the fundamental property that one can diagonalize among such families (see \prettyref{lem:Diag} below). We will then give sufficient conditions for being a D-family, and examples.

\bigskip

\subsection{Definition and first properties} \label{EllentuckTop}

As seen in the previous section, the main ingredient to prove dichotomies of a Ramsey-theoretic nature in a given family of subspaces is the possibility to diagonalize among elements of this family. Inspired by Lemma \ref{ppp}, we will define D-families as families of subspaces that are $G_\delta$ for a certain topology. This will ensure, on one hand, that a diagonalisation property similar to this in the definition of a $P^+$-coideal will be satisfied by these families, and on the other hand that they have a good behaviour relative to FDD's.

\smallskip

Fix $X$ a Banach space. For $F \in \Subfin(X)$ and $Y \in \Sub(X)$ such that $F \subseteq Y$, let  $[F, Y] := \{Z \in \Sub(X) \mid F \subseteq Z \subseteq Y\}$; and for $\varepsilon > 0$, let $[F, Y]_\varepsilon^X$ be the set of $Z \in \Sub(X)$ for which there exist $Z' \in [F, Y]$ and an isomorphism $T \colon Z' \to Z$ with $\|T - \Id_{Z'}\| < \varepsilon$ (this latter set will simply be denoted by $[F, Y]_\varepsilon$ when there is no ambiguity on the ambient space $X$). To avoid any misunderstanding, let us note that  this notation $[F,Y]$ should not be confused with the one used to denote the closed linear span of a sequence of vectors or of finite-dimensional subspaces.

\begin{lem}\label{lem:Ellentuck}

The sets $[F, Y]_\varepsilon$, for $\varepsilon > 0$, $F \in \Subfin(X)$ and $Y \in \Sub(X)$ such that $F \subseteq Y$, form a basis for a topology on $\Sub(X)$. Given $Y \in \Sub(X)$, a basis of neighborhoods of $Y$ for this topology is given by the $[F, Y]_\varepsilon$'s, for $\varepsilon > 0$ and $F \subseteq Y$.

\end{lem}

\begin{proof}

What we have to show is that given $\varepsilon_i > 0$, $Y_i \in \Sub(X)$, and $F_i \in \Subfin(X)$ such that $F_i \subseteq Y_i$ for $1 \leqslant i \leqslant n$, and given $Z \in \bigcap_{i = 1}^n [F_i, Y_i]_{\varepsilon_i}$, there exist $\varepsilon > 0$ and a finite-dimensional subspace $F \subseteq Z$ such that $[F, Z]_\varepsilon \subseteq \bigcap_{i = 1}^n [F_i, Y_i]_{\varepsilon_i}$. For each $i$, fix $Z_i \in [F_i, Y_i]$ and $T_i \colon Z_i \to Z$ an isomorphism such that $\|T_i - \Id_{Z_i}\| < \varepsilon_i$. Fix $\varepsilon > 0$ such that for every $i$, $\|T_i - \Id_{Z_i}\| + \varepsilon(1 + \varepsilon_i) < \varepsilon_i$, and let $F = \sum_{i = 1}^n T_i(F_i)$. Then we have $F \subseteq Z$. Now let $W \in [F, Z]_\varepsilon$, and fix $1 \leqslant i \leqslant n$; we show that $W \in [F_i, Y_i]_{\varepsilon_i}$. Fix $W' \in [F, Z]$ and $T \colon W' \to W$ an isomorphism such that $\|T - \Id_{W'}\| < \varepsilon$. Then $T \circ (T_i)_{\restriction T_i^{-1}(W')}$ is an isomorphism from $T_i^{-1}(W')$ to $W$, and we have $T_i^{-1}(W') \in [F_i, Y_i]$. Moreover,
\begin{eqnarray*}
\left\| T \circ (T_i)_{\restriction T_i^{-1}(W')} - \Id_{T_i^{-1}(W')} \right\| & \leqslant & \left\| \left(T - \Id_{W'}\right) \circ (T_i)_{\restriction T_i^{-1}(W')} \right\| + \left\| T_i - \Id_{Z_i} \right\| \\ & \leqslant & \varepsilon (1 + \varepsilon_i) + \left\| T_i - \Id_{Z_i} \right\| \\ & < & \varepsilon_i,
\end{eqnarray*}
concluding the proof.

\end{proof}

The topology on $\Sub(X)$ defined in \prettyref{lem:Ellentuck} will be called the \emph{Ellentuck topology}. It does not depend on the choice of the equivalent norm on $X$. This name was given because of the similarity between this topology and other topologies that arise in the context of Ramsey spaces, and that are also called \emph{Ellentuck}. See \cite{todorcevicorange} for more details. 

\begin{defin}\label{defin:D-family}

A \textit{D-family} of subspaces of $X$ is a family $\mathcal{H} \subseteq \Subinf(X)$ satisfying the two following properties:
\begin{enumerate}
    \item $\H$ is stable under finite-dimensional modifications, i.e. for every $Y \in \Subinf(X)$ and every $F \in \Subfin(X)$, we have $Y \in \mathcal{H}$ if and only if $Y + F \in \mathcal{H}$;
    \item $\mathcal{H}$, seen as a subset of $\Sub(X)$, is $G_\delta$ for the Ellentuck topology.
\end{enumerate}

\end{defin}

We now prove a few properties of D-families. In what follows, we fix $\H$ a D-family of subspaces of $X$.

\begin{defin}

Let $Y \in \Sub(X)$. The \textit{restriction} of $\H$ to $Y$ is the set $\H_{\restriction Y} = \H \cap \Sub(Y)$.

\end{defin}

\begin{lem}
Let $Y \in \Subinf(X)$. The Ellentuck topology on $\Sub(Y)$ coincides with the topology induced on $\Sub(Y)$ by the Ellentuck topology on $\Sub(X)$. In particular, $\H_{\restriction Y}$ is a D-family of subspaces of $Y$.
\end{lem}

\begin{proof}

Observe that for every $\varepsilon > 0$, every $Z \in \Sub(Y)$ and every finite-dimensional subspace $F \subseteq Z$, we have $[F, Z]_\varepsilon^Y = [F, Z]_\varepsilon^X \cap \Sub(Y)$. The left-hand-side of this equality is the general form of a basic neighborhood of $Z$ in the Ellentuck topology on $\Sub(Y)$, and the right-hand-side is the general form of a basic neighborhood of $Z$ in the topology induced on $\Sub(Y)$ by the Ellentuck topology on $\Sub(X)$. Thus, these topologies coincide.

\smallskip

Therefore, since $\H$ is $G_\delta$ for the Ellentuck topology on $\Sub(X)$, its intersection with $\Sub(Y)$ is $G_\delta$ for the Ellentuck topology on $\Sub(Y)$, proving the second part.

\end{proof}

\begin{lem}\label{lem:Diag}
Let $(Y_n)_{n \in \N}$ be a decreasing family of elements of $\H$. Then there exists $Y_\infty \in \H$ such that for every $n \in \N$, $Y_\infty \subseteq^* Y_n$.
\end{lem}

\begin{proof}

Let $(\mathcal{U}_n)_{n \in \N}$ be a decreasing family of Ellentuck-open subsets of $\Sub(X)$ such that $\H = \bigcap_{n \in \N} \mathcal{U}_n$. We define inductively an increasing sequence $(F_n)_{n \in \N}$ of finite-dimensional subspaces of $X$ in the following way. Let $F_0 = \{0\}$. The space $F_n$ being defined, by axiom 1. in the definition of a D-family, the subspace $Y_n + F_n$ is in $\H$, so in $\U_n$; thus there exists a finite-dimensional subspace $F_{n+1} \subseteq Y_n + F_n$ such that $[F_{n+1}, Y_n + F_n] \subseteq \U_n$. We can even assume that $F_n \subseteq F_{n+1}$ and that $\dim(F_{n+1}) \geqslant n+1$. This achieves the construction.

\smallskip

Now let $Y_\infty = \overline{\bigcup_{n \in \N} F_n}$. By construction, for every $n \in \N$ we have $Y_\infty \subseteq F_n + Y_n$, so $Y_\infty \subseteq^* Y_n$. This also implies that $Y_\infty \in [F_{n+1}, Y_n + F_n] \subseteq \U_n$, so finally $Y_\infty \in \H$.

\end{proof}

\begin{cor}\label{cor:DefLocalGowersSpace}

$\G_\H = (\H, S_X, \delta_{\|\cdot\|}, \subseteq, \subseteq^*, \in)$ is an approximate Gowers space.

\end{cor}

\begin{defin}

An \textit{$\H$-good FDD} is an FDD $(F_n)_{n \in \N}$ of a subspace of $X$ such that for every infinite $A \subseteq \N$, the subspace $[F_n \mid n \in A]$ is in $\H$.

\end{defin}

This terminology is motivated by the fact that we want to prove dichotomies where the outcome space is in the family $\H$ (so for instance, non-Hilbertian). In our dichotomies, good FDD's will play a similar role as the basis $(y_n)_{n \in \N}$ in the statement of \prettyref{thm:GowersThm} does.

\begin{lem}\label{lem:GoodBlocking}

Let $(F_n)_{n \in \N}$ be an FDD of a subspace of $X$. Suppose that $[F_n \mid n \in \N] \in \H$. Then there exists a blocking $(G_n)_{n \in \N}$ of $(F_n)$ which is $\H$-good.

\end{lem}

\begin{proof}

Let $(\mathcal{U}_n)_{n \in \N}$ be a decreasing family of Ellentuck-open subsets of $\Sub(X)$ such that $\H = \bigcap_{n \in \N} \mathcal{U}_n$. Let, for every $k \in \N$, $Y_k = [F_l \mid l \geqslant k]$.  We build $(G_n)$ by induction as follows. Suppose that the $G_m$'s have been built for $m < n$, and let $k_n = (\max \, \supp (G_{n - 1})) + 1$ if $n \geqslant 1$, $k_n = 0$ otherwise. By axiom 1. in the definition of a D-family, for every $A \subseteq \{0, \ldots, n-1\}$, we have that $[G_m \mid m \in A] \oplus Y_{k_n} \in \U_n$, so there exists a finite-dimensional subspace $K_n^A \subseteq [G_m \mid m \in A] \oplus Y_{k_n}$ and $\varepsilon_n^A > 0$ such that $\Big[K_n^A, \; [G_m \mid m \in A] \oplus Y_{k_n}\Big]_{\varepsilon_n^A} \subseteq \U_n$. We can assume that $K_n^A = [G_m \mid m \in A] \oplus H_n^A$ for some finite-dimensional subspace $H_n^A \subseteq Y_{k_n}$. Now let $H_n$ be the finite-dimensional subspace of $Y_{k_n}$ generated by all the $H_n^A$'s, $A \subseteq \{0, \ldots, n-1\}$, and let $\varepsilon_n = \min\{\varepsilon_n^A \mid A \subseteq \{0, \ldots, n-1\}\}$. We have that for every $A \subseteq \{0, \ldots, n-1\}$, $\Big[[G_m \mid m \in A] \oplus H_n, \; [G_m \mid m \in A] \oplus Y_{k_n}\Big]_{\varepsilon_n} \subseteq \U_n$. Now consider an isomorphism $T_n \colon Y_0 \to Y_0$ such that
\begin{itemize}
    \item $T_n$ is equal to the identity on $[F_k \mid k < k_n]$;
    \item $T_n(Y_{k_n}) = Y_{k_n}$;
    \item $T_n(H_n) \subseteq [F_k \mid k_n \leqslant k < k_{n + 1}]$ for some $k_{n + 1} > k_n$;
    \item $\|T_n - \Id_{Y_0}\| < \varepsilon_n$.
\end{itemize}
We let $G_n = [F_k \mid k_n \leqslant k < k_{n + 1}]$, and this achieves the construction.

\smallskip

It is clear that $(G_n)$ is a blocking of $(F_n)$. We show that it is $\H$-good. Let $A \subseteq \N$ be infinite and $n \in A$, we show that $[G_m \mid m \in A] \in \mathcal{U}_n$, which is enough to conclude. We know that $T_n$ fixes $[G_m \mid m \in A, \, m < n]$, and we have $T_n(H_n) \subseteq G_n$, thus $(T_n)^{-1}([G_m \mid m \in A])$ contains $[G_m \mid m \in A, \, m < n] \oplus H_n$. Moreover, $(T_n)^{-1}$ stabilizes $Y_{k_n}$, which contains the $G_m$'s for $m \geqslant n$, so $(T_n)^{-1}([G_m \mid m \in A])$ is contained in $[G_m \mid m \in A, \, m < n] \oplus Y_{k_n}$. Hence, we have:
$$(T_n)^{-1}([G_m \mid m \in A]) \in \Big[[G_m \mid m \in A, \, m < n] \oplus H_n, \; [G_m \mid m \in A, \, m < n] \oplus Y_{k_n}\Big],$$
and since $\|T_n - \Id_{Y_0}\| < \varepsilon_n$, we finally get:
$$[G_m \mid m \in A] \in \Big[[G_m \mid m \in A, \, m < n] \oplus H_n, \; [G_m \mid m \in A, \, m < n] \oplus Y_{k_n}\Big]_{\varepsilon_n} \subseteq \mathcal{U}_n,$$
as wanted.

\end{proof}

\begin{lem}\label{lem:GoodFDD}

For every $Y \in \H$ and every $\varepsilon > 0$, there exists a subspace of $Y$ having an $\H$-good FDD $(F_n)_{n \in \N}$ with constant at most $1 + \varepsilon$.

\end{lem}

\begin{proof}

By \prettyref{lem:GoodBlocking}, it is enough to build the FDD $(F_n)$ in such a way that $[F_n \mid n \in \N] \in \H$; passing to a blocking, we can turn it into an $\H$-good FDD having the same constant. Let $(\mathcal{U}_n)_{n \in \N}$ be a decreasing family of Ellentuck-open subsets of $\Sub(X)$ such that $\H = \bigcap_{n \in \N} \mathcal{U}_n$. We build the FDD $(F_n)_{n \in \N}$ by induction on $n$. Suppose that $F_0, \ldots, F_{n-1}$ have been built. Let $Y_n$ be a finite-codimensional subspace of $Y$, with $Y_n \cap [F_i \mid i < n] = \{0\}$, and such that the natural projection from $[F_i \mid i < n] \oplus Y_n$ onto  $[F_i \mid i < n]$ has norm at most $1 + \varepsilon$. If $n \geqslant 1$, we can even assume that $Y_n \subseteq Y_{n - 1}$. We have that $[F_i \mid i < n] \oplus Y_n \in \U_n$, so we can find a finite-dimensional subspace $F_n \subseteq Y_n$ such that $\Big[[F_i \mid i \leqslant n], \; [F_i \mid i < n] \oplus Y_n\Big] \subseteq \U_n$. This achieves the construction.

\smallskip

By construction, for every $n \in \N$, we have $[F_i \mid i \geqslant n] \subseteq Y_n$, so the natural projection from $[F_i \mid i < n] \oplus [F_i \mid i \geqslant n] $ onto $[F_i \mid i < n]$ has norm at most $1 + \varepsilon$. This shows that $(F_n)$ is an FDD with constant at most $1 + \varepsilon$. Moreover, for every $n \in \N$, we have $[F_i \mid i \in \N] \in \Big[[F_i \mid i \leqslant n], \; [F_i \mid i < n] \oplus Y_n\Big] \subseteq \U_n$, so $[F_i \mid i \in \N] \in \H$.

\end{proof}

The next lemma is an $\H$-good version of Bessaga--Pe\l czy\'nski's selection principle.

\begin{lem}\label{lem:BessagaPelczynski}

Let $Y$ be a subspace of $X$ having an FDD $(F_n)_{n \in \N}$, and let $U \in \H$ be such that $U \subseteq Y$. Then there exists a subspace $Z$ of $Y$ spanned by an $\H$-good block-FDD $(G_n)_{n \in \N}$ of $(F_n)$, such that $Z$ isomorphically embeds into $U$.

\end{lem}

\prettyref{lem:BessagaPelczynski} was stated in this form for greater clarity, but for several applications in this paper, we will need  a more general and more precise version of it,  stated and proven below. This can be seen as an amalgamation of \prettyref{lem:BessagaPelczynski} and \prettyref{lem:Diag}.

\begin{lem}\label{lem:GoodBlockFDD}

Let $(Y_k)_{k \in \N}$ be a family of subspaces of $X$ such that for every $k \in \N$, $Y_k$ has an FDD $(F_n^k)_{n \in \N}$. Assume that for every $k \in \N$, $(F_n^{k+1})_{n \in \N}$ is a block-FDD of $(F_n^k)_{n \in \N}$. Let $(U_k)_{k \in \N}$ be a decreasing family of elements of $\H$ and let $(\Delta_k)_{k \in \N}$ be a sequence of positive real numbers. Assume that for every $k \in \N$, we have $U_k \subseteq Y_k$.

\smallskip

Then there exist a subspace $Z \subseteq X$ generated by an $\H$-good FDD $(G_k)_{k \in \N}$, such that for every $k \in \N$, $(G_l)_{l \geqslant k}$ is a block-FDD of $(F_n^k)_{n \in \N}$; and there exists an isomorphic embedding $T \colon Z \to X$ such that for every $k \in \N$, we have $T(G_k) \subseteq U_k$ and $\left\|(T - \Id_Z)_{\restriction [G_l \mid l \geqslant k]}\right\| \leqslant \Delta_k$, and such that the FDD $(T(G_k))_{k \in \N}$ of $T(Z)$ is $\H$-good.

\smallskip

Moreover, if we are given, for every $k \in \N$, a subset $D_k \subseteq S_X$ such that, for every finite $A \subseteq \N$, $D_k \cap [F_n^k \mid n \in A]$ is dense in the unit sphere of $[F_n^k \mid n \in A]$, then we can ensure that for every $k \in \N$, the space $G_k$ has a basis made of elements of $D_k$.

\end{lem}

\begin{proof}

Without loss of generality, we can assume that $\Delta_0 \leqslant \frac{1}{2}$ and that for every $k \in \N$, $\Delta_{k+1} \leqslant \frac{\Delta_k}{2}$. Let $(\mathcal{U}_k)_{k \in \N}$ be a decreasing family of Ellentuck-open subsets of $\Sub(X)$ such that $\H = \bigcap_{k \in \N} \mathcal{U}_k$. Let $C$ be the constant of the FDD $(F_n^0)$. We build by induction on $k$ the FDD $(G_k)$, along with an FDD $(H_k)_{k \in \N}$ of a subspace of $X$, such that for all $k \in \N$, $\dim(G_k) = \dim(H_k)$. We also build, at the same time, a sequence of isomorphisms $T_k \colon G_k \rightarrow H_k$; the embedding $T$ will be defined as the unique bounded linear mapping on $Z$ extending all the $T_k$'s.

\smallskip

Suppose that the $G_l$'s, the $H_l$'s and the $T_l$'s are built for $l < k$. Let $r_k \in \N$ be defined as follows: if $k = 0$, then $r_k = 0$, and otherwise, $r_k$ is such that $\supp \left(G_{k-1}\right) < \supp \left(F_{r_k}^k\right)$, where the supports are taken with respect to the FDD $(F_n^{k-1})_{n \in \N}$. Let $Y_k' = [F_n^k \mid n \geqslant r_k]$, and let $U_k'$ be a finite-codimensional subspace of $U_k$ such that $U_k' \subseteq Y_k'$ and $U_k' \cap [H_l \mid l < k] = \{0\}$; if $k \geqslant 1$, we moreover suppose that $U_k' \subseteq U_{k - 1}'$. For every $A \subseteq \{0, \ldots, k-1\}$, the subspaces $[G_l \mid l \in A] \oplus U_k'$ and $[H_l \mid l \in A] \oplus U_k'$ are in $\H$, so as in the proof of \prettyref{lem:GoodBlocking}, we can find $\varepsilon_k > 0$ and a nonzero finite-dimensional subspace $H_k \subseteq U_k'$ such that for all $A$, both basic open sets: $$\Big[[G_l \mid l \in A] \oplus H_k, [G_l \mid l \in A] \oplus U_k'\Big]_{\varepsilon_k}$$ and $$\Big[[H_l \mid l \in A] \oplus H_k, [H_l \mid l \in A] \oplus U_k'\Big]_{\varepsilon_k}$$ are contained in $\U_k$. We can even assume that for all $l < k$, $\varepsilon_k \leqslant \frac{\varepsilon_l}{2^{k - l}}$. Since $H_k \subseteq Y_k' = [F_n^k \mid n \geqslant r_k]$, we can find a finite-dimensional subspace $G_k \subseteq Y_k'$ having finite support relative to the FDD $(F_n^k)_{n \geqslant r_k}$ of $Y_k'$, and a linear mapping $T_k \colon G_k \to H_k$ such that $\|T_k - \Id_{G_k}\| \leqslant \frac{\Delta_k}{4C}$ and $\|T_k^{-1} - \Id_{H_k}\| \leqslant \frac{\varepsilon_k}{24C(C+1)}$. We can even ensure that $G_k$ has a basis made of elements of $D_k$. This finishes the induction.


\smallskip

As wanted, for every $k \in \N$, $(G_l)_{l \geqslant k}$ is a block-FDD of $(F_n^k)_{n \in \N}$. In particular $(G_k)_{k \in \N}$ is a block-FDD of $(F_n^0)$ and hence has constant at most $C$. Let $Z = [G_k \mid k \in \N]$ and $\widetilde{Z} = \bigoplus_{k \in \N} G_k$, a dense vector subspace of $Z$. Define $\widetilde{T} \colon \widetilde{Z} \to X$ as the unique linear mapping extending all the $T_k$'s on their domains. For every eventually null sequence $(x_l)_{l \in \N}$ with $\forall l \in \N \; x_l \in G_l$, and for every $k \in \N$, we have:
\begin{eqnarray*}
\left\|(\widetilde{T}-\Id_{\widetilde{Z}})\left(\sum_{l=k}^\infty x_l\right)\right\| & \leqslant & \sum_{l = k}^\infty \|(T_l - \Id_{G_l})(x_l)\|\\
& \leqslant & \sum_{l=k}^\infty \frac{\Delta_l}{4C}\|x_l\|\\
& \leqslant & \sum_{l=k}^\infty \frac{\Delta_k}{2^{l - k + 2}C}\|x_l\|\\
& \leqslant & \sum_{l=k}^\infty \frac{\Delta_k}{2^{l - k + 2}C}\cdot 2C\left\|\sum_{l=k}^\infty x_l\right\|\\
&\leqslant & \Delta_k \left\|\sum_{l=k}^\infty x_l\right\|.
\end{eqnarray*}
This shows that $\left\|(\widetilde{T}-\Id_{\widetilde{Z}})_{\restriction \bigoplus_{l \geqslant k} G_l}\right\| \leqslant \Delta_k$. In particular, $\widetilde{T}$ is a bounded operator on $\widetilde{Z}$, so it extends to a bounded operator $T \colon Z \to X$ still satisfying $\left\|(T - \Id_Z)_{\restriction [G_l \mid l \geqslant k]}\right\| \leqslant \Delta_k$ for every $k \in \N$. In particular, since $\Delta_0 \leqslant \frac{1}{2}$, the latter inequality shows that $T$ is an isomorphic embedding, with $\|T\|\leqslant \frac{3}{2}$ and $\|T^{-1}\|\leqslant 2$. In particular, $(H_k)_{k \in \N}$ is an FDD of a subspace of $X$, with constant at most $3C$.

\smallskip

It remains to show that the FDD's $(G_k)$ and $(H_k)$ are $\H$-good. For $(H_k)$, the proof is similar as in \prettyref{lem:GoodBlocking}: given $A \subseteq \N$ infinite and $k \in A$, we have: $$[H_l \mid l \in A, \, l \geqslant k] \subseteq U_k',$$ so:
$$[H_l \mid l \in A] \in \Big[[H_l \mid l \in A, \, l < k] \oplus H_k, [H_l \mid l \in A, \, l < k] \oplus U_k'\Big],$$
and by construction, the set on the right hand side is contained in $\mathcal{U}_k$, which concludes. %

\smallskip

For $(G_k)$, we need one more estimate. Let, for all $k$, $K_k = [G_l \mid l < k]$,  $V_k = [H_l \mid l \geqslant k]$, and $W_k = K_k \oplus V_k$. Define $S_k \colon W_k \to Z$ as the unique operator which is equal to the identity on $K_k$, and to $T^{-1}$ on $V_k$. For all $l \geqslant k$, we have: $$\|T_l^{-1} - \Id_{H_l}\| \leqslant \frac{\varepsilon_l}{24C(C+1)} \leqslant \frac{\varepsilon_k}{2^{l - k}\cdot 24C(C+1)}.$$ Thus, knowing that $(H_l)_{l \geqslant k}$ is an FDD of $V_k$ with constant at most $3C$, and using the same proof as for $T$, we can show that $$\|\left(T^{-1}\right)_{\restriction V_k} - \Id_{V_k}\| \leqslant \sum_{l \geqslant k} \frac{\varepsilon_k}{2^{l - k} \cdot24C(C+1)} \cdot 6C = \frac{\varepsilon_k}{2(C+1)}.$$ Now recall that, by construction, $\supp(K_k) < \supp(Y_k')$, where the supports are taken with respect to the FDD $(F_n^0)$; and that $V_k \subseteq Y_k'$. Since the FDD $(F_n^0)$ has constant $C$, the natural projection $W_k \to V_k$ has norm at most $C + 1$. Since $S_k - \Id_{W_k}$ is the composition of this projection and $\left(T^{-1}\right)_{\restriction V_k} - \Id_{V_k}$, we deduce that $\|S_k - \Id_{W_k}\| \leqslant \frac{\varepsilon_k}{2}$.

\smallskip

We are now ready to prove that the FDD $(G_k)$ is $\H$-good. Let $A$ be an infinite subset of $\N$, and let $k \in A$, we want to prove that $[G_l \mid l \in A] \in \U_k$. We have:
$$[G_l \mid l\in A] = S_k\Big([G_l \mid l \in A, l < k] \oplus H_k \oplus [H_l \mid l \in A, l > k]\Big),$$
and:
$$[G_l \mid l \in A, l < k] \oplus H_k \oplus [H_l \mid l \in A, l > k] \in \Big[[G_l \mid l \in A, l < k] \oplus H_k, \; [G_l \mid l \in A, l < k] \oplus U_k'\Big],$$
so using the fact that $\|S_k - \Id_{W_k}\| < \varepsilon_k$, we get that:
$$[G_l \mid l\in A] \in \Big[[G_l \mid l \in A, l < k] \oplus H_k, \; [G_l \mid l \in A, l < k] \oplus U_k'\Big]_{\varepsilon_k}.$$
And we know, by construction, that this latter basic open set is contained in $\U_k$, as wanted.



\end{proof}

In the rest of this section, we introduce sufficient conditions for being a D-family, which will be convenient for applications.

\bigskip

\subsection{Wijsman and slice topologies}

Let $X$ a Banach space with a fixed norm $\|\cdot\|$. For $N$ an equivalent norm on $X$, for $Y \in \Sub(X)$ and for $x \in X$, we denote by $N_{X/Y}(x)$ the norm of the class of $x$ in the quotient $X/Y$, when this quotient is equipped with the corresponding quotient norm. Thus, we have an injective mapping:
$$\begin{array}{rccl}
\varphi_N \colon & \Sub(X) & \to & \R^X \\
     & Y & \mapsto & N_{X/Y}
\end{array}.$$
The \emph{Wijsman topology} associated to $N$ on $\Sub(X)$ is the topology obtained by pulling back through $\varphi_N$ the product topology on $\R^X$. For this topology, a net $(Y_\lambda)$ of elements of $\Sub(X)$ is converging to $Y \in \Sub(X)$ if and only if for every $x \in X$, $N_{X/Y_\lambda}(x) \to N_{X/Y}(x)$. 
In general, this topology depends on the choice of the equivalent norm $N$ (see \cite{Beer}, Section 2.4).

\smallskip

The \emph{slice topology} on $\Sub(X)$ is the topology generated by sets of the form $\{Y \in \Sub(X) \mid Y \cap U \neq \varnothing\}$ and $\{Y \in \Sub(X) \mid \delta_{\|\cdot\|}(Y, C) > 0\}$, where $U$ ranges over open subsets of $X$, $C$ ranges over nonempty bounded closed convex subsets of $X$, and $\delta_{\|\cdot\|}(Y, C) = \inf_{x \in C, y \in Y} \|x - y\|$. The name \textit{slice topology} comes from the fact that in the previous definition, it is actually enough to make $C$ range over slices of closed balls, that are, nonempty sets of the form $\{x \in X \mid \|x\| \leqslant r, \, x^*(x) \geqslant a\}$, where $r > 0$, $x^* \in X^*$, and $a \in \R$ (see \cite{Beer}, Lemma 2.4.4). It is easy to see that the slice topology on $\Sub(X)$ only depends on the isomorphic structure of $X$, but not of the norm.

\smallskip

The main properties of the Wijsman and the slice topologies can be found in \cite{Beer}. We reproduce some useful ones below.

\begin{thm}[see \cite{Beer}]

\alaligne

\begin{enumerate}
    \item If $X$ is separable, then all the Wijsman topologies on $\Sub(X)$ associated to equivalent norms are Polish.
    \item If $X$ is separable and has separable dual, then the slice topology on $\Sub(X)$ is Polish.
    \item The slice topology on $\Sub(X)$ is the coarsest topology refining all the Wijsman topologies associated to equivalent norms on $X$.
    \item \emph{(Hess' theorem)} If $X$ is separable, then the Borel $\sigma$-algebra associated with any Wijsman topology on $\Sub(X)$ coincides with the Effros Borel structure on this set.
    \item If $X$ is separable and has separable dual, then the Borel $\sigma$-algebra associated with the slice topology on $\Sub(X)$ coincides with the Effros Borel structure on this set.
\end{enumerate}

\end{thm}

These topologies are easier to manipulate than the Ellentuck topology. However, we have the following result:

\begin{prop}

The Ellentuck topology on $\Sub(X)$ is finer than the slice topology. In particular, it is finer than all the Wijsman topologies associated to equivalent norms.

\end{prop}

\begin{proof}

Fix $U$ a nonempty open subset of $X$. We show that $\mathcal{U} = \{Y \in \Sub(X) \mid Y \cap U \neq \varnothing\}$ is Ellentuck-open. For this, consider $Y \in \mathcal{U}$. We fix $x_0 \in U \cap Y$ and $\varepsilon > 0$ such that $\Ball(x_0, \varepsilon\|x_0\|) \subseteq U$; we show that $[\R x_0, Y]_\varepsilon \subseteq \mathcal{U}$. Let $Z \in [\R x_0, Y]_\varepsilon$, and fix $Z' \in [\R x_0, Y]$ and $T \colon Z' \to Z$ an isomorphism with $\|T - \Id_{Z'}\| < \varepsilon$. If $x_0 \neq 0$, then $\|T(x_0) - x_0\| < \varepsilon \|x_0\|$, so by the choice of $\varepsilon$, we have $T(x_0) \in U \cap Z$; this conclusion remains true if $x_0 = 0$. This shows that $Z \in \mathcal{U}$.

\smallskip

Now fix $C$ a nonempty bounded closed convex subset of $X$. We show that $\mathcal{V} = \{Y \in \Sub(X) \mid \delta_{\|\cdot\|}(Y, C) > 0\}$ is Ellentuck-open. For this, we fix $Y \in \mathcal{V}$ and we show that $[\{0\}, Y]_\varepsilon \subseteq \mathcal{V}$, for small enough $\varepsilon$. More precisely, let $\eta = \delta_{\|\cdot\|}(Y, C)$ and let $R \geqslant 1$ such that $C \subseteq \Ball(0, R - \eta)$. Then we take $\varepsilon = \frac{\eta}{4R}$ (so in particular, $\varepsilon \leqslant \frac{1}{2}$). Let $Z \in [\{0\}, Y]_\varepsilon$, and fix $Z' \in [\{0\}, Y]$ and $T \colon Z' \to Z$ an isomorphism with $\|T - \Id_{Z'}\| < \varepsilon$. We pick $x \in Z'$ and we show that $\delta_{\|\cdot\|}(T(x), C) \geqslant \frac{\eta}{2}$, which is enough to conclude. If $\|x\| > 2R$, then $\|T(x)\| \geqslant \|x\| - \|T(x) - x\| \geqslant \frac{\|x\|}{2} > R$, so $\delta_{\|\cdot\|}(T(x), C) \geqslant \eta$. If $\|x\| \leqslant 2R$, then $\|T(x) - x\| \leqslant \frac{\eta}{2}$. And since $x \in Y$, we have $\delta_{\|\cdot\|}(x, C) \geqslant \eta$, so $\delta_{\|\cdot\|}(T(x), C) \geqslant \frac{\eta}{2}$.

\end{proof}

\begin{cor}

Consider $\mathcal{H} \subseteq \Subinf(X)$ satisfying the two following properties:
\begin{enumerate}
    \item For every $Y \in \Subinf(X)$ and every $F \in \Subfin(X)$, we have $Y \in \mathcal{H}$ if and only if $Y + F \in \mathcal{H}$;
    \item $\mathcal{H}$, seen as a subset of $\Sub(X)$, is $G_\delta$ for one of the Wijsman topologies, or for the slice topology.
\end{enumerate}
Then $\H$ is a D-family of subspaces of $X$.

\end{cor}

\bigskip

\subsection{Degrees}

Degrees will be our main way of defining D-families throughout this paper. 
A degree allows one to define a notion of largeness on the class of all Banach spaces, and this notion gives rise to a D-family when restricted to the set of subspaces of some fixed Banach space.

\smallskip

We define an \textit{approximation pair} as a pair $(X, F)$ where $X$ is a (finite- or infinite-dimensional) Banach space, and $F$ is a finite-dimensional subspace of $X$. We denote by $\AP$ the class of approximation pairs. If $(X, F), (Y, G) \in \AP$, a \emph{morphism} from $(X, F)$ to $(Y, G)$ is a pair $\varphi = (S, T)$, where $S \colon G \longrightarrow F$ and $T \colon X \longrightarrow Y$ are operators that make the following diagram commute:
$$\begin{tikzcd}
F \arrow[hook]{r}{\iota}
& X \arrow{d}{T} \\
G \arrow[hook]{r}{\iota} \arrow{u}{S}
& Y
\end{tikzcd}$$
where the $\iota$'s stand for the inclusions. The \textit{norm} of the morphism $\varphi$ is defined as $\| \varphi \| = \| S \| \cdot \| T \|$ if $G \neq \{0\}$, and $\|\varphi\| = 1$ if $G=\{0\}$.

\begin{defin}\label{defin:degree}

A \textit{degree} is a mapping $d \colon \AP \longrightarrow \R_+$ for which there exists $K_d \colon [1, \infty) \times \R_+ \longrightarrow \R_+$ such that:
\begin{itemize}
    \item $K_d$ is non-decreasing in both of its variables;
    \item for all $t \in \R_+$, $\lim_{s \rightarrow 1} K_d(s, t) = t$;
\end{itemize}
and for every $(X, F), (Y, G) \in \AP$ and for every morphism $\varphi \colon (X, F) \longrightarrow (Y, G)$, we have $d(Y,G) \leqslant K_d(\| \varphi \|, d(X,F))$.

\smallskip
\end{defin}

\begin{defin}\label{defin:d-small}
Given a degree $d$, we say that a Banach space $X$ is \textit{$d$-small} if \linebreak $\sup_{F \in \Subfin(X)} d(X, F) < \infty$, and that $X$ is \textit{$d$-large} otherwise.
\end{defin}

For most degrees we will consider, the value of $d(X, F)$ will actually only depend on $F$. Degrees satisfying this property will be called \emph{internal degrees}, and $d(X, F)$ will simply be denoted by $d(F)$. To verify that $d \colon \Banfin \to \R_+$ is an internal degree, it is enough to find $K_d \colon [1, \infty) \times \R_+ \longrightarrow \R_+$ as above, such that for every embedding $S \colon G \to F$ between two finite-dimensional spaces, we have $d(G) \leqslant K_d(\|S\| \cdot \|S^{-1}\|, d(F))$ (where $S^{-1}$ is defined on $S(G)$, and with the convention that $\|S\| \cdot \|S^{-1}\| = 1$ when $G = \{0\}$).

\begin{exs}\label{exs:Degree}

\alaligne

\begin{enumerate}

\item Let $d(F) = \dim(F)$. Then $d$ is an internal degree, witnessed by $K_d(s, t) = t$. A space is $d$-small if and only if it is finite-dimensional.

\item Let $d(F) = d_{BM}(F, \ell_2^{\dim(F)})$. Then $d$ is an internal degree, witnessed by $K_d(s, t) = st$. A space is $d$-small if and only if it is Hilbertian (a consequence of Kwapien's theorem \cite{kwapien}).

\item Fix $1 \leqslant p \leqslant 2 \leqslant q \leqslant \infty$. Let $d(F)$ be the type-$p$ constant (resp. the cotype-$q$ constant) of $F$. Then $d$ is an internal degree, witnessed by $K_d(s, t) = st$. A space is $d$-small if and only if it has type $p$ (resp. cotype $q$). If $X$ is $d$-small, then $\sup_{F \in \Subfin(X)} d(F)$ is the type-$p$ constant (resp. the cotype-$q$ constant) of $X$.

\item Fix $1 \leqslant p \leqslant \infty$. For $(X, F) \in \AP$, define $d(X, F)$ as the infimum of the $M$'s for which the canonical inclusion of $F$ into $X$ $M$-factorizes through some $\ell_p^n$, meaning there exist $n \in \N$ and operators $U \colon F \to \ell_p^n$ and $V \colon \ell_p^n \to X$ with $\| U \| \cdot \| V \| = M$, making the following diagram commute:
$$\begin{tikzcd}
& \ell_p^n \arrow{dr}{V} & \\
F \arrow[hook]{rr}{\iota} \arrow{ur}{U} & & X
\end{tikzcd}$$
Then $d$ is degree, witnessed by $K_d(s, t) = st$. By \cite{LindenstraussRosenthal}, Theorem 4.3 (and the classical fact that $\ell_2^n$'s are uniformly complemented in $L_p, 1<p<\infty$), we have that:
\begin{itemize}
    \item if $1 < p < \infty$, a space is $d$-small if and only if it is either an $\mathcal{L}_p$-space, or a Hilbertian space;
    \item if $p = 1$ or $p = \infty$, a space is $d$-small if and only if it is an $\mathcal{L}_p$-space.
\end{itemize}

\item For $(X, F) \in \AP$, define $d(X, F)$ as the infimum of the $M$'s for which there exist a space $Z$ with a $1$-unconditional basis and operators $U \colon F \to Z$ and $V \colon Z \to X$ with $\| U \| \cdot \| V \| = M$, making the following diagram commute:
$$\begin{tikzcd}
& Z \arrow{dr}{V} & \\
F \arrow[hook]{rr}{\iota} \arrow{ur}{U} & & X
\end{tikzcd}$$
Then $d$ is a degree, witnessed by $K_d(s, t) = st$. A space is $d$-small if and only if it has Gordon-Lewis local unconditional structure (GL-lust) \cite{GordonLewis}.
If $X$ is $d$-small, then $\sup_{F \in \Subfin(X)} d(X, F)$ is the
GL-lust constant of $X$.

\item For $(X,F) \in \AP$, define $d(X,F)$ as the infimum of the $K$'s such that $F$ is $K$-complemented in $X$. Then $d$ is not a degree. To see this pick $F=G \subset X \subset Y$ in such a way that $F$ is $1$-complemented in $X$ but not $n$-complemented in $Y$, so that $d(X,F) =1$ and $d(Y,G) \geqslant n$; and take $S=\Id_F$, $T$ the canonical inclusion of $X$ into $Y$. 

This may be surprising, in view of the fact that a space is Hilbertian if and only if there exists $K \geq 1$ such that all its finite-dimensional subspaces are $K$-complemented in it (see \cite{kalton}, Theorem 12.42), i.e., Hilbertian spaces would be characterized as the $d$-small spaces if $d$ were a degree.
\end{enumerate}

\end{exs}

\begin{rem}
The notion of degree is closely related to the classical Pietsch's theory of operator ideals  \cite{Pietsch}.
Recall (\cite{Pietsch} Chapter 6) that a \textit{quasi-norm} on an operator ideal $\mathfrak{A}$ is a mapping $A \colon \mathfrak{A} \to \R_+$ such that: 
\begin{enumerate}
 \item $A(\Id_X) = 1$ whenever $X$ is $1$-dimensional,
 \item there exists a constant $K \geqslant 1$ such that $A(S+T) \leqslant K(A(S) + A(T))$ whenever $S,T$ belong to $\mathfrak{A}$ and $S+T$ makes sense,
 \item
  $A(U \circ T \circ S) \leqslant \|U\| \cdot A(T) \cdot \|S\|$ whenever $T$ belongs to $\mathfrak{A}$ and $U \circ T \circ S$ makes sense.
  \end{enumerate}
Hence, to every quasi-normed operator ideal $(\mathfrak{A}, A)$, we can associate a degree $d_A$ defined by $d_A(X, F) = A(\iota_{F, X})$ for every approximation pair $(X, F)$, where $\iota_{F, X}$ denotes the inclusion map of $F$ into $X$. It is interesting to observe that all examples of degrees given in \prettyref{exs:Degree} come in this fashion:
\begin{enumerate}
\item The dimension is the degree associated to the ideal of \textit{nuclear operators} (see \cite{Pietsch}, 6.3.1); this is a consequence of Auerbach's lemma.
\item The internal degree $d(F) = d_{BM}(F, \ell_2^{\dim(F)})$ is the degree associated to the ideal of \textit{Hilbert operators} (see \cite{Pietsch}, 6.6.1).
\item The type $p$ and cotype $q$ constants are the degrees associated to the ideals of \textit{type $p$} and \textit{cotype $q$ operators}, normed with the type $p$ and cotype $q$ constants, respectively (see \cite{PisierFactorization}, Section 3.a). 
\item The degree defined in 4. of \prettyref{exs:Degree} is the degree associated to the ideal of \textit{discretely $p$-factorable} operators (see \cite{Pietsch}, 19.3.11).
\item The degree defined in 5. of \prettyref{exs:Degree} is the degree associated to the ideal of \textit{$\sigma$-nuclear operators} (see \cite{Pietsch}, 23.2.1); this is a consequence of Theorem 23.2.5 in \cite{Pietsch}.
\end{enumerate}
We warn the reader that given a quasi-normed operator ideal $(\mathfrak{A}, A)$, the class of $d_A$-small spaces does in general not coincide with the space ideal associated to $\mathfrak{A}$ (see \cite{Pietsch}, 2.1.2). This is for instance not the case for examples 4 and 5 above.

\smallskip

We will not further develop the link between degrees and quasi-normed operator ideals in this paper, as this would be of limited practical use. Indeed, exhibiting a degree adapted to a given situation is in general much easier than exhibiting a quasi-normed operator ideal, as shown by the examples above.
\end{rem}

We can also define a notion of asymptotic smallness:

\begin{defin}
Let $d$ be a degree. A Banach space $X$ is said to be \textit{asymptotically $d$-small} if there exists a constant $K$ such that for every $n \in \N$, there exists a finite-codimensional subspace $X_n \subseteq X$ such that every $n$-dimensional subspace $F \subseteq X_n$ satisfies $d(X, F) \leqslant K$.
\end{defin}

When $d(X, F) = d_{BM}(F, \ell_2^{\dim(F)})$, asymptotically $d$-small Banach spaces are exactly asymptotically Hilbertian Banach spaces (see \prettyref{defin:AsympHilbert}).

\smallskip

If $d$ is an internal degree, then a subspace of a $d$-small space is itself $d$-small, and a subspace of an asymptotically $d$-small space is itself asymptotically $d$-small. This is not true in general; for example, $L_p([0, 1])$ is an $\mathcal{L}_p$-space, and for $1 \leqslant p \neq 2 \leqslant \infty$, the only non-Hilbertian subspaces of $L_p([0, 1])$ which are $\mathcal{L}_p$ are the complemented ones \cite{LindenstraussRosenthal}. Similarly, the property of having Gordon-Lewis local unconditional structure is not stable under passing to subspaces; consider $\ell_p$ spaces for $1 \leqslant p \neq 2 < \infty$, a consequence of, e.g., \cite{KomorowskiTomczak}.

\begin{rem}\label{rem:BasicRemarkDegrees}
A useful property of degrees is the fact that for $F \subseteq G \subseteq Y \subseteq X$, where the spaces $F$ and $G$ are finite-dimensional and the spaces $X$ and $Y$ are arbitrary, we have $d(X, F) \leqslant d(Y, G)$. To see this, just consider the morphism $(\Id_F, \Id_Y)$ from $(Y, G)$ to $(X, F)$.
\end{rem}

In the rest of this subsection, we fix a degree $d$.

\begin{lem}\label{lem:DegFinDim}

For every $n \in \N$, there exists a constant $C_d(n)$ such that for every $(X, F) \in \AP$ with $\dim(F) = n$, we have $d(X, F) \leqslant C_d(n)$. In particular, every finite-dimensional space is $d$-small.

\end{lem}

\begin{proof}

Let $(X, F) \in \AP$ with $\dim(F) = n$. If $n \geq 1$, then let $T : \ell_1^n \longrightarrow F$ be an 
$n$-isomorphism. Then $(T^{-1}, T)$ is a morphism from the pair $(\ell_1^n, \ell_1^n)$ to the pair $(X, F)$, with norm at most 
$n$. So, letting $C_d(n) = 
K_d(n, d(\ell_1^n, \ell_1^n))$, it follows that $d(X, F) \leqslant C_d(n)$. The proof of the case $n=0$ is similar, replacing $\ell_1^n$ by $\{0\}$.

\end{proof}

\begin{lem}

The properties of being $d$-small, $d$-large, and asymptotically $d$-small are invariant under isomorphism.

\end{lem}

\begin{proof}

Let $X$ and $Y$ be two Banach spaces, and $T \colon X \to Y$ be an isomorphism. First suppose that $X$ is $d$-small, and let $K = \sup_{F \in \Subfin(X)} d(X, F)$. Let $G \in \Subfin(Y)$. Then $((T^{-1})_{\restriction G}, T)$ is a morphism from the pair $(X, T^{-1}(G))$ to the pair $(Y, G)$, so $d(Y, G) \leqslant K_d(\|(T^{-1})_{\restriction G}\| \cdot \|T\|, d(X, T^{-1}(G))) \leqslant K_d(\|T^{-1}\|\cdot\|T\|, K)$. This bound does not depend on $G$, so $Y$ is $d$-small.

\smallskip

Now suppose that $X$ is asymptotically $d$-small and fix a constant $K$ witnessing it. We show that $Y$ is asymptotically $d$-small, witnessed by the constant $L = K_d(\|T^{-1}\|\cdot\|T\|, K)$. Let $n \in \N$. There exists a finite-codimensional subspace $X_n \subseteq X$ such that for every $n$-dimensional subspace $F \subseteq X_n$, we have $d(X, F) \leqslant K$. Let $Y_n = T(X_n)$, and consider $G \subseteq Y_n$ an $n$-dimensional subspace. Then $((T^{-1})_{\restriction G}, T)$ is a morphism from the pair $(X, T^{-1}(G))$ to the pair $(Y, G)$, so $d(Y, G) \leqslant K_d(\|T^{-1}\| \cdot \|T\|, d(X, T^{-1}(G)))$. Since $T^{-1}(G) \subseteq X_n$, we have $d(X, T^{-1}(G)) \leqslant K$, so $d(Y, G) \leqslant L$, as wanted.

\end{proof}

\begin{lem}\label{lem:DegComplemented}

\alaligne
\begin{enumerate}

\item A complemented subspace of a $d$-small space is $d$-small.

\item A complemented subspace of an asymptotically $d$-small space is asymptotically $d$-small.

\end{enumerate}

\end{lem}

\begin{proof}

Fix $X$ a Banach space, $Y$ a complemented subspace of $X$, and $P \colon X \to Y$ a projection.

\begin{enumerate}

\item  Suppose $X$ is $d$-small, and let $K = \sup_{F \in \Subfin(X)} d(X, F)$. Let $F \subseteq Y$ be a finite-dimensional subspace. Then $(\Id_F, P)$ is a morphism from the pair $(X, F)$ to the pair $(Y, F)$, so we have $d(Y, F) \leqslant K_d(\|P\|, d(X, F)) \leqslant K_d(\|P\|, K)$. Hence $Y$ is $d$-small.

\item Suppose $X$ is asymptotically $d$-small, witnessed by a constant $K$. We show that $Y$ is asymptotically $d$-small, witnessed by the constant $K_d(\|P\|, K)$. Let $n \in \N$, and fix a finite-codimensional subspace $X_n \subseteq X$ such that for every $n$-dimensional subspace $F \subseteq X_n$, we have $d(X, F) \leqslant K$. Let $Y_n = X_n \cap Y$. For an $n$-dimensional subspace $F \subseteq Y_n$, $(\Id_F, P)$ is a morphism from the pair $(X, F)$ to the pair $(Y, F)$, so we have $d(Y, F) \leqslant K_d(\|P\|, d(X, F)) \leqslant K_d(\|P\|, K)$, as wanted.

\end{enumerate}

\end{proof}

\begin{lem}\label{lem:DegFinCodim}

Let $X$ be a Banach space, and $Y$ be a finite-codimensional subspace of $X$. Then:

\begin{enumerate}

\item $X$ is $d$-small iff $Y$ is $d$-small;

\item $X$ is asymptotically $d$-small iff $Y$ is asymptotically $d$-small;

\end{enumerate}

\end{lem}

\begin{proof}

Since $Y$ is complemented in $X$, we know by \prettyref{lem:DegComplemented} that if $X$ is $d$-small (resp. asymptotically $d$-small), then so is $Y$. So in both cases, we just have one direction to show.

\begin{enumerate}

\item Suppose that $Y$ is $d$-small, and let $K = \sup_{G \in \Subfin(Y)} d(Y, G)$. By \prettyref{lem:DegFinDim}, we can suppose that $X$ is infinite-dimensional. We denote by $k$ the codimension of $Y$ in $X$. Recall that by Lemma 3 in \cite{FerencziRosendalOnTheNumber}, every $k$-codimensional subspace of $X$ is $A(k)$-isomorphic to $Y$, where the constant $A(k)$ only depends on $k$.

\smallskip

Let $F \subseteq X$ be finite-dimensional; we want to bound $d(X, F)$. Find $Z \subseteq X$ a subspace with codimension $k$ containing $F$. Let $T \colon Z \to Y$ be an $A(k)$-isomorphism. We have $d(Y, T(F)) \leqslant K$. Moreover, $(T_{\restriction F}, T^{-1})$ is a morphism from the pair $(Y, T(F))$ to the pair $(X, F)$, so $d(X, F) \leqslant K_d(\|T\|\cdot\|T^{-1}\|, K) \leqslant K_d(A(k), K)$, as wanted.

\item Suppose that $Y$ is asymptotically $d$-small. Then there exist a constant $K$ and finite-codimensional subspaces $Y_n \subseteq Y$ for all $n$, such that for every $n$-dimensional subspace $F \subseteq Y_n$, we have $d(Y, F) \leqslant K$. In particular, for such an $F$, we also have $d(X, F) \leqslant K$, showing that $X$ is asymptotically $d$-small.

\end{enumerate}

\end{proof}

\begin{prop}\label{prop:DegDFamily}

Let $X$ be a Banach space, and $\H$ be the set of subspaces of $X$ that are $d$-large. Then $\H$ is a D-family.

\end{prop}

\begin{proof}

\prettyref{lem:DegFinDim} shows that $\H$ contains only infinite-dimensional subspaces. The stability of $\H$ under finite-dimensional modifications comes from \prettyref{lem:DegFinCodim}. Now we need to prove that $\H$ is Ellentuck-$G_\delta$. For every $n \in \N$, let $\U_n = \{Y \in \Sub(X) \mid \exists F \in \Subfin(Y) \; d(Y, F) > n\}$, so that $\H = \bigcap_{n \in \N} \U_n$. We show that all the $\U_n$'s are open.

\smallskip

Fix $n \in \N$ and $Y \in \U_n$. Let $F \in \Subfin(Y)$ be such that $d(Y, F) > n$. We know that $\lim_{s \rightarrow 1} K_d(s, n) = n$, so there exists $\varepsilon \in (0, 1)$ such that $K_d(\frac{1 + \varepsilon}{1 - \varepsilon}, n) < d(Y, F)$. We show that $[F, Y]_\varepsilon \subseteq \U_n$. Let $Z \in [F, Y]_\varepsilon$, and let $Z' \in [F, Y]$ and $T \colon Z' \to Z$ be an isomorphism such that $\|T - \Id_{Z'}\| < \varepsilon$. Then $\|T\| \leqslant 1 + \varepsilon$ and $\|T^{-1}\|\leqslant \frac{1}{1 - \varepsilon}$. So $\left(T_{\restriction F}, T^{-1}\right)$ is a morphism of norm at most $\frac{1 + \varepsilon}{1 - \varepsilon}$ from the approximation pair $(Z, T(F))$ to the pair $(Y,F)$. Thus, $d(Y, F) \leqslant K_d(\frac{1+\varepsilon}{1 - \varepsilon}, d(Z, T(F)))$. If we had $d(Z, T(F)) \leqslant n$, we would have $d(Y, F) \leqslant K_d(\frac{1+\varepsilon}{1 - \varepsilon}, n)$, contradicting the choice of $\varepsilon$. So $d(Z, T(F)) > n$, witnessing that $Z \in \mathcal{U}_n$.
\end{proof}

\begin{cor} Given a sequence $(d_n)_{n \in \N}$ of degrees and a space $X$, the family of subspaces of $X$ that are large for all the $d_n$'s is a D-family, and in the same way, for fixed $N \in \N$, the family of subspaces of $X$ that are large for at least one $d_n$, $n \leqslant N$, is also a D-family. \end{cor}
  
\begin{proof} Since the class of $G_\delta$ subsets of a topological space is closed under countable intersections and under finite unions, this is a consequence of \prettyref{prop:DegDFamily}. \end{proof}

For instance, for $2 < q_0 \leqslant \infty$ fixed, the family of subspaces of $X$ that do not have any cotype $q < q_0$ is a D-family.

\smallskip

If $d$ is a degree and $X$ a Banach space, the D-family defined in \prettyref{prop:DegDFamily} will be denoted by $\H_d^X$, or by $\H_d$ when there is no ambiguity. An $\H_d$-good FDD will simply be called \emph{$d$-good}. In the case of families defined by a degree, we have a useful strengthening of the notion of good FDD's:

\begin{defin}

An FDD $(F_n)_{n \in \N}$ of a Banach space $X$ is \textit{$d$-better} if $d(X, F_n) \xrightarrow[n \to \infty]{} \infty$.
\end{defin}

This implies that $(F_n)$ is a $d$-good FDD. Indeed, if $A \subseteq \N$ is infinite, then for every $n \in A$, we have $d([F_m \mid m \in A], F_n) \geqslant d(X, F_n)$. Below we prove a weak converse to this; this can be seen as the $d$-better version of \prettyref{lem:GoodBlocking}.

\begin{lem}\label{lem:VeryGoodBlocking}

Let $(F_n)_{n \in \N}$ be an FDD of a $d$-large Banach space $X$. Then there exists a blocking $(G_n)_{n \in \N}$ of $(F_n)$ which is a $d$-better FDD.

\end{lem}

\begin{proof}

Let $C$ be the constant of the FDD $(F_n)$. For each $k \in \N$, let $X_k = [F_l \mid l \geqslant k]$ and let $P_k \colon X \to X_k$ the projection associated to the FDD. We build $(G_n)$ by induction as follows. Suppose that the $G_m$'s have been built for $m < n$, and let $k_n = (\max \, \supp (G_{n - 1})) + 1$ if $n \geqslant 1$, $k_n = 0$ otherwise. The space $X_{k_n}$ is $d$-large, so there exists $H_n \in \Subfin(X_{k_n})$ such that $d(X_{k_n}, H_n) \geqslant n$. Now consider an isomorphism $T_n \colon X_{k_n} \to X_{k_n}$ such that $\|T_n\|\cdot\|T_n^{-1}\| \leqslant 2$ and such that $T_n(H_n) \subseteq [F_k \mid k_n \leqslant k < k_{n + 1}]$ for some $k_{n + 1}$. Let $G_n = [F_k \mid k_n \leqslant k < k_{n + 1}]$. This finishes the construction of $(G_n)$.

\smallskip

To prove that $(G_n)$ is $d$-better, fix $n$ and consider the morphism $((T_n)_{\restriction H_n}, T_n^{-1}\circ P_{k_n})$ from the pair $(X, G_n)$ to the pair $(X_{k_n}, H_n)$. It has norm at most $2(1 + C)$, so $n \leqslant d(X_{k_n}, H_n) \leqslant K_d(2(1 + C), d(X, G_n))$. In particular, for every constant $K$, as soon as $n > K_d(2(1+C), K)$, we have $d(X, G_n) > K$. This shows that $d(X, G_n) \xrightarrow[n \to \infty]{} \infty$.

\end{proof}

As an illustration, note that if $d(F)$ is the dimension of $F$, then any FDD is $d$-good, while a  $d$-better FDD is an FDD where the dimensions of the summands tend to infinity.

\bigskip\bigskip

\section{The first  dichotomy}\label{sec:FirstDichotomy}

In this section, we generalize Gowers' Ramsey-type \prettyref{thm:GowersThm} to D-families. As an application, we prove our first dichotomy (\prettyref{thm:FirstDichoto}), a local version of Gowers' first dichotomy (\prettyref{thm:Gowers1stDichoto}).

\bigskip

\subsection{A local version of Gowers' Ramsey-type theorem}

In this subsection, we fix a Banach space $X$, and a D-family $\H$ of subspaces of $X$. We work in the approximate Gowers space $\G_\H = (\H, S_X, \delta_{\|\cdot\|}, \subseteq, \subseteq^*, \in)$ defined in last section (see \prettyref{cor:DefLocalGowersSpace}). Each time we will mention Gowers' game or the asymptotic game, we will be referring to the games relative to this space. Note that Gowers' game relative to this space is in general different from the original game defined by Gowers. For $Y \in \H$, the game $G_Y$ has the following form:

\smallskip

\begin{tabular}{ccccccc}
\textbf{I} & $Y_0$ & & $Y_1$ & & $\hdots$ & \\
\textbf{II} & & $y_0$ & & $y_1$ & & $\hdots$ 
\end{tabular}

\smallskip

\noindent where the $y_n$'s are elements of $S_X$, and the $Y_n$'s are elements of $\H$, with the constraint that for all $n \in \N$, $Y_n \subseteq Y$ and $y_n \in Y_n$. The outcome is, as usual, the sequence $(y_n)_{n \in \N} \in \left(S_X\right)^\N$.

\smallskip

Our local version of Gowers' \prettyref{thm:GowersThm} is the following:

\begin{thm}\label{thm:LocalGowersThm}

Let $\X \subseteq \left(S_X\right)^\N$ be analytic, let $Y\in \H$, let $\Delta$ be a sequence of positive real numbers and let $\varepsilon > 0$. Then there exists $Z \in \H_{\restriction Y}$ such that:

\begin{itemize}
    \item either $Z$ has an $\H$-good FDD $(G_n)_{n \in \N}$, with constant at most $1+\varepsilon$, and all of whose normalized block-sequences belong to $\X^c$;
    \item or player \II{} has a strategy in $G_Z$ to reach $(\X)_\Delta$.
\end{itemize}

\noindent Moreover:

\begin{itemize}
    \item if $\H = \H_d$ for some degree $d$, then we can even assume that the FDD $(G_n)$ is $d$-better;
    \item if $Y$ comes with a fixed FDD $(F_n)_{n \in \N}$, then we can also assume that $(G_n)$ is a block-FDD of $(F_n)$.
\end{itemize}

\end{thm}

Gowers' \prettyref{thm:GowersThm} is just the special case of the last theorem when $\H = \Subinf(X)$ (which is a D-family, the family of $d$-large subspaces of $X$ for the internal degree $d(F) = \dim(F)$).

\begin{proof}[Proof of \prettyref{thm:LocalGowersThm}]

We start with the general case; the ``moreover'' part will be dealt with separately at the end of the proof. We proceed in the same way as in the proof of \prettyref{thm:GowersThm}: we apply the abstract Gowers' \prettyref{thm:AbstractGowers} to the approximate Gowers space $\G_\H$, endowed with the system of relatively compact sets $(\mathcal{K}_X, \oplus_X)$ defined in \prettyref{subsec:ApproxGowersSpaces} (recall that the elements of $\mathcal{K}_X$ are the unit spheres of finite-dimensional subspaces of $X$, and that $S_F \oplus_{X} S_G = S_{F + G}$). In case we get $Z \in \H_{\restriction Y}$ such that player \II{} has a strategy in $G_Z$ to reach $(\X)_\Delta$, we are done. So we now suppose that there exists $U \in \H_{\restriction Y}$ such that player \I{} has a strategy $\tau$ in $SF_U$ to build a sequence $(K_n)_{n \in \N}$ with $\bs((K_n)_{n \in \N}) \subseteq \X^c$. We can assume that the strategy $\tau$ is such that for every run of the game $SF_U$:

\smallskip

\begin{tabular}{ccccccc}
\textbf{I} & $U_0$ & & $U_1$ & & $\hdots$ & \\
\textbf{II} & & $S_{G_0}$ & & $S_{G_1}$ & & $\hdots$ 
\end{tabular}

\smallskip

\noindent played according to $\tau$, we have $[G_i \mid i < n] \cap U_n = \{0\}$ and the natural projection $[G_i \mid i < n] \oplus U_n \to [G_i \mid i < n]$ has norm at most $1 + \varepsilon$. We build an FDD $(G_n)_{n\in\N}$ of a subspace of $U$, with constant at most $1 + \varepsilon$, such that $[G_n \mid n \in\N] \in\H$, and all of whose normalized block-sequences belong to $\X^c$; by \prettyref{lem:GoodBlocking}, this will be enough to conclude. 

\smallskip

Let $(\U_n)_{n \in \N}$ be a decreasing sequence of Ellentuck-open subsets of $\Sub(X)$ such that $\bigcap_{n \in \N} \U_n = \H$. We describe a run $(U_0, S_{G_0}, U_1, S_{G_1}, \ldots)$ of the game $SF_U$ where \I{} plays according to $\tau$, by describing the moves of \II{}. Suppose that $U_0, S_{G_0}, \ldots, U_{n - 1}, S_{G_{n - 1}}$ have just been played. According to the strategy $\tau$, player \I{} plays $U_n$, a finite-codimensional subspace of $U$. Since $U \in \H$, we have $[G_i \mid i < n] \oplus U_n\in \H \subseteq \U_n$. So we can find a finite-dimensional subspace $G_n \subseteq U_n$ such that $\Big[[G_i \mid i \leqslant n], [G_i \mid i < n] \oplus U_n\Big] \subseteq \U_n$. We make \II{} play $S_{G_n}$, finishing the construction.

\smallskip

Exactly in the same way as in the proof of \prettyref{lem:GoodFDD}, we can prove that $(G_n)_{n \in \N}$ is an FDD of a subspace $Z \in \H_{\restriction U}$, with constant at most $1 + \varepsilon$. Since the game $SF_U$ has been played according to $\tau$, we have that $\bs((S_{G_n})_{n \in \N}) \subseteq \X^c$, which exactly means that every normalized block-sequence of the FDD $(G_n)$ is in $\X^c$.

\smallskip

In the case where $\H = \H_d$ for some degree $d$, then by \prettyref{lem:VeryGoodBlocking}, we can replace the FDD $(G_n)$ with one of its blocking which is $d$-better, and this blocking will still satisfy the conclusion of the theorem.

\smallskip

We now prove the refinement of the theorem in the case where $Y$ has a fixed FDD $(F_n)_{n \in \N}$. Without loss of generality, we assume that the sequence $\Delta = (\Delta_n)_{n \in \N}$ is decreasing, that $\Delta_0 < 1$, and that $\left(1 + \frac{\varepsilon}{2}\right)\left(\frac{1 + \Delta_0}{1 - \Delta_0}\right) \leqslant 1 + \varepsilon$. The general case applied to $\X' = (\X)_{\frac{\Delta}{2}}$ (which is still analytic) and to the sequence $\Delta' = \frac{\Delta}{2}$ gives a $U \in \H_{\restriction Y}$ such that either player \II{} has a strategy in $G_U$ to reach $(\X')_{\Delta'}$, or $U$ has an $\H$-good FDD $(K_n)_{n \in \N}$ with constant at most $1 + \frac{\varepsilon}{2}$ all of whose normalized block-sequences belong to $(\X')^c$. In the first case, we are done, since $(\X')_{\Delta'} \subseteq \X_\Delta$. In the second case, we apply \prettyref{lem:GoodBlockFDD} to $Y_k = Y$, $(F_n^k)_{n \in \N} = (F_n)_{n \in \N}$, $U_k = [K_n \mid n \geqslant k]$ for every $k \in \N$. This gives us a $Z \in \H_{\restriction Y}$ spanned by an $\H$-good block-FDD $(G_n)_{n \in \N}$ of $(F_n)$, and an isomorphic embedding $T \colon Z \to U$ such that for every $n \in \N$, $\left\|(T - \Id_Z)_{\restriction [G_k \mid k \geqslant n]}\right\| < \frac{\Delta_n'}{2}$ and $T(G_n) \subseteq U_n$. Modifying $T$ if necessary, we can even assume that for every $n$, $T(G_n)$ has finite support on the FDD $(K_n)$ (of course, doing such a modification does not necessarily preserve the fact that the FDD $(T(G_n))_{n \in \N}$ is $\H$-good, but this fact will not matter in this proof). Since $T(G_n) \subseteq U_n$ for every $n$, we have that $\lim_{n \to \infty} \min \supp(G_n) = \infty$ (the supports being taken with respect to the FDD $(K_n)$). Thus, extracting a subsequence if necessary, we can assume that $(T(G_n))$ is a block-FDD of $(K_n)$.

\smallskip

We now prove that the FDD $(G_n)$ is as wanted. Recall that $\|T - \Id_Z\| \leqslant \Delta_0$, so $\|T\|\leqslant 1 + \Delta_0$ and $\|T^{-1}\| \leqslant \frac{1}{1 - \Delta_0}$. Since $(T(G_n))$ is a block-FDD of $(K_n)$, which has constant at most $1 + \frac{\varepsilon}{2}$, we deduce that $(T(G_n))$ as well has constant at most $1 + \frac{\varepsilon}{2}$. Hence, $(G_n)$ has constant at most $\left(1 + \frac{\varepsilon}{2}\right)\left(\frac{1 + \Delta_0}{1 - \Delta_0}\right) \leqslant 1 + \varepsilon$, as wanted. Now let $(x_i)_{i \in \N}$ be a normalized block-sequence of $(G_n)$ ; we prove that $(x_i) \in \X^c$. For every $i$, $x_i \in [G_n \mid n \geqslant i]$ so $\|T(x_i) - x_i\| \leqslant \frac{\Delta_i'}{2}$. Hence, letting $y_i = \frac{T(x_i)}{\|T(x_i)\|}$, we have $\|x_i - y_i\| \leqslant \Delta_i'$. Observe that $(y_i)_{i \in \N}$ is a normalized block-sequence of $(T(G_n))$, so of $(K_n)$; hence, it is in $(\X')^c$. Thus, $(x_i)$ is in $\left((\X')^c\right)_{\Delta'}$, which is contained in $\X^c$. This finishes the proof.

\end{proof}

\begin{rem}
The essential difference between \prettyref{thm:LocalGowersThm} and Smythe's local version of Gowers' Ramsey-type theorem proved in \cite{Smythe} is the fact that, in Smythe's theorem, the original Gowers' game appears: player \I{} is allowed to play whatever subspace he wants, not only elements of $\H$. The cost is that the conditions on the family $\H$ are much more restrictive in Smythe's theorem than in our theorem. Thus, it is not clear at all that Smythe's theorem could apply for the families we shall consider (for instance, the family of non-Hilbertian subspaces of a given Banach space).
\end{rem}

\bigskip

\subsection{The dichotomy}

We now want to prove a local version of Gowers' first dichotomy (\prettyref{thm:Gowers1stDichoto}), that is, a similar dichotomy where we moreover ensure that the subspace given as a result will be in a fixed D-family. To do this, we will need local versions of the two possible conclusions. In particular, we will need a weakening of the notion of HI space.

\begin{defin}

Let $X$ be a Banach space and let $\H$ be a D-family of subspaces of $X$. We say that a subspace of $X$ is \emph{$\H$-decomposable} if it is equal to the direct sum of two elements of $\H$. The space $X$ is \emph{hereditarily $\H$-indecomposable}, or \emph{$\H$-HI}, if $X \in \H$ and $X$ contains no $\H$-decomposable subspace.
    
\smallskip

If $d$ is a degree, we call a space $X$ \emph{hereditarily $d$-indecomposable}, or \emph{$d$-HI}, if it is hereditarily $\H_d^X$-indecomposable. In other words, if it is $d$-large and if no subspace of it is a direct sum of two $d$-large subspaces.
\end{defin}

In the case where $\H = \Subinf(X)$, i.e. where $d$ is the dimension, we recover the notion of HI spaces.

\begin{thm}[The first dichotomy]\label{thm:FirstDichoto}

Let $X$ be a Banach space, and let $\H$ be a D-family of subspaces of $X$, containing $X$. Then there exists $Y \in \H$ such that:

\begin{itemize}
    \item either $Y$ has an $\H$-good UFDD;
    \item or $Y$ is hereditarily $\H_{\restriction Y}$-indecomposable.
\end{itemize}

\noindent Moreover, if $X$ comes with a fixed FDD, then in the first case, the UFDD of $Y$ can be taken as a block-FDD of the FDD of $X$.

\end{thm}

This is a true dichotomy in the sense that both classes it defines are, in some way, hereditary with respect to $\H$ (every block-FDD of a UFDD is a UFDD, and if $Y$ is hereditarily $\H_{\restriction Y}$-indecomposable, then every subspace $Z\in \H_{\restriction Y}$ is hereditarily $\H_{\restriction Z}$-indecomposable); and these classes are very obviously disjoint, since a subspace $Y$ with an $\H$-good UFDD has continuum-many decompositions as a direct sum of elements of $\H$.

We spell out the version of the dichotomy when the D-family is induced by a degree $d$, taking into account \prettyref{lem:VeryGoodBlocking}.

\begin{thm}[The first dichotomy for degrees]\label{thm:FirstDichotofordegrees}

Let $X$ be a Banach space, and let $d$ be a degree such that $X$ is $d$-large. Then there exists $Y \subseteq X$ a $d$-large subspace which either has a $d$-better UFDD, or
     is hereditarily $d$-indecomposable.


\end{thm}

\begin{proof}[Proof of \prettyref{thm:FirstDichoto}]

 We fix $\Delta=(\Delta_i)_{i \in \N}$ a sequence of positive real numbers satisfying the following property: for every normalized basic sequence $(x_i)_{i \in \N}$ in $X$ with constant at most $2$, and for every normalized sequence $(y_i)_{i \in \N}$ in $X$ such that $\forall i \in \N \; \|x_i - y_i\| \leqslant \Delta_i$, the sequences $(x_i)$ and $(y_i)$ are $2$-equivalent.
 Let $\X$ be the set of sequences $(x_i)_{i \in \N} \in (S_X)^\N$ satisfying the following property: for every $N \in \N$, there exists an eventually null sequence $(a_i)_{i \in \N} \in \R^\N$ such that $\left\|\sum_{i\text{ even}} a_{i}x_{i} \right\| > N \left\|\sum_{i \in \N} a_ix_i \right\|$. The set $\X$ is a $G_\delta$ subset of $(S_X)^\N$. We apply \prettyref{thm:LocalGowersThm} to $X$, to the set $\X$, to the sequence $\Delta$, and to $\varepsilon = 1$.


\smallskip

\paragraph{First case:}\emph{There exists $Y \in \H$ with an $\H$-good FDD $(F_n)_{n \in \N}$ such that no normalized block-sequence of $(F_n)$ belongs to $\X$. Moreover, if $X$ comes with a fixed FDD, then $(F_n)$ is a block-FDD of the FDD of $X$.}

\smallskip

We then show that $(F_n)$ is a UFDD. Let, for every $n$, $y_n \in F_n$ and let $A \subseteq \N$ be infinite and coinfinite, and suppose that $\sum_{n\in\N} y_n$ converges; we show that $\sum_{n \in A}y_n$ converges. Without loss of generality, we can assume that $0\in A$. Consider a sequence $0 = n_0 < n_1 < n_2 < ...$ of integers such that $A = \bigcup_{i\text{ even}} \llbracket n_i, n_{i+1}-1 \rrbracket$. For every $i \in \N$, consider $x_i \in [F_n \mid n_i \leqslant n < n_{i+1}]$ with $\|x_i\| = 1$ and $a_i \in \R$ such that $\sum_{n_i \leqslant n < n_{i+1}} y_n = a_i x_i$. Then $(x_i)_{i \in \N}$ is a normalized block-sequence of $(F_n)$, so does not belong to $\X$. Hence, there exists $N\in\N$ such that for every $k \leqslant l$, we have: $$\left\|\sum_{\substack{k \leqslant i < l\\ i\text{ even}}} a_{i}x_{i} \right\| \leqslant N \left\|\sum_{k \leqslant i < l} a_ix_i \right\|.$$

\smallskip

We show that $\sum_{n \in A}y_n$ converges using Cauchy criterion. Fix $\varepsilon > 0$; there is $n_\varepsilon \in \N$ such that for every $q\geqslant p \geqslant n_\varepsilon$, we have $\left\|\sum_{p \leqslant n < q} y_n\right\| \leqslant \varepsilon$; we can moreover assume that $n_\varepsilon$ is one of the $n_k$'s. Fix $q \geqslant p \geqslant n_\varepsilon$. Fixing $k$ and $l$ such that $n_{k-1} \leqslant p < n_k$ and $n_l \leqslant q < n_{l+1}$, we have:
\begin{eqnarray*}
\left\|\sum_{\substack{p \leqslant n < q\\ n \in A}} y_n\right\| & \leqslant & \left\|\sum_{p \leqslant n < n_k} y_n \right\| + \left\|\sum_{\substack{n_k \leqslant n < n_l\\ n \in A}} y_n\right\| + \left\|\sum_{n_l \leqslant n < q} y_n \right\|\\
&\leqslant& 2\varepsilon + 
\left\|\sum_{\substack{k \leqslant i < l\\ i\text{ even}}} a_{i}x_{i} \right\|\\
& \leqslant& 2 \varepsilon +N \left\|\sum_{k \leqslant i < l} a_ix_i \right\|\\
&=&2 \varepsilon + N\left\|\sum_{n_k \leqslant n < n_l} y_n \right\|\\
&\leqslant & (N+2)\varepsilon,
\end{eqnarray*}
concluding this first case.


\smallskip

\paragraph{Second case:}\emph{There exists $Y \in \H$ such that player \II{} has a strategy in $G_Y$ to reach $(\X)_\Delta$.}

\smallskip


We then show that $Y$ is hereditarily $\H_{\restriction Y}$-indecomposable. Fix $U, V \in \H_{\restriction Y}$ and $N \in \N$. We will build $u \in U$ and $v \in V$ such that $\|u\| > \frac{N}{4} \|u + v\|$, which will be enough to conclude. Consider a run of the game $G_Y$:

\smallskip

\begin{tabular}{ccccccccccc}
\textbf{I} & $Z_0 \cap U$ & & $Z_1 \cap V$ & & $Z_2 \cap U$ & & $Z_3 \cap V$ & & $\hdots$ & \\
\textbf{II} & & $z_0$ & & $z_1$ & & $z_2$ & & $z_3$ & & $\hdots$ 
\end{tabular}

\smallskip

\noindent where \II{} plays using a strategy to reach $(\X)_\Delta$, and where \I{} plays as follows:
\begin{itemize}
    \item if $n$ is even, \I{} plays $Z_n \cap U$ where $Z_n$ is a finite-codimensional subspace of $Y$ such that the natural projection $[z_i \mid i < n] \oplus Z_n \to [z_i \mid i < n]$ has norm at most $2$, and such that $Z_n \subseteq Z_{n-1}$ if $n\geqslant 1$;
    \item if $n$ is odd, \I{} plays $Z_n \cap V$ for $Z_n$ exactly as previously.
\end{itemize}
At the end of the game, player \II{} has built a normalized basic sequence $(z_n)_{n \in \N}$ with constant at most $2$, which is in $(\X)_\Delta$, and such that for $n$ even, $z_n \in U$, and for $n$ odd, $z_n \in V$.

\smallskip

Now choose a sequence $(z_n')_{n \in \N} \in \X$ such that for every $n \in \N$, $\|z'_n - z_n\| \leqslant \Delta_n$. Choose $(a_n)_{n \in \N} \in \R^\N$ eventually null such that $\left\|\sum_{n\text{ even}} a_{n}z_{n}' \right\| > N \left\|\sum_{n \in \N} a_nz_n' \right\|$. By the choice of $\Delta$, the sequences $(z_n)$ and $(z_n')$ are $2$-equivalent, so we have:$$\left\|\sum_{n\text{ even}} a_{n}z_{n} \right\| \geqslant \frac{1}{2} \left\|\sum_{n\text{ even}} a_{n}z_{n}' \right\| > \frac{N}{2} \left\|\sum_{n \in \N} a_nz_n' \right\| \geqslant \frac{N}{4} \left\|\sum_{n \in \N} a_nz_n \right\|.$$ Thus, $u = \sum_{n\text{ even}} a_{n}z_{n}$ and $v = \sum_{n\text{ odd}} a_{n}z_{n}$ satisfy the wanted property.
\end{proof}

\bigskip\bigskip

\section{The second  dichotomy}

In this section, we prove our second dichotomy, a local version of Ferenczi--Rosendal's dichotomy between minimal and tight spaces (\prettyref{thm:MinimalTight}). We also discuss some consequences, in particular concerning ergodicity.

\bigskip

\subsection{The statement of the dichotomy}

In this subsection, we state our second dichotomy. As for the first one, we first need to provide appropriate local versions of the notions of minimality and of tightness. In the whole section, we fix a Banach space $X$, a D-family $\H$ of subspaces of $X$, and a degree $d$.

\begin{defin}

We say that $X$ is \emph{$\H$-minimal} if $X \in \H$ and if $X$ isomorphically embeds into every element of $\H$. If $\H$ is induced by the degree $d$, then we say that $X$ is $d$-minimal.

\end{defin}

So $X$ is $d$-minimal if it is $d$-large and embeds into any of its $d$-large subspaces - again note that $d$-minimality is not a notion of ``smallness".
In particular, if $d$ is the dimension (or equivalently if $\H = \Subinf(X)$), we recover the usual notion of minimality. Also observe that if $X$ is $\H$-minimal, then it is separable; this is for instance a consequence of \prettyref{lem:GoodFDD}.

\begin{defin}

Let $(F_n)_{n \in \N}$ be an FDD of a subspace of $X$.

\begin{enumerate}
  \item The FDD $(F_n)$ is said to be \textit{$\H$-tight}  if every $Y \in \H$  is tight in $(F_n)$.
    
    \item The space $X$ is said to be \textit{$\H$-tight} if $X \in \H$ and if $X$ has an $\H$-tight FDD.
    
\end{enumerate}

If $\H$ is induced by the degree $d$, then we say that the FDD $(F_n)$ is $d$-tight, and that the space $X$ is $d$-tight.

\end{defin}



    


So $X$ is $d$-tight if it is $d$-large and has an FDD in which every $d$-large Banach space is tight. When $d(F)$ is the dimension of $F$, we recover the usual notion of tight FDD.

\smallskip

Note that
$\H$-minimality is a hereditary notion in the sense that if $X$ is $\H$-minimal, then every $Y \in \H$ is $\H_{\restriction Y}$-minimal. The notion of $\H$-tightness is also hereditary in the following sense:

\begin{lem}\label{lem:HereditaryTightness}

Let $(F_n)_{n \in \N}$ be an FDD of a subspace of $X$.

\begin{enumerate}

\item If a Banach space $Y$ is tight in $(F_n)$, then it is also tight in all of its block-FDD's.

\item If $(F_n)$ is $\H$-tight, then all of its block-FDD's are $\H$-tight. In particular, if the FDD $(F_n)$ witnesses that $X$ is $\H$-tight, then every $Y\in\H$ generated by a block-FDD of $(F_n)$ is $\H_{\restriction Y}$-tight.

\end{enumerate}

\end{lem}

\begin{proof}

We only prove 1., since 2. is an immediate consequence. Let $(G_m)_{m \in \N}$ be a block-FDD of $(F_n)$, and let $I_0 < I_1 < I_2 < \ldots$ be a sequence of nonempty successive intervals witnessing the tightness of $Y$ in $(F_n)$. Observe that every infinite subsequence of $(I_i)$ still witnesses the tightness of $Y$ in $(F_n)$. Thus, without loss of generality, we can assume that for every $m \in \N$, there is at most one $i \in \N$ such that $I_i \cap \supp(G_m) \neq \varnothing$. If there are infinitely many $I_i$'s that intersect no set of the form $\supp(G_m)$, then by tightness, $Y \nsqsubseteq \left[G_m \mid m \in \N\right]$ so $Y$ is tight in $(G_m)$. Otherwise, passing again to a subsequence if necessary, we can assume that for every $i \in \N$, $I_i$ intersects at least one of the $\supp(G_m)$'s. We let, for every $i \in \N$, $J_i = \{m \in \N \mid I_i \cap \supp(G_m) \neq \varnothing\}$. Then the $J_i$'s are nonempty intervals and satisfy $J_0 < J_1 < J_2 < \ldots$; moreover, by construction, for every infinite $A \subseteq \N$ we have $\left[G_m \mid m \notin \bigcup_{i \in A}J_i\right] \subseteq \left[F_n \mid n \notin \bigcup_{i \in A}I_i\right]$, so $Y \nsqsubseteq \left[G_m \mid m \notin \bigcup_{i \in A}J_i\right]$. This shows that $Y$ is tight in $(G_m)$.

\end{proof}


\begin{cor}\label{cor:GoodTightFDD}
If $X$ is $\H$-tight (resp. $d$-tight), then it has an FDD which is $\H$-tight and $\H$-good (resp. $d$-tight and $d$-better).
\end{cor}

\begin{proof}
In the case of a D-family, starting from any $\H$-tight FDD $(F_n)_{n \in \N}$ of $X$, we can find a blocking $(G_n)_{n \in \N}$ of this FDD which is $\H$-good, using \prettyref{lem:GoodBlocking}. The FDD $(G_n)$ is still $\H$-tight, by \prettyref{lem:HereditaryTightness}. In the case of a degree, the proof is the same, using this time \prettyref{lem:VeryGoodBlocking} to pass to a better blocking.
\end{proof}


\begin{thm}[The second dichotomy]\label{thm:2ndDichoto}

Suppose that $X \in \H$. Then $X$ has a subspace $Y \in \H$ such that:

\begin{itemize}
    \item either $Y$ is $\H_{\restriction Y}$-minimal;
    
    \item or $Y$ is $\H_{\restriction Y}$-tight.
\end{itemize}

\noindent Moreover, if $X$ comes with a fixed FDD, then in the second case, the $\H$-tight FDD of $Y$ can be taken as a block-FDD of the FDD of $X$.

\end{thm}

This is a true dichotomy: indeed, as we already saw, the notions of $\H$-minimality and $\H$-tightness are hereditary in a certain sense, and obviously an $\H$-tight space cannot be $\H$-minimal.

\smallskip

It is worth spelling out the version of the second dichotomy for degrees:

\begin{thm}[The second dichotomy for degrees]\label{thm:2ndDichotoForDegrees}

Suppose that $X$ is $d$-large. Then $X$ has a $d$-large subspace $Y$ which is either $d$-minimal or $d$-tight.


\end{thm}

In the case where $d(F) = \dim(F)$, we get back \prettyref{thm:MinimalTight}.

\smallskip

The rest of this section is organized as follows. In \prettyref{subsec:Proof2ndDichoto}, we prove \prettyref{thm:2ndDichoto}. Then, in \prettyref{subsec:MinimalTight}, we study the properties of $\H$-minimal and $\H$-tight spaces, and we deduce some consequences of \prettyref{thm:2ndDichoto}.

\bigskip

\subsection{The proof of \prettyref{thm:2ndDichoto}}\label{subsec:Proof2ndDichoto}

This proof is inspired by the proof by Rosendal of a variant of the minimal/tight dichotomy \cite{RosendalAlphaMinimal}. This dichotomy will again be proved using combinatorial methods, however its proof is quite delicate and thus, cannot be done in the formalism of approximate Gowers spaces. We will, instead, use the formalism of Gowers spaces, and work with countable vector spaces instead of Banach spaces.

\smallskip

In the general case, we can reduce to the case where $X$ has an $\H$-good FDD, using \prettyref{lem:GoodFDD}. In the case where $X$ already comes with a fixed FDD, then we can assume that this FDD is $\H$-good, using \prettyref{lem:GoodBlocking}. So, in what follows, we will consider that $X$ comes with a fixed $\H$-good FDD $(E_n)_{n \in \N}$, and we will prove that either $X$ has a subspace $Y$ which is $\H_{\restriction Y}$-minimal, or that $(E_n)$ has an $\H$-tight block-FDD.

\smallskip

Let $C$ be the constant of the FDD $(E_n)$. For every $n \in \N$, let $d_n = \sum_{i < n} \dim(E_n)$ and fix $(e_i)_{d_n \leqslant i < d_{n + 1}}$ a normalized basis of $E_n$. Let $K$ be a countable subfield of $\R$ having the following property: for every eventually null sequence $(x_i)_{i \in \N} \in K^\N$, we have $\left\|\sum_{i \in \N} x_ie_i\right\| \in K$. Such a field can be built in the following way: we fix $K_0 = \mathbb{Q}$, for every $n \in \N$, we let $K_{n+1}$ be the subfield of $\R$ generated by $K_n$ and by all reals of the form $\left\|\sum_{i \in \N} x_ie_i\right\|$, where $(x_i)_{i \in \N}$ is an eventually null sequence of elements of $K_n$, and finally we let $K = \bigcup_{n \in \N} K_n$. In the rest of this subsection, vector spaces on $K$ will be denoted by capital script roman letters, and closed $\R$-vector subspaces of $E$ (of finite or infinite dimension) will be denoted by block roman letters. Let $\mathscr{X}$ be the $K$-vector subspace of $X$ generated by all the $e_i$'s. For $\mathscr{Y}$ a (finite- or infinite-dimensional) $K$-vector subspace of $\mathscr{X}$, we let $\overline{\mathscr{Y}}$ be its closure in $X$. This is an $\R$-vector subspace of $X$, and we have $\overline{\mathscr{X}} = X$. Also let $S_\mathscr{Y}$ be the set of normalized vectors of $\mathscr{Y}$.  Since, for $x \in \mathscr{Y} \setminus \{0\}$, we have $\frac{x}{\|x\|} \in \mathscr{Y}$, we deduce that $S_\mathscr{Y}$ is dense in $S_{\overline{\mathscr{Y}}}$.

\begin{lem}

Let $\mathscr{Y}$ be a $K$-vector subspace of $\mathscr{X}$. Then $\overline{\mathscr{Y}}$ is $\R$-finite-dimensional if and only if $\mathscr{Y}$ is $K$-finite-dimensional, and in this case, their dimensions are equal.

\end{lem}

\begin{proof}

Let $(f_0, \ldots, f_{k-1})$ be a $K$-free family in $\mathscr{Y}$. Let $N \in \N$ be such that all the $f_l$'s are in $\spa_K(e_0, \ldots e_{N-1})$, and let $M$ be the matrix of the family $(f_1, \ldots f_k)$ in the family $(e_0, \ldots, e_{N-1})$. Then, on the field $K$, the matrix $M$ has at least one $k \times k$ nonzero minor. But the determinant does not depend on the field, so this is also true on $\R$. Hence, the family $(f_0, \ldots, f_{k-1})$ is $\R$-free. We deduce that $\dim_\R(\overline{\mathscr{Y}}) \geqslant \dim_K(\mathscr{Y})$.

\smallskip

Conversely, if $(f_0, \ldots, f_{k-1})$ is a $K$-generating family in $\mathscr{Y}$, then this is a $\R$-generating family in $\spa_\R(\mathscr{Y})$, which is equal to $\overline{\mathscr{Y}}$ since it is finite-dimensional. So $\dim_\R(\overline{\mathscr{Y}}) \leqslant \dim_K(\mathscr{Y})$.

\end{proof}

All along this subsection, we will use the following notation: if $(\mathscr{U}_i)_{i \in I}$ is a sequence of finite-dimensional vector subspaces of $\mathscr{X}$, we let $[\mathscr{U}_i \mid i \in I]$ be the $K$-vector subspace of $\mathscr{X}$ spanned by the $\mathscr{U}_i$'s. For every $n \in \N$, we let $\mathscr{E}_n$ be the $K$-vector subspace of $E_n$ generated by the $e_i$'s for $d_n \leqslant i < d_{n + 1}$, and we let $\mathscr{E} = (\mathscr{E}_n)_{n \in \N}$. Obviously we have $\overline{\mathscr{E}_n} = E_n$ and $\mathscr{X} = [\mathscr{E}_n \mid n \in \N]$. For $(\mathscr{F}_n)_{n \in \N}$ a sequence of nonzero finite-dimensional $K$-vector subspaces of $\mathscr{X}$ whose sum is a direct sum, we define a \textit{block-FDD} of $(\mathscr{F}_n)$ as a sequence $(\mathscr{G}_m)_{m \in \N}$ of nonzero finite-dimensional $K$-vector subspaces of $\mathscr{X}$ for which there exists a sequence $A_0 < A_1 < \ldots$ of finite sets of integers such that for every $m$, we have $\mathscr{G}_m \subseteq \oplus_{n \in A_m} \mathscr{F}_n$. In what follows, we will only consider block-FDD's of $\mathscr{E}$. A block-FDD $(\mathscr{F}_m)_{m \in \N}$ of $\mathscr{E}$ will often be denoted by the letter $\mathscr{F}$; thus, when we speak about a block-FDD $\mathscr{F}$ without further explanation, it will be supposed that its terms are denoted by $\mathscr{F}_m$, and we will also let $[\mathscr{F}] = [\mathscr{F}_m \mid m \in \N]$. Observe that if $\mathscr{F}$ is a block-FDD of $\mathscr{E}$, then $\left(\overline{\mathscr{F}_m}\right)_{m \in \N}$ is a block-FDD of $(E_n)$. So we will say that $\mathscr{F}$ is \textit{good} if and only if $\left(\overline{\mathscr{F}_m}\right)$ is an $\H$-good block-FDD of $(E_n)$. If $\mathscr{F}$ is a block-FDD of $\mathscr{E}$ and $m_0 \in \N$, we will denote by $\mathscr{F}^{(m_0)}$ the block-FDD $(\mathscr{F}_{m + m_0})_{m \in \N}$. If $\mathscr{F}$ is good, then $\mathscr{F}^{(m_0)}$ is also good.

\smallskip

We now define the Gowers space in which we will work. We let $\mathbb{P}$ be the set of good block-FDD's of $\mathscr{E}$. If $\mathscr{F}, \mathscr{G} \in \mathbb{P}$, we let $\mathscr{F} \leqslant \mathscr{G}$ if $\mathscr{F}$ is a block-FDD of $\mathscr{G}$. We let $\mathscr{F} \leqslant^* \mathscr{G}$ if there exists $m \in \N$ such that $\mathscr{F}^{(m)} \leqslant \mathscr{G}$. We let $\Pi$ be the set of pairs $(\mathscr{U}, x)$ where $\mathscr{U}$ is a nonzero finite-dimensional subspace of $\mathscr{X}$ and $x$ is an element of $\mathscr{X}$. For $\mathscr{F} \in \mathbb{P}$ and a sequence $(\mathscr{U}_0, x_0, \ldots, \mathscr{U}_k, x_k) \in \Seq(\Pi)$, we say that $(\mathscr{U}_0, x_0, \ldots, \mathscr{U}_k, x_k) \vartriangleleft \mathscr{F}$ if $\mathscr{U}_k  \subseteq [\mathscr{F}]$ and $x_k \in [\mathscr{U}_l\mid l \leqslant k]$.

\begin{lem}

$\G = (\mathbb{P}, \Pi, \leqslant, \leqslant^*, \vartriangleleft)$ is a Gowers space.

\end{lem}

\begin{proof}

The only nontrivial thing to verify is that the diagonalization axiom is satisfied. So let $(\mathscr{F}^k)_{k \in \N}$ be a $\leqslant$-decreasing sequence of elements of $\mathbb{P}$. We apply \prettyref{lem:GoodBlockFDD} to $U_k = Y_k = \overline{[\mathscr{F}^k]}$, to $F_n^k = \overline{\mathscr{F}_n^k}$, and to $D_k = S_{[\mathscr{F}^k]}$. We get an $\H$-good FDD $(G_n)_{n \in \N}$ of a subspace of $X$ such that for every $k \in \N$, $(G_{n + k})_{n \in \N}$ is a block-FDD of $(F_n^k)_{n \in \N}$, and such that $G_k$ has a basis in $D_k$. This last condition shows that $G_k$ can be written as $\overline{\mathscr{G}_k}$, where $\mathscr{G}_k$ is a finite-dimensional subspace of $[\mathscr{F}^k]$. Since $(G_{n + k})_{n \in \N}$ is a block-FDD of $(\overline{\mathscr{F}_n^k})_{n \in \N}$, and since all of the $\mathscr{G}_{n + k}$'s, for $n \in \N$, are vector subspaces of $[\mathscr{F}^k]$, we deduce that $(\mathscr{G}_{n + k})_{n \in \N}$ is a block-FDD of $\mathscr{F}^k$. Thus, $\mathscr{G} = (\mathscr{G}_n)_{n \in \N}$ is in $\mathbb{P}$, and we have, for every $k \in \N$, $\mathscr{G} \leqslant^* \mathscr{F}^k$, as wanted.




\end{proof}

From now, we work in the Gowers space $\G$. The asymptotic game $F_\mathscr{F}$, Gowers' game $G_\mathscr{F}$, or the adversarial Gowers' games $A_\mathscr{F}$ and $B_\mathscr{F}$, will always be considered with respect to this space. To save writing, we will make the following abuse of notation: in a play of $F_\mathscr{F}$ or $G_\mathscr{F}$ played as follows:

\smallskip

\begin{tabular}{ccccccc}
\textbf{I} & $\mathscr{F}^0$ & & $\mathscr{F}^1$ & & $\hdots$ & \\
\textbf{II} & & $\mathscr{U}_0, \, x_0$ & & $\mathscr{U}_1, \, x_1$ & & $\hdots$ 
\end{tabular}

\smallskip

\noindent we will consider that the outcome of the game is the sequence $(x_0, x_1, \ldots)$ (according to the definition given in \prettyref{sec:GowersSpaces}, this should be $(\mathscr{U}_0, x_0, \mathscr{U}_1, x_1, \ldots)$). Similarly, in a play of $A_\mathscr{F}$ or $B_\mathscr{F}$ played as follows:

\smallskip

\begin{tabular}{cccccccc}
\textbf{I} & & $\mathscr{U}_0, \, x_0, \,  \mathscr{G}^0$ & & $\mathscr{U}_1, \, x_1, \,  \mathscr{G}^1$ & & $\hdots$ \\
\textbf{II} & $\mathscr{F}^0$ & & $\mathscr{V}_0, \, y_0, \,  \mathscr{F}^1$ & & $\mathscr{V}_1, \, y_1, \,  \mathscr{F}^2$ & & $\hdots$ 
\end{tabular}

\smallskip

\noindent we will consider that the outcome of the game is the pair of sequences $((x_0, x_1, \ldots), \linebreak (y_0, y_1, \ldots))$. Hence, for instance saying that player \II{} has a strategy in $B_\mathscr{F}$ to produce two equivalent sequences means that player \II{} has a strategy to ensure that the sequences $(x_i)_{i \in \N}$ and $(y_i)_{i \in \N}$ produced during the game are equivalent, for the usual notion of equivalence between sequences in a Banach space.

\smallskip

Observe that in this Gowers space, for $\mathscr{F}, \mathscr{G} \in \mathbb{P}$, if $\mathscr{F} \lessapprox \mathscr{G}$, then there exist $m, n \in \N$ such that $\mathscr{F}^{(m)} = \mathscr{G}^{(n)}$.

\smallskip

\begin{lem}\label{lem:AdvRamseyCor}

There exists $\mathscr{F} \in \mathbb{P}$ having the following property: either player \I{} has a strategy in $A_\mathscr{F}$ to produce two inequivalent sequences, or player \II{} has a strategy in $B_\mathscr{F}$ to produce two equivalent sequences.

\end{lem}

\begin{proof}

The set of pairs $((x_i)_{i \in \N}, (y_i)_{i \in \N}) \in \left(\mathscr{X}^\N\right)^2$ that are equivalent is an $F_\sigma$ subset of $\left(\mathscr{X}^\N\right)^2$ for the product of the discrete topologies on $\mathscr{X}$. Thus, this result is a direct consequence of \prettyref{thm:AdvRamsey}.

\end{proof}

We now fix an $\mathscr{F} \in \mathbb{P}$ as given by \prettyref{lem:AdvRamseyCor}. We say that a sequence $(u_i)_{i \in \N} \in \mathscr{X}^\N$ is \textit{$\mathscr{F}$-correct} if there exist $\mathscr{G} \leqslant \mathscr{F}$ and a partition of $\N$ into nonempty successive intervals $I_0 < I_1 < \ldots$ such that for every $m \in \N$, the finite sequence $(u_i)_{i \in I_m}$ is a basis of $\mathscr{G}_m$. The next proposition contains the combinatorial content of \prettyref{thm:2ndDichoto}.

\begin{prop}\label{prop:2ndDichotoCombi}

At least one of the following statements is satisfied:

\begin{enumerate}

\item For every $\mathscr{F}$-correct sequence $(u_i)_{i \in \N}$, player \I{} has a strategy in $F_\mathscr{F}$ to build a sequence $(x_i)_{i \in \N}$ that is not equivalent to $(u_i)$;

\item There exists an $\mathscr{F}$-correct sequence $(u_i)_{i \in \N}$ such that player \II{} has a strategy in $G_\mathscr{F}$ to build a sequence $(x_i)_{i \in \N}$ that is equivalent to $(u_i)$.

\end{enumerate}

\end{prop}

\begin{proof}

We assume that 1. is not satisfied and we prove 2. For the rest of the proof, we fix an $\mathscr{F}$-correct sequence $(u_i)_{i \in \N}$ such that player \I{} has no strategy in $F_\mathscr{F}$ to build a sequence that is not equivalent to $(u_i)$. By the determinacy of this game (for the fundamentals on the theory of determinacy, see \cite{Kechris}, Section 20), player \II{} has a strategy $\tau$ in $F_\mathscr{F}$ to build a sequence which is equivalent to $(u_i)$. By correctness of the sequence $(u_i)$, we can also fix $\mathscr{G} \leqslant \mathscr{F}$ and a partition of $\N$ in nonempty successive intervals $I_0 < I_1 < \ldots$ such that for every $m \in \N$, $(u_i)_{i \in I_m}$ is a basis of $\mathscr{G}_m$.

\smallskip

\paragraph{First step:} \textit{Player \II{} has a strategy in $A_\mathscr{F}$ to build two equivalent sequences.}

\smallskip

We describe this strategy on a play $(\mathscr{G}, \mathscr{U}_0, x_0, \mathscr{F}^0, \mathscr{V}_0, y_0, \mathscr{G}, \ldots)$ of $A_\mathscr{F}$, in which the FDD's played by \II{} will always be equal to $\mathscr{G}$. This game will be played at the same time as an auxiliary play $(\mathscr{H}^0, \mathscr{W}_0, z_0, \mathscr{H}^1, \mathscr{W}_1, z_1, \ldots)$ of $F_\mathscr{F}$ during which player \II{} always plays according to her strategy $\tau$. Actually, the $\mathscr{U}_i$'s played by \I{} in $A_\mathscr{F}$ will not matter at all in this proof, so we will omit them in the notation; the only thing to observe is that for every $i \in \N$, we will necessarily have $x_i \in [\mathscr{G}]$. At the same time as the games are played, a sequence of integers $0 = k_0 < k_1 < \ldots$ will be constructed. The idea is that the turn $i$ of the game $A_\mathscr{F}$ will be played at the same time as the turns $k_i, k_i + 1, \ldots, k_{i + 1} - 1$ of the game $F_\mathscr{F}$. Suppose that we are just before the turn $i$ of the game $A_\mathscr{F}$, so the $x_j$'s, the $\mathscr{F}^j$'s, the $\mathscr{V}_j$'s, and the $y_j$'s have been defined for all $j < i$. Also suppose that the integers $k_j$ have been defined for all $j \leqslant i$, and that we are just before the turn $k_i$ of the game $F_\mathscr{F}$, so the $\mathscr{H}^k$'s, the $\mathscr{W}_k$'s and the $z_k$'s have been played for all $k < k_i$. We represent on the diagram below the turn $i$ of the game $A_\mathscr{F}$, and the turns $k_i, \ldots, k_{i + 1} - 1$ of the game $F_\mathscr{F}$.

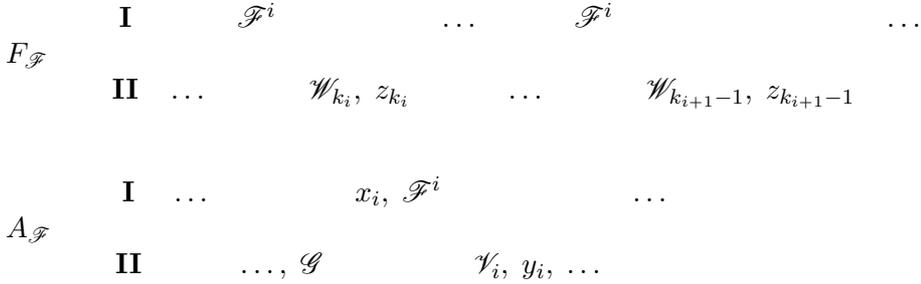
\bigskip

\begin{tabular}{ccccccccccc}
              & & \textbf{I}  &          & $\mathscr{F}^i$ &           & $\hdots$ &          & $\mathscr{F}^i$ &          & $\hdots$ \\
$F_\mathscr{F}$ & &             &          &               &           &          &          &               &          &          \\
              & & \textbf{II} & $\hdots$ &               & $\mathscr{W}_{k_i}, \; z_{k_i}$ &          & $\hdots$ &               & $\mathscr{W}_{k_{i + 1} - 1}, \; z_{k_{i + 1} - 1}$ &          \\
\end{tabular}

\bigskip\bigskip

\begin{tabular}{cccccccc}
              & & \textbf{I}  & $\hdots$ &                      & $x_i, \; \mathscr{F}^{i}$ &                                  & $\hdots$ \\
$A_\mathscr{F}$ & &             &          &                      &                       &                                  &       \\
              & & \textbf{II} &          & $\hdots, \; \mathscr{G}$ &                       & $\mathscr{V}_{i}, \; y_i, \; \hdots$ &       \\
\end{tabular}

\bigskip

We now describe how these turns are played. In $A_\mathscr{F}$, the strategy of player \II{} will first consist in playing $\mathscr{G}$. Then player \I{} answers with a vector $x_i \in [\mathscr{G}]$ and an FDD $\mathscr{F}^{i} \lessapprox \mathscr{F}$. Thus, $x_i$ can be decomposed on the basis $(u_k)_{k \in \N}$: we can find $k_{i + 1} \in \N$ and $(a_i^k)_{k < k_{i+1}} \in K^{k_{i + 1}}$ such that $x_i = \sum_{k < k_{i + 1}} a_i^ku_k$. Moreover, we can assume that $k_{i + 1} > k_i$.

\smallskip

Now, during the $k_{i + 1} - k_i$ following turns of the game $F_\mathscr{F}$, we will let player \I{} play $\mathscr{F}^i$ (so we have, for every $k_i \leqslant k < k_{i + 1}$, $\mathscr{H}^k = \mathscr{F}^i$). According to the strategy $\tau$, player \II{} will answer with $\mathscr{W}_{k_i}, \; z_{k_i}, \ldots, \mathscr{W}_{k_{i + 1} - 1}, \; z_{k_{i + 1} - 1}$. We now let $\mathscr{V}_{i} = \mathscr{W}_{k_i} + \ldots + \mathscr{W}_{k_{i + 1} - 1}$, and $y_i = \sum_{k < k_{i + 1}} a_i^kz_k$. Since all the $\mathscr{W}_{k}$'s, for $k_i \leqslant k < k_{i + 1}$ are finite-dimensional subspaces of $[\mathscr{F}^i]$, then $\mathscr{V}_{i}$ is itself a finite-dimensional subspace of $[\mathscr{F}^i]$. And since all the $z_k$'s, for $k_i \leqslant k < k_{i + 1}$, are elements of $\mathscr{W}_{0} + \ldots + \mathscr{W}_{k_{i + 1} - 1} = \mathscr{V}_{0} + \ldots + \mathscr{V}_{i}$, then $y_i$ is itself an element of $\mathscr{V}_{0} + \ldots + \mathscr{V}_{i}$. So we can let \II{} play $\mathscr{V}_{i}$ and $y_i$ in $A_\mathscr{F}$, what finishes the description of the strategy.

\smallskip

The fact that in $F_\mathscr{F}$, player \II{} always plays according to the strategy $\tau$, ensures that the sequences $(u_k)_{k \in \N}$ and $(z_k)_{k \in \N}$ are equivalent. Observe that the sequence $(x_i)_{i \in \N}$ is built from $(u_k)$ in exactly the same way that the sequence $(y_i)_{i \in \N}$ is built from $(z_k)$; so this ensures that $(x_i)_{i \in \N}$ and $(y_i)_{i \in \N}$ are equivalent, concluding this step of the proof.

\smallskip

\paragraph{Second step:} \textit{Player \II{} has a strategy $\sigma$ in $B_\mathscr{F}$ to build two equivalent sequences.}

\smallskip

Indeed, by the first step, \I{} has no strategy in $A_\mathscr{F}$ to build two inequivalent sequences; so the conclusion immediately follows from the choice of $\mathscr{F}$.

\smallskip

\paragraph{Third step:} \textit{Player \II{} has a strategy in $G_\mathscr{F}$ to build a sequence $(y_i)_{i \in \N}$ that is equivalent to $(u_i)$.}

\smallskip

This is the conclusion of the proof. We describe this strategy on a play of $G_\mathscr{F}$ that will be played simultaneously with a play of $B_\mathscr{F}$ where \II{} will play according to her strategy $\sigma$, and a play of $F_\mathscr{F}$ where \II{} will play according to her strategy $\tau$ (for a fixed $i \in \N$, the turns $i$ of all of these three games will be played at the same time). The moves of the players during the turn $i$ of the games are described in the diagram below.

\bigskip

\begin{tabular}{ccccccc}
              & & \textbf{I}  &          & $\mathscr{F}^{i}$ &                       & $\hdots$ \\
$F_\mathscr{F}$ & &             &          &               &                       &          \\
              & & \textbf{II} & $\hdots$ &               & $\mathscr{U}_{i}, \; x_i$ &          \\
\end{tabular}

\bigskip\bigskip

\begin{tabular}{cccccc}
              & & \textbf{I}  &                          & $\mathscr{U}_{i}, \; x_i, \; \mathscr{H}^{i}$ &          \\
$B_\mathscr{F}$ & &             &                          &                                       &          \\
              & & \textbf{II} & $\hdots, \; \mathscr{F}^{i}$ &                                       & $\mathscr{V}_{i}, \; y_i, \; \hdots$ \\
\end{tabular}

\bigskip\bigskip

\begin{tabular}{ccccccc}
              & & \textbf{I}  &          & $\mathscr{H}^{i}$ &                       & $\hdots$ \\
$G_\mathscr{F}$ & &             &          &               &                       &          \\
              & & \textbf{II} & $\hdots$ &               & $\mathscr{V}_{i}, \; y_i$ &          \\
\end{tabular}

\bigskip

We describe these moves more precisely. Suppose that in $G_\mathscr{F}$, player \I{} plays $\mathscr{H}^{i}$. We look at the move $\mathscr{F}^{i}$ made by \II{} in $B_\mathscr{F}$ according to her strategy $\sigma$, and we let \I{} copy this move in $F_\mathscr{F}$. In this game, according to her strategy $\tau$, player \II{} will answer with some $\mathscr{U}_{i}$ and $x_i$. Now, in $B_\mathscr{F}$, we can let \I{} answer with $\mathscr{U}_{i}$, $x_i$ and $\mathscr{H}^{i}$. In this game, according to her strategy $\sigma$, player \II{} answers with some $\mathscr{V}_{i}$ and some $y_i$. Then the strategy of player \II{} in $G_\mathscr{F}$ will consist in answering with $\mathscr{V}_{i}$ and $y_i$.

\smallskip

Let us verify that this strategy is as wanted. The outcome of the game $F_\mathscr{F}$ is the sequence $(x_i)_{i \in \N}$; the use of the strategy $\tau$ by \II{} ensures that this sequence is equivalent to $(u_i)$. The outcome of the game $B_\mathscr{F}$ is the pair of sequences $((x_i)_{i \in \N}, (y_i)_{i \in \N})$; the use by \II{} of her strategy $\sigma$ ensures that these two sequences are equivalent. We deduce that the sequences $(u_i)$ and $(y_i)$ are equivalent, concluding the proof.

\end{proof}

Now let, for every $m \in \N$, $F_m = \overline{\mathscr{F}_m}$. The sequence $(F_m)_{m \in \N}$ is an $\H$-good block-FDD of $(E_n)$ and we can let $Y = [F_m \mid m \in \N]$. By \prettyref{prop:2ndDichotoCombi}, \prettyref{thm:2ndDichoto} will be proved once we have proved the two following lemmas:

\begin{lem}\label{lem:CaracHMin}

Suppose that there exists an $\mathscr{F}$-correct sequence $(u_i)_{i \in \N}$ such that player \II{} has a strategy in $G_\mathscr{F}$ to build a sequence $(x_i)_{i \in \N}$ that is equivalent to $(u_i)$. Let $Z = [u_i \mid i \in \N]$. Then $Z$ is $\H_{\restriction Z}$-minimal.

\end{lem}

\begin{lem}\label{lem:CaracHTight}

Suppose that for every $\mathscr{F}$-correct sequence $(u_i)_{i \in \N}$, player \I{} has a strategy in $F_\mathscr{F}$ to build a sequence $(x_i)_{i \in \N}$ that is not equivalent to $(u_i)$. Then the FDD $(F_i)_{i \in \N}$ is $\H$-tight.

\end{lem}

We start with the following technical lemma:

\begin{lem}\label{lem:ReductionToGoodFDDs}

For every $U \in \H_{\restriction Y}$, there exists a $\mathscr{G} \leqslant \mathscr{F}$ such that $\overline{[\mathscr{G}]}$ isomorphically embeds into $U$.

\end{lem}

\begin{proof}
This is a consequence of \prettyref{lem:GoodBlockFDD}. Indeed, apply it to $Y_k = Y$, to $F_n^k = F_n$, to $U_k = U$, and to $D_k = S_{[\mathscr{F}]}$. Then \prettyref{lem:GoodBlockFDD} gives us a subspace $Z \subseteq Y$ generated by an $\H$-good block-FDD $(G_n)_{n \in \N}$ of $(F_n)$, such that $Z$ can be isomorphically embedded into $U$. Moreover, for every $n \in \N$, $G_n$ has a basis made of elements of $S_{[\mathscr{F}]}$, so $G_n = \overline{\mathscr{G}_n}$ for some finite-dimensional subspace $\mathscr{G}_n$ of $[\mathscr{F}]$. Hence, $\mathscr{G} = (\mathscr{G}_n)_{n \in \N}$ is a good block-FDD of $\mathscr{F}$, and $\overline{[\mathscr{G}]}$ isomorphically embeds into $U$, as wanted.
\end{proof}

\begin{proof}[Proof of \prettyref{lem:CaracHMin}]

By the definition of correctness, we have $Z \in \H$. We want to prove that $Z$ isomorphically embeds into every element of $\H_{\restriction Z}$; by \prettyref{lem:ReductionToGoodFDDs}, it is enough to prove that $Z$ isomorphically embeds into $\overline{[\mathscr{G}]}$ for every $\mathscr{G} \leqslant \mathscr{F}$. For this, consider a play of $G_\mathscr{F}$ where player \I{} always plays $\mathscr{G}$, and \II{} plays using her strategy. The outcome will be a sequence $(x_i)_{i \in \N}$ of elements of $\mathscr{G}$ which is equivalent to $(u_i)$. Thus the mapping $u_i \mapsto x_i$ uniquely extends to an isomorphic embedding $Z \to \overline{[\mathscr{G}]}$.

\end{proof}

\begin{proof}[Proof of \prettyref{lem:CaracHTight}]

By \prettyref{lem:ReductionToGoodFDDs}, it is enough to prove that every subspace of the form $\overline{[\mathscr{G}]}$, for $\mathscr{G} \leqslant \mathscr{F}$, is tight in $(F_n)$. So we fix such a $\mathscr{G}$ and we let $Z = \overline{[\mathscr{G}]}$.

\smallskip

\paragraph{First step:}\textit{ For every $K \geqslant 1$, there exists an infinite sequence of nonempty intervals of integers  $I_0^K < I_1^K < \ldots$ such that for every infinite $A \subseteq \N$ with $0 \in A$, we have $Z \nsqsubseteq_K [F_n \mid n \notin \bigcup_{k \in A} I_k^K]$.}

\smallskip

We let, for every $n \in \N$, $d_n = \sum_{m < n} \dim{\mathscr{G}_m}$, and we fix a normalized basis \linebreak $(u_i)_{d_n \leqslant i < d_{n + 1}}$ of $\mathscr{G}_n$ that is also a $2$-bounded minimal system (see \prettyref{subsec:DefNot}); this can be done by taking, first, an Auerbach basis of $\overline{\mathscr{G}_n}$, and then a small perturbation of it. The sequence $(u_i)_{i \in \N}$ we just built is $\mathscr{F}$-correct and is a  $4C$-bounded minimal system. Thus, we can fix a strategy $\tau$ for player \I{} in $F_\mathscr{F}$ to build a sequence $(x_i)_{i \in \N}$ that is not equivalent to $(u_i)$. In the game $F_\mathscr{F}$, we will consider that player \II{} is allowed to play against the rules, but immediately loses if she does; hence, we can consider that the strategy $\tau$ is a mapping defined on the whole set $\Pi^{< \N}$ of finite sequences of elements of $\Pi$. For every such sequence $s$, $\tau(s)$ is an element of $\mathbb{P}$ such that $\tau(s) \lessapprox \mathscr{F}$; hence without loss of generality, we can assume that $\tau(s) = \mathscr{F}^{(\ttau(s))}$, for some $\ttau(s) \in \N$. This defines a mapping $\ttau \colon \Pi^{< \N} \to \N$.

\smallskip

Let $R = \{x \in X \mid 1 \leqslant \|x\| \leqslant K\}$. Let $\delta > 0$ be having the following property: for every $4CK^2$-bounded minimal system $(x_i)_{i \in I}$ and for every family $(y_i)_{i \in I}$ in $X$, if: $$\sum_{i \in I}\frac{\|x_i - y_i\|}{\|x_i\|} \leqslant \delta,$$ then the families $(x_i)$ and $(y_i)$ are equivalent. For every finite-dimensional subspace $\mathscr{U}$ of $[\mathscr{F}]$ and for every $i \in \N$, we let $\mathfrak{N}_i(\mathscr{U})$ be a finite $(2^{-(i+2)}\delta)$-net in $\mathscr{U} \cap R$. Given $n \in \N$, we say that a sequence $(\mathscr{U}_0, x_0, \ldots \mathscr{U}_{i-1}, x_{i-1}) \in \Pi^{< \N}$ is \textit{$n$-small} if it satisfies the following properties:

\begin{itemize}
    
    \item there exists a sequence of successive nonempty intervals of integers $J_0 < \ldots < J_{i-1} < n$ such that for every $j < i$, $ \mathscr{U}_j = [\mathscr{F}_m \mid m \in J_j]$;
    
    \item for every $j < i$, we have $x_j \in \mathfrak{N}_j([\mathscr{U}_k \mid k \leqslant j])$.
    
\end{itemize}

\noindent For $n$ fixed, there are only finitely many $n$-small sequences. Hence we can define a sequence $(n_k)_{k \in \N}$ of integers in the following way: let $n_0 = 0$, and for $k \in \N$, choose $n_{k + 1} > n_k$ such that for every $n_k$-small sequence $s \in \Pi^{< \N}$, we have $n_{k+1} \geqslant \ttau(s)$. We now let, for every $k \in \N$, $I_k^K = \llbracket n_k, n_{k + 1} - 1\rrbracket$. We show that the sequence of intervals $I_0^K < I_1^K < \ldots$ is as wanted.

\smallskip

Suppose not. Then there exists an infinite $A \subseteq \N$ with $0 \in A$, and there exists an isomorphic embedding $T \colon Z \to [F_n \mid n \notin \bigcup_{k \in A} I_k^K]$ such that $\|T^{-1}\| = 1$ and $\|T\| \leqslant K$. In particular, the sequence $(T(u_i))_{i \in \N}$ is $K$-equivalent to $(u_i)$, so it is a $4CK^2$-bounded minimal system. We also have that, for every $i \in \N$, $1 \leqslant \|T(u_i)\| \leqslant K$. For every $i \in \N$, we fix $y_i \in [\mathscr{F}_n \mid n \notin \bigcup_{k \in A} I_k^K] \cap R$ such that $\|y_i - T(u_i)\| \leqslant 2^{-(i+2)}\delta$. Since $A$ is infinite, we can find $k_{i+1} \in A$ such that $\supp(y_i) < I_{k_{i+1}}^K$ (here, the support is taken with respect to the FDD $\mathscr{F}$). We can also let $k_0 = 0$; hence, we defined a sequence $(k_i)_{i \in \N}$ of elements of $A$. We can even assume that for every $i$, we have $k_{i + 1} \geqslant k_i + 2$. We let, for every $i \in \N$, $J_i = \llbracket n_{k_i + 1}, n_{k_{i+1}} - 1 \rrbracket$, and $\mathscr{U}_i = [\mathscr{F}_n \mid n \in J_i]$. Hence, we have a partition of $\N$ into an infinite sequence of nonempty successive intervals: $I_{k_0}^K < J_0 < I_{k_1}^K < J_1 < \ldots$. Since all the $k_i$'s are in $A$, we have that $[\mathscr{F}_n \mid n \notin \bigcup_{k \in A} I_k^K] \subseteq [\mathscr{F}_n \mid n \in \bigcup_{i \in \N}J_i]$; so all the $y_i$'s are in $[\mathscr{F}_n \mid n \in \bigcup_{i \in \N}J_i]$. Thus, for every $i \in \N$, we have $y_i \in [\mathscr{F}_n \mid n \in \bigcup_{j \leqslant i} J_j] = [\mathscr{U}_j \mid j \leqslant i]$. Hence, we can find $x_i \in \mathfrak{N}_i([\mathscr{U}_j \mid j \leqslant i])$ satisfying $\|x_i - y_i\| \leqslant 2^{-(i+2)}\delta$. In particular, $\|x_i - T(u_i)\| \leqslant 2^{-(i+1)}\delta$. So we have:
$$\sum_{i \in \N} \frac{\|x_i - T(u_i)\|}{\|T(u_i)\|} \leqslant \sum_{i \in \N} \|x_i - T(u_i)\| \leqslant \delta,$$
so since $(T(u_i))$ is a $4CK^2$-bounded minimal system, and by the choice of $\delta$, we deduce that the sequences $(T(u_i))$ and $(x_i)$ are equivalent. In particular, the sequences $(u_i)$ and $(x_i)$ are equivalent.

\smallskip

Towards a contradiction, we now prove that $(u_i)$ and $(x_i)$ are \textit{not} equivalent. For this, we first observe that for every $i \in \N$, the sequence $(\mathscr{U}_0, x_0, \ldots \mathscr{U}_{i-1}, x_{i-1})$ is $n_{k_i}$-small. Thus, letting $p_i = \ttau(\mathscr{U}_0, x_0, \ldots \mathscr{U}_{i-1}, x_{i-1})$, we deduce that $p_i \leqslant n_{k_i +1} = \min J_i$. In particular, $\mathscr{U}_i \subseteq [\mathscr{F}^{(p_i)}]$. Since, moreover, $x_i \in [\mathscr{U}_j \mid j \leqslant i]$, we deduce that in the following play of $F_\mathscr{F}$:

\smallskip

\begin{tabular}{ccccccc}
\textbf{I} & $\mathscr{F}^{(p_0)}$ & & $\mathscr{F}^{(p_1)}$ & & $\hdots$ & \\
\textbf{II} & & $\mathscr{U}_0, \, x_0$ & & $\mathscr{U}_1, \, x_1$ & & $\hdots$ 
\end{tabular}

\smallskip

\noindent player \II{} always respects the rules. Since, moreover, player \I{} plays according his strategy $\tau$, we deduce that he wins the game and that the outcome $(x_i)$ is not equivalent to $(u_i)$. This is a contradiction.

\smallskip

\paragraph{Second step:}\textit{$Z$ is tight in $(F_n)$.}

\smallskip

This is the conclusion of the proof. We keep the sequences of intervals $(I_i^K)_{i \in \N}$ built as a result of the previous step. %
We recall the following classical result: for every $d \in \N$, there exists a constant $c(d) \geqslant 1$ such that for every Banach space $U$ and for every two subspaces $V, W \subseteq U$ both having codimension $d$, $V$ and $W$ are $c(d)$-isomorphic (see \cite{FerencziRosendalOnTheNumber}, Lemma 3) -- incidentally, an upper bound $c(d) \leq 
4d(1+\sqrt{d})^2$ may be obtained, through the fact that any $d$-codimensional subspace is $(\sqrt{d}+1+\varepsilon$)-complemented for any $\varepsilon>0$ (a consequence of local reflexivity and the Kadets-Snobar theorem, Theorem 12.1.6 in \cite{kalton}) and John's result that all $d$-dimensional spaces are $\sqrt{d}$-isomorphic to $\ell_2^d$ (Theorem 12.1.4 in \cite{kalton}).

 We build a sequence $I_1 < I_2 < \ldots$ of nonempty successive intervals of integers in the following way. All the $I_l$'s, for $l < k$, being defined, we can choose $I_k$ such that:

\begin{itemize}

\item for every positive integer $N \leqslant k$, $I_k$ contains at least one interval of the sequence $(I_i^N)_{i \in \N}$;

\item $\max(I_k) \geqslant d_k + \max(\max(I_0^{N_k}), \min(I_k))$, where $d_k = \dim([F_n \mid n < \min(I_k)])$ and $N_k = \lceil kc(d_k) \rceil$.

\end{itemize}

\noindent We show that the sequence $(I_k)_{k \geqslant 1}$ witnesses the tightness of $Z$ in $(F_n)$.

\begin{claim}

For every infinite $A \subseteq \N \setminus \{0\}$ and for every $k_0 \in A$, we have:
$$\left[F_n \left|\,  n \notin \bigcup_{k \in A} I_k\right.\right] \sqsubseteq_{c(d_{k_0})} \left[F_n \left|\, n \notin I_0^{N_{k_0}} \cup \left(\bigcup_{\substack{k \in A \\ k > k_0}} I_k\right)\right.\right].$$

\end{claim}

\begin{proof}

Let $n_0 = \min I_{k_0}$, so that $d_{k_0} = \dim[F_n \mid n < n_0]$. It is enough to prove that:
$$[F_n \mid n < n_0] \oplus \left[F_n \left|\,  n \geqslant n_0, \; n \notin \bigcup_{k \in A} I_k\right.\right] \sqsubseteq_{c(d_{k_0})} \left[F_n \left|\, n \notin I_0^{N_{k_0}} \cup \left(\bigcup_{\substack{k \in A \\ k > k_0}} I_k\right)\right.\right].$$
Since $\max(I_{k_0}) \geqslant d_{k_0} + \max(\max(I_0^{N_{k_0}}), \min(I_{k_0}))$, then in particular: $$\dim[F_n \mid \max(\max(I_0^{N_{k_0}}), \min(I_{k_0})) < n \leqslant \max(I_{k_0})] \geqslant d_{k_0}.$$ So we can find a finite-dimensional subspace $H \subseteq [F_n \mid n \in I_{k_0}]$ with $I_0^{N_{k_0}} < \supp(H)$ and $\dim(H) = d_{k_0}$ (here, the supports are taken with respect to the FDD $(F_n)$). Since $k_0 \in A$, we have:
$$H \cap \left[F_n \left|\,  n \geqslant n_0, \; n \notin \bigcup_{k \in A} I_k\right.\right] = \{0\}.$$
Thus, both subspaces: $$[F_n \mid n < n_0] \oplus \left[F_n \left|\,  n \geqslant n_0, \; n \notin \bigcup_{k \in A} I_k\right.\right]$$ and $$H \oplus \left[F_n \left|\,  n \geqslant n_0, \; n \notin \bigcup_{k \in A} I_k\right.\right]$$ have codimension $d_{k_0}$ in: $$[F_n \mid n < n_0] \oplus H \oplus  \left[F_n \left|\,  n \geqslant n_0, \; n \notin \bigcup_{k \in A} I_k\right.\right],$$ so they are $c(d_{k_0})$-isomorphic. Hence, to conclude the proof, it is enough to see that: $$H \oplus \left[F_n \left|\,  n \geqslant n_0, \; n \notin \bigcup_{k \in A} I_k\right.\right] \subseteq \left[F_n \left|\, n \notin I_0^{N_{k_0}} \cup \left(\bigcup_{\substack{k \in A \\ k > k_0}} I_k\right)\right.\right].$$ The inclusion: $$H  \subseteq \left[F_n \left|\, n \notin I_0^{N_{k_0}} \cup \left(\bigcup_{\substack{k \in A \\ k > k_0}} I_k\right)\right.\right]$$ is a consequence of the fact that $\supp(H) \subseteq I_{k_0}$ and $I_0^{N_{k_0}} < \supp(H)$. And to prove the inclusion: $$\left[F_n \left|\,  n \geqslant n_0, \; n \notin \bigcup_{k \in A} I_k\right.\right] \subseteq \left[F_n \left|\, n \notin I_0^{N_{k_0}} \cup \left(\bigcup_{\substack{k \in A \\ k > k_0}} I_k\right)\right.\right],$$ it is enough to see that for all $n \geqslant n_0$, if $n \in I_0^{N_{k_0}}$, then $n \in \bigcup_{k \in A} I_k$. This is a consequence of the fact that $n_0 = \min(I_{k_0})$ and $\max(I_0^{N_{k_0}}) \leqslant \max(I_{k_0})$.

\end{proof}

We now conclude the proof of the lemma. Let $A \subseteq \N \setminus \{0\}$ be infinite and assume, towards a contradiction, that $Z \sqsubseteq \left[F_n \left| n \notin \bigcup_{k \in A} I_k\right.\right]$. Then we can choose $k_0 \in A$ such that $Z \sqsubseteq_{k_0} \left[F_n \left| n \notin \bigcup_{k \in A} I_k\right.\right]$. Using the claim and the fact that $k_0c(d_{k_0}) \leqslant N_{k_0}$, we get that: $$Z \sqsubseteq_{N_{k_0}} \left[F_n \left|\, n \notin I_0^{N_{k_0}} \cup \left(\bigcup_{\substack{k \in A \\ k > k_0}} I_k\right)\right.\right].$$ But by construction, $I_0^{N_{k_0}} \cup \left(\bigcup_{\substack{k \in A \\ k > k_0}} I_k\right)$ contains infinitely many intervals of the sequence $\left(I^{N_{k_0}}_i\right)_{i \in \N}$, including its initial term $I^{N_{k_0}}_0$. This contradicts the definition of the sequence $\left(I^{N_{k_0}}_i\right)_{i \in \N}$.

\end{proof}

\bigskip

\subsection{$\H$-minimal and $\H$-tight spaces}\label{subsec:MinimalTight}

In this section, we prove several properties of $\H$-minimal and $\H$-tight spaces. We deduce consequences of \prettyref{thm:2ndDichoto}. We start with studying $\H$-tight spaces.

\begin{defin}
We say that the D-family $\H$ is \textit{invariant under isomorphism} if for every $Y, Z \in \Subinf(X)$ such that $Y$ and $Z$ are isomorphic, we have $Y \in \H \Leftrightarrow Z \in \H$.
\end{defin}

\begin{thm}\label{thm:TNHImpliesErgodic}

\alaligne

\begin{enumerate}

\item Suppose that $X$ is $\H$-tight and that $\H$ is invariant under isomorphism. Then $X$ is ergodic.

\item Suppose that $X$ is $d$-tight. Then $X$ is ergodic.

\end{enumerate}

\end{thm}

An important consequence of \prettyref{thm:2ndDichoto} and \prettyref{thm:TNHImpliesErgodic} is the following:

\begin{cor}\label{cor:HMinErgodic}

\alaligne
\begin{enumerate}

\item Suppose that $X \in \H$, that $X$ is non-ergodic and that $\H$ is invariant under isomorphism. Then there exists $Y \in \H$ which is $\H_{\restriction Y}$-minimal.

\item Suppose that $X$ is $d$-large and non-ergodic. Then $X$ has a $d$-minimal subspace.

\end{enumerate}

\end{cor}

To prove \prettyref{thm:TNHImpliesErgodic}, we will use a sufficient condition for the reducibility of $\Ezero$ proved by Rosendal in \cite{RosendalCaracErgodic} (Theorem 15). Let $\mathbf{E}_0'$ be the equivalence relation on $\Part(\N)$ (identified with the Cantor space) defined as follows: if $A, B \in \Part(\N)$, we say that $A \mathbf{E}_0' B$ if there exists $n \in \N$ such that $\left|A \cap \llbracket 0, n \rrbracket\right| = |B \cap \llbracket 0, n \rrbracket|$ and $A \setminus \llbracket 0, n \rrbracket = B \setminus \llbracket 0, n \rrbracket$. The result proved by Rosendal is the following:

\begin{prop}\label{prop:RosendalE0}

Let $E$ be a meager equivalence relation on $\Part(\N)$, with $\mathbf{E}_0' \subseteq E$. Then $\Ezero \leqslant_B E$.

\end{prop}

To prove \prettyref{thm:TNHImpliesErgodic}, we will combine \prettyref{prop:RosendalE0} with ideas developed by Ferenczi and Godefroy in \cite{FerencziGodefroy}. In this paper, they prove that if $(e_i)_{i \in \N}$ is a basis and $X$ a Banach space, then $X$ is tight in $(e_i)$ if and only if the set of $A \subseteq \N$ such that $X \sqsubseteq [e_i \mid i \in A]$ is meager in $\Part(\N)$. This extends immediately to the case when $(e_i)$ is replaced by an FDD $(F_i)$.

\begin{proof}[Proof of \prettyref{thm:TNHImpliesErgodic}]

As usual, we only prove the result for D-families. By \prettyref{cor:GoodTightFDD}, we can find an $\H$-good, $\H$-tight FDD $(F_n)_{n \in \N}$ of $X$. We fix $(e_i)_{i \in \N}$ a sequence of elements of $X$, and a partition of $\N$ into nonempty successive intervals $J_0 < J_1 < \ldots$ such that for every $n \in \N$, $(e_i)_{i \in J_n}$ is a basis of $F_n$. For every infinite $A \subseteq \N$, we let $X_A = [e_i \mid i \in A]$, and we define the equivalence relation $E$ on $\Part(\N)$ by $A E B$ if and only if $X_A$ and $X_B$ are isomorphic. Since the mapping $A \mapsto X_A$ from $\Part(\N)$ to $\Sub(X)$ is Borel (see \prettyref{lem:PNtoSub}), it is enough to prove that $\Ezero \leqslant_B E$.

\smallskip

We have $\mathbf{E}_0' \subseteq E$: indeed, if $A \mathbf{E}_0' B$, then fixing $n \in\N$ witnessing it, we have $X_A = X_{A \setminus \llbracket 0, n \rrbracket} \oplus X_{A \cap \llbracket 0, n \rrbracket}$ and $X_B = X_{A \setminus \llbracket 0, n \rrbracket} \oplus X_{B \cap \llbracket 0, n \rrbracket}$, and moreover $\dim(X_{A \cap \llbracket 0, n \rrbracket}) = \dim(X_{B \cap \llbracket 0, n \rrbracket})$ is finite, so $X_A$ and $X_B$ are isomorphic. So, by \prettyref{prop:RosendalE0}, it is enough to prove that $E$ is meager. Since $E$ is analytic, it has the Baire property, so by Kuratowski--Ulam's theorem (see \cite{Kechris}, Theorem 8.41), it is enough to prove that for every $A \in \Part(\N)$, the $E$-equivalence class of $A$ is meager. We distinguish two cases.

\smallskip

\paragraph{First case:} \textit{$\H_{\restriction {X_A}} = \varnothing$.}

\smallskip

For all $N \in \N$, let $\mathcal{U}_N = \{B \in \Part(\N) \mid \exists n \geqslant N \; J_n \subseteq B\}$. This is a dense open subset of $\Part(\N)$, so $\mathcal{C}:=\bigcap_{N \in \N} \mathcal{U}_N$ is comeager in $\Part(\N)$. For $B \in \mathcal{C}$, the space $X_B$ contains infinitely many of the $F_n$'s. Since $(F_n)$ is an $\H$-good FDD, this implies that $\H_{\restriction {X_B}} \neq \varnothing$. Since $\H_{\restriction {X_A}} = \varnothing$ and since $\H$ is invariant under isomorphism, this implies that $X_A$ and $X_B$ are not isomorphic. Hence, the set of $B \in \Part(\N)$ such that $X_B$ is isomorphic to $X_A$ is meager in $\Part(\N)$.

\smallskip

\paragraph{Second case:} \textit{$\H_{\restriction {X_A}} \neq \varnothing$.}

\smallskip

In this case, $X_A$ has a subspace which is tight in $(F_n)$, so $X_A$ itself is tight in $(F_n)$. Let $I_0 < I_1 < \ldots$ be a sequence of intervals witnessing it. For all $k \in \N$, let $K_k = \bigcup_{n \in I_k} J_n$. We have that for every infinite $D \subseteq \N$, $X_A \nsqsubseteq \left[e_i \mid i \notin \bigcup_{k \in D} K_k\right]$. For all $N \in \N$, let $\mathcal{U}_N = \{B \in \Part(\N) \mid \exists k \geqslant N \; K_k \cap B = \varnothing\}$. This is a dense open subset of $\Part(\N)$, so $\mathcal{C}:=\bigcap_{N \in \N} \mathcal{U}_N$ is comeager in $\Part(\N)$. If $B \in \mathcal{C}$, then there exists an infinite $D \subseteq \N$ such that $X_B \subseteq \left[e_i \mid i \notin \bigcup_{k \in D} K_k\right]$. In particular, $X_B$ cannot be isomorphic to $X_A$. Hence, the set of $B \in \Part(\N)$ such that $X_B$ is isomorphic to $X_A$ is meager in $\Part(\N)$.

\end{proof}

We now study the properties of $\H$-minimal spaces.

\begin{defin}
We say that $X$ is \textit{uniformly $\H$-minimal} if $X \in \H$ and if there exists a constant $C$ such that $X$ $C$-isomorphically embeds into every element of $\H$. We say that $X$ is \textit{uniformly $d$-minimal} if it is uniformly $\mathcal{H}_d$-minimal.
\end{defin}

The statement of the following proposition was improved from a previous version of this paper thanks to an observation of O. Kurka.

\begin{prop}\label{prop:UnifMinimal}

\alaligne

\begin{enumerate}
    \item Suppose that the D-family $\H$ is invariant under isomorphisms. If $X$ is $\H$-minimal, then it is uniformly $\H$-minimal. 
    
    \item If $X$ is $d$-minimal, then it is uniformly $d$-minimal.
\end{enumerate}

\end{prop}

\begin{proof}

2. is a consequence of 1. and the fact that $\H_d$ is invariant under isomorphisms. 
%
%
To prove 1., we start with the following claim:

\begin{claim}\label{claim:UnifHMin}

There exists $Y \in \H$ which is uniformly $\H_{\restriction Y}$-minimal.

\end{claim}

\begin{proof}
Let $(\mathcal{U}_n)_{n \in \N}$ be a decreasing sequence of Ellentuck-open subsets of $\Sub(X)$ such that $\H = \bigcap_{n \in \N} \mathcal{U}_n$. The hyperplanes of $X$ are in $\H$, and they are pairwise isomorphic with a uniform constant. Thus, there exists a constant $K \geqslant 1$ such that $X$ $K$-isomorphically embeds into all of its hyperplanes. As a consequence, we get that for every $m \in \N$, $X$ $K^m$-isomorphically embeds into all of its subspaces of codimension $m$.

\smallskip

Suppose that 1. is not satisfied. We build inductively a decreasing sequence $(Y_n)_{n \in \N}$ of elements of $\H$ and an increasing sequence $(F_n)_{n \in \N}$ of finite-dimensional subspaces of $X$ in the following way. Let $Y_0 = X$ and $F_0 = \{0\}$. If $Y_n$ and $F_n$ have been defined, then by assumption, $Y_n$ is not uniformly $\H_{\restriction Y_n}$-minimal, so there exists $Y_{n + 1} \in \H_{\restriction Y_n}$ such that $Y_n$ does not $(nK^{\dim(F_n)})$-embed into $Y_{n + 1}$. The subspace $Y_{n + 1} + F_n$ is also in $\H$, so in $\mathcal{U}_n$; thus, we can choose $F_{n + 1}$ such that $F_n \subseteq F_{n + 1} \subseteq Y_{n + 1} + F_n$ and $[F_{n + 1}, Y_{n + 1} + F_n] \subseteq \mathcal{U}_n$. This achieves the induction.

\smallskip

We now let $Y = \overline{\bigcup_{n \in \N} F_n}$. For every $n \in \N$, we have $Y \subseteq Y_{n + 1} + F_n$, so $Y \in [F_{n + 1}, Y_{n + 1} + F_n] \subseteq \mathcal{U}_n$; hence, $Y \in \H$. Since $X$ is $\H$-minimal, there exists a $C$-isomorphic embedding $T \colon X \to Y$ for some constant $C$. For all $n \in \N$, let $X_n = T^{-1}(Y \cap Y_{n + 1})$. Recall that $Y \subseteq Y_{n + 1}+ F_n$; we deduce that $X_n$ has codimension at most $\dim(F_n)$ in $X$. Hence, $X$ $K^{\dim(F_n)}$-isomorphically embeds into $X_n$, so $X$ $(CK^{\dim(F_n)})$-isomorphically embeds into $Y_{n+1}$. In particular, $Y_n$ $(CK^{\dim(F_n)})$-isomorphically embeds into $Y_{n+1}$. For $n \geqslant C$, this contradicts the definition of $Y_{n + 1}$.
\end{proof}

We now prove that $X$ is uniformly $\H$-minimal. Let $Y \in \H$ be uniformly $\H_{\restriction Y}$-minimal, with constant $K$, given by \prettyref{claim:UnifHMin}. Let $C$ be such that $X$ $C$-isomorphically embeds into $Y$. If $Z$ is an arbitrary element of $\H$, then $Z$ also $C$-isomorphically embeds into $Y$, so the uniformly $\H_{\restriction Y}$-minimal space $Y$ $CK$-isomorphically embeds into $Z$, and therefore $X$ $C^2K$-isomorphically embeds into $Z$.
\end{proof}

An interesting consequence of \prettyref{prop:UnifMinimal} in the case of internal degrees is that, if $X$ is $d$-minimal, then $d$-large subspaces of $X$ are uniformly $d$-large, in the following sense:

\begin{lem}\label{lem:UnifLarge}

Suppose that $d$ is an internal degree, and that $X$ is $d$-minimal. Then there exists a mapping $\Gamma \colon \N \to \R_+$ with $\lim_{n \to \infty} \Gamma(n) = \infty$ having the following property: for every $d$-large subspace $Y \subseteq X$, and for every $n \in \N$, there exists an $n$-dimensional subspace $F \subseteq Y$ with $d(F) \geqslant \Gamma(n)$.

\end{lem}

\begin{proof}

Recall that if $d$ is an internal degree, when writing $d(F)$ for $F \in \Banfin$, we actually mean $d(X, F)$ for any $X \in \Ban$ such that $F \subseteq X$. In particular, given an isomorphism $S \colon G \to F$ for any $F, G \in \Banfin$, then $(S, S^{-1})$ is a morphism from the pair $(F, F)$ to the pair $(G, G)$, so we have $d(G) \leqslant K_d(\|S\|\cdot \|S^{-1}\|, d(F))$.

\smallskip

By \prettyref{prop:UnifMinimal}, there exists a constant $C$ such that $X$ $C$-isomorphically embeds into all of its $d$-large subspaces. For all $n \in \N$, let $\gamma(n) = \sup\{d(F)\mid F \in \Subfin(X),\linebreak \dim(F) = n\}$, which is finite by \prettyref{lem:DegFinDim}. By \prettyref{rem:BasicRemarkDegrees}, $\gamma$ is non-decreasing, and since $X$ is $d$-large, it tends to infinity. Now let, for all $n\in \N$, $$\Gamma(n) = \sup\left\{t \in \R_+ \Bigg | K_d(C, t) \leqslant \frac{\gamma(n)}{2}\right\},$$ with the convention that $\sup \varnothing = 0$. This defines a mapping $\Gamma \colon \N \to [0, \infty]$; we will see later that it actually only takes finite values.


\smallskip

We first show that $\lim_{n \to \infty} \Gamma(n) = \infty$. Fix $K \geqslant 0$. There is $n_0 \in \N$ such that $\gamma(n_0) \geqslant 2K_d(C, K)$. Now fix $n \geqslant n_0$. For all $t \leqslant K$, we have $K_d(C, t) \leqslant K_d(C, K) \leqslant \frac{\gamma(n_0)}{2} \leqslant \frac{\gamma(n)}{2}$, so by definition of $\Gamma(n)$, we have $\Gamma(n) \geqslant K$, as wanted.

\smallskip

Now, we fix $Y$ a $d$-large subspace of $X$, and $n \in \N$, and we build an $n$-dimensional subspace $F \subseteq Y$ such that $d(F) \geqslant \Gamma(n)$ (this will in particular show that $\Gamma(n)$ is finite). Let $T \colon X \to Y$ be a $C$-isomorphic embedding. Fix $G \subseteq X$ an $n$-dimensional subspace with $d(G) > \frac{\gamma(n)}{2}$. Let $F = T(G)$. Then by the remark at the beginning of the proof, $\frac{\gamma(n)}{2} < d(G) \leqslant K_d(C, d(F))$. In particular, if $t \in \R_+$ is such that $K_d(C, t) \leqslant \frac{\gamma(n)}{2}$, then $t < d(F)$. Thus, $\Gamma(n) \leqslant d(F)$, as wanted.

\end{proof}

A $d$-minimal space that is not minimal has to be saturated with $d$-small subspaces. If $d$ is an internal degree, then for such a space $X$, \prettyref{lem:UnifLarge} is quite surprising: it implies that for subspaces $Y \subseteq X$, either the degrees of finite-dimensional subspaces of $Y$ are bounded, or their maximal value grows quite fast to infinity (at least at the same speed as $\Gamma$), but no intermediate growth is possible. This suggests that the structure of finite-dimensional subspaces of such a space $X$ must be rather peculiar. We do not know any example of a $d$-minimal space that is not minimal, and this last remark makes us think that maybe, such spaces do not exist when the degree $d$ is internal.

\begin{question}\label{question:ExistdMinimal}
Does there exist an internal degree $d$ such that some infinite-dimensional spaces are $d$-large, and for which all $d$-minimal Banach spaces are minimal? Does there exist one for which there exist $d$-minimal, non-minimal Banach spaces?
\end{question}

An immediate consequence of \prettyref{lem:UnifLarge} is the following:

\begin{cor}\label{cor:MinimalAsymp}

If $d$ is an internal degree, then $d$-minimal spaces cannot be asymptotically $d$-small.

\end{cor}

As a consequence, we obtain:

\begin{thm}\label{thm:AniscaGeneralization}

If $d$ is an internal degree, then $d$-large, asymptotically $d$-small Banach spaces are ergodic.

\end{thm}

\begin{proof}
Suppose $X$ is a $d$-large, asymptotically $d$-small Banach space. Then all subspaces of $X$ are asymptotically $d$-small, so $X$ has no $d$-minimal subspaces. By \prettyref{cor:HMinErgodic}, $X$ is ergodic.
\end{proof}

Theorem \ref{thm:AniscaGeneralization} is a generalization of
Anisca's \prettyref{thm:AniscaAsympHilbert}, which corresponds to the case of the degree defined by $d_{BM}(F, \ell_2^{\dim(F)})$. This degree is studied in details in the next section.

\bigskip\bigskip

\section{The Hilbertian degree}\label{sec:Hilbert}

In this last section, we study the consequences of all the previous results in the special case of the \textit{Hilbertian degree},
that is, the internal degree defined by $d_{BM}(F, \ell_2^{\dim(F)})$, for which small spaces are exactly Hilbertian spaces, as a consequence of Kwapi\'en's theorem \cite{kwapien}. We shall denote this degree $d_2$:
$$d_2(F):=d_{BM}(F,\ell_2^{\dim(F)}).$$
To save notation, $d_2$-better FDD's will sometimes be called \textit{better FDD's} in this section.
Let us spell out that a non-Hilbertian space is therefore a $d_2$-HI space if it contains no direct sum of two non-Hilbertian subspaces, and
 $d_2$-minimal if it embeds into all of its non-Hilbertian subspaces (``minimal among non-Hilbertian spaces"). An FDD
 is $d_2$-tight if all non-Hilbertian spaces are tight in it.
In the case of the Hilbertian degree, our two dichotomies can be summarized as follows:

\begin{thm}\label{thm:GowersList}
Let $X$ be a non-Hilbertian Banach space. Then $X$ has a non-Hilbertian subspace $Y$ satisfying one of the following mutually exclusive properties:
\begin{description}
    \item[(1)] $Y$ is \MNH{} and has a $d_2$-better UFDD;
    \item[(2)] $Y$ has a $d_2$-better \TNH{} UFDD;
    \item[(3)] $Y$ is \MNH{} and $d_2$-hereditarily indecomposable;
    \item[(4)] $Y$ is \TNH{} and $d_2$-hereditarily indecomposable.
\end{description}
\end{thm}

It is clear from the definitions that if a Banach space $X$ does not contain any isomorphic copy of $\ell_2$, then the \HHP{} property is just the HI property and the $d_2$-minimality is just classical minimality. It is also easy to check that if $X$ is not $\ell_2$-saturated, then our two local dichotomies do not provide more information than the original ones.

\smallskip

In the case of $\ell_2$-saturated Banach spaces,  \prettyref{thm:GowersList} is more interesting and can be seen as the starting point of a Gowers list for $\ell_2$-saturated, non-Hilbertian spaces. It would be interesting to extend and to study more carefully this Gowers list (this could also be done in the case of other degrees). In particular, in the case of $\ell_2$-saturated spaces, the only class of those defined by \prettyref{thm:GowersList} that we know to be nonempty is \textbf{(2)}, as it will be seen in \prettyref{cor:UnconditionalTNH}.

\begin{question}
Which classes of those defined by \prettyref{thm:GowersList} contain $\ell_2$-saturated Banach spaces?
\end{question}

It would also be interesting to know where the classical $\ell_2$-saturated spaces lie in this classification. Perhaps the most iconic example of such a space is James' quasi-reflexive space \cite{James}. Another important one is Kalton-Peck twisted Hilbert space \cite{KaltonPeck} $Z_2$. Since $Z_2$ has a $2$-dimensional UFDD which is symmetric and therefore is a good UFDD, the cases (3) and (4) are excluded for subspaces of $Z_2$. Of course other twisted Hilbert spaces than $Z_2$ are also relevant. Note that Kalton proved that (non-trivial) twisted Hilbert spaces fail to have an unconditional basis \cite{Kaltonuncbasis}. Another $\ell_2$-saturated space of interest could be G. Petsoulas' space \cite{petsoulas}, whose properties have some similarities (but are
weaker) than those of $d_2$-HI spaces.
Another example was announced very recently by Argyros, Manoussakis and Motakis and will be commented upon in the subsection on \HHP{} spaces.

\begin{question}
Does James' space belong to one of the classes defined by \prettyref{thm:GowersList}? If not, in which of those classes can we find subspaces of James' space?
\end{question}

\begin{question}
Does Kalton-Peck space contain a non-Hilbertian $d_2$-minimal subspace?
\end{question}

\bigskip

\subsection{The property of minimality among non-Hilbertian spaces}

In this subsection, we study basic properties of minimality among non-Hilbertian spaces (or $d_2$-minimality). This property is particularly important in the study of ergodicity, since in the case of the Hilbertian degree, \prettyref{cor:HMinErgodic} takes the following form:

\begin{thm}\label{thm:ErgodicMNH}

Every non-ergodic, non-Hilbertian separable Banach space contains a \MNH{} subspace.

\end{thm}

In particular, Ferenczi--Rosendal's \prettyref{conj:Ergodic} reduces to the special case of \MNH{} spaces.

\smallskip

Concerning their relationship with Johnson's \prettyref{question:Johnson}, we can even say more. Indeed, the following result has been proved by Anisca \cite{AniscaUnconditional} (originally under a finite cotype hypothesis which may be removed due to, e.g., 
 \prettyref{thm:CuellarNearHilbert}).

\begin{thm}[Anisca]\label{thm:AniscaUnconditional}
A separable Banach space having finitely many different subspaces, up to isomorphism, contains an isomorphic copy of $\ell_2$.
\end{thm}

The result of Anisca is based on the construction, in unconditional spaces with finite cotype not containing copies of $\ell_2$, and for each $n$, of a subspace  having $n$-dimensional UFDD's but no UFDD of smaller dimension.


\smallskip

 In particular, this applies to Johnson spaces and we get:

\begin{prop}

Every Johnson space is \MNH{}.

\end{prop}

 In the rest of this paper, \MNH{} spaces that are not minimal will be called \textit{non-trivial \MNH{}}; these spaces are necessarily $\ell_2$-saturated. We do not know any example of a non-trivial \MNH{} space. If $X$ is such a space, then \prettyref{lem:UnifLarge} shows that there is a uniform lower bound on the growth rates of the functions $n \mapsto \sup\{d_{BM}(F, \ell_2^n) \mid F \in \Subfin(Y), \; \dim(F)=n\}$, where $Y$ ranges over non-Hilbertian subspaces of $X$. This very surprising property suggests that either non-trivial \MNH{} spaces do not exist, or the structure of their finite-dimensional subspaces is rather peculiar. Note that, however, this uniform growth property holds for the spaces $L_p$ for $2 < p < \infty$, which are not $\ell_2$-saturated (nor 
 \MNH{}) but contain copies of $\ell_2$. This is a consequence of the fact that the spaces $L_p$ for $2 < p < \infty$ are finitely representable in all of their non-Hilbertian subspaces (this can be obtained from Proposition 3.1 in \cite{PelczynskiRosenthal}).

\begin{question}\label{question:ExistMNH}
Does there exist a non-trivial \MNH{} space?
\end{question}



We now study additional properties of \MNH{} spaces, in particular those related to the existence of basic sequences. In the case of the Hilbertian degree, \prettyref{cor:MinimalAsymp} takes the following form:

\begin{prop}\label{prop:AsympMNH}

An asymptotically Hilbertian Banach space cannot be \MNH{}.

\end{prop}

\begin{exa}\label{exa:UnconditionalTNH}

Let $(p_n)_{n \in \N}$ be a sequence of real numbers greater than $1$ and tending to $2$, and let $(k_n)_{n \in \N}$ be a sequence of natural numbers tending to $\infty$ such that\linebreak $\lim_{n \to \infty} d_{BM}(\ell_{p_n}^{k_n}, \ell_2^{k_n}) = \infty$. Consider the space $X = \left(\bigoplus_{n \in \N} \ell_{p_n}^{k_n}\right)_{\ell_2}$. This space has a better UFDD and is non-Hilbertian, $\ell_2$-saturated and asymptotically Hilbertian (this last property can be obtained as a consequence of Corollary 5 in \cite{Lewis}). In particular, it cannot have a \MNH{} subspace. So by \prettyref{thm:2ndDichotoForDegrees}, some block-FDD of its UFDD is \TNH{}. 
\end{exa}

\prettyref{exa:UnconditionalTNH} shows:

\begin{cor}\label{cor:UnconditionalTNH}
The class of non-Hilbertian, $\ell_2$-saturated Banach spaces having a better  \TNH{} UFDD is nonempty.
\end{cor}

The property of being asymptotically Hilbertian is closely related to property (H) of Pisier.

\begin{defin}[Pisier, \cite{PisierWeakHilbert}]

A Banach space $X$ is said to have the property (H) if for every $\lambda  \geqslant 1$, there exists a constant $K(\lambda)$ such that for every finite, normalized, $\lambda$-unconditional basic sequence $(x_i)_{i < n}$ of elements of $X$, we have: $$\frac{\sqrt{n}}{K(\lambda)} \leqslant \left\|\sum_{i < n} x_i\right\| \leqslant K(\lambda)\sqrt{n}.$$

\end{defin}

Recall that all normalized $\lambda$-unconditional basic sequences in Hilbert spaces are $\lambda$-equivalent to the canonical basis of $\ell_2$ (see for instance \cite{kalton}, Theorem 8.3.5). A consequence is that every Hilbertian space has property (H). Thus, property (H) is a property of proximity to Hilbertian spaces. The proof of the following result of Johnson (unpublished) can be found in Pisier's paper \cite{PisierWeakHilbert}.

\begin{prop}[Johnson]

Every space with property (H) is asymptotically Hilbertian.

\end{prop}

In particular, \MNH{} spaces fail property (H). A consequence is the following: 

\begin{lem}\label{lem:MNHandPropH}

Let $X$ be a \MNH{} space. Then there exists $\lambda_0 \geqslant 1$ satisfying the following property: in every non-Hilbertian subspace $Y$ of $X$, one can find finite-dimensional subspaces $F$ with a normalized $\lambda_0$-unconditional basis for which the Banach-Mazur distance $d_{BM}(F, \ell_2^{\dim(F)})$ is arbitrarily large.

\end{lem}

\begin{proof}

By \prettyref{prop:UnifMinimal}, $X$ uniformly embeds into all of its non-Hilbertian subspaces. In particular, it is enough to prove the result in the case where $Y = X$. Let $\lambda_0$ be witnessing that $X$ fails property (H). Towards a contradiction, suppose the existence of a constant $C$ such that every finite-dimensional subspace of $X$ with a normalized\linebreak $\lambda_0$-unconditional basis is $C$-isomorphic to a Euclidean space. Let $F$ be such a subspace and $(x_i)_{i < n}$ be its unconditional basis. Choose an isomorphism $T \colon F \to \ell_2^n$ with $\|T\| \leqslant C$ and $ \|T^{-1}\| = 1$, and let $y_i = T(x_i)$ and $z_i = \frac{y_i}{\|y_i\|}$ for all $i < n$. Then $(y_i)_{i < n}$ is\linebreak $C\lambda_0$-unconditional, and so is $(z_i)_{i < n}$. Hence, $(z_i)$ is $C\lambda_0$-equivalent to the canonical basis of $\ell_2^n$. Since, for all $i < n$, we have $1 \leqslant \|y_i\|\leqslant C$, and since $(z_i)$ is $C\lambda_0$-unconditional, we have, for every sequence $(a_i)_{i < n} \in \R^n$: $$\frac{1}{C\lambda_0}\left\|\sum_{i < n} a_i z_i\right\| \leqslant \left\|\sum_{i < n} a_i y_i\right\| \leqslant C^2 \lambda_0 \left\|\sum_{i < n} a_i z_i\right\|,$$ hence $(y_i)$ and $(z_i)$ are $C^2\lambda_0$-equivalent. Moreover, we know that $(x_i)$ and $(y_i)$ are $C$-equivalent. We deduce that $(x_i)$ is $C^4\lambda_0^2$-equivalent to the canonical basis of $\ell_2^n$. In particular, for $K = C^4\lambda_0^2$, we have: $$\frac{\sqrt{n}}{K} \leqslant \left\|\sum_{i < n} x_i\right\| \leqslant K\sqrt{n},$$ contradicting the choice of $\lambda_0$.

\end{proof}

\begin{thm}\label{thm:BasesMNH}

\alaligne

\begin{enumerate}

\item Every \MNH{} space has a non-Hilbertian subspace with a Schauder basis.

\item Every \MNH{} space having an unconditional FDD has a non-Hilbertian subspace with an unconditional basis.

\end{enumerate}

\end{thm}

A consequence of this theorem is that the alternative \textbf{(1)} in \prettyref{thm:GowersList} can be replaced with ``$Y$ is \MNH{} and has an unconditional basis''.

Knowing that every non-Hilbertian subspace of a Johnson space is isomorphic to the space itself, another consequence is:

\begin{cor}\label{cor:Johnsonbasis}

Every Johnson space has a Schauder basis. Moreover it has an unconditional basis if and only if it is isomorphic to its square.

\end{cor}

\begin{proof} If it has an unconditional basis then it is isomorphic to its square by Theorem \ref{square}. Conversely if it is isomorphic to its square then it is not $d_2$-HI and by the first local dichotomy (\prettyref{thm:FirstDichotofordegrees}), it must have a UFDD. It follows from \prettyref{thm:BasesMNH} that the space  has an unconditional basis.\end{proof}

A few additional restrictions on the existence of Johnson spaces follow from \prettyref{cor:Johnsonbasis}. 
Every Johnson space is HAPpy (every subspace has the Approximation Property).
If a Johnson space $X$ has an unconditional basis then it is reflexive, all its subspaces have GL-lust and therefore the GL-property,
so $X$ has weak cotype $2$ (Theorem 40 in \cite{MT}). On the other hand since $X$ is not weak Hilbert, 
$X$ cannot have weak type $2$ in this case (see \cite{PisierWeakHilbert} for these notions). 
For non-Hilbertian examples of
HAPpy spaces with a symmetric basis (and therefore also non asymptotically Hilbertian), see \cite{JohnsonSzankowski}. 

\smallskip

\prettyref{thm:BasesMNH} naturally opens the following two questions (the first one had already been asked by Pe\l czy\'nski \cite{Pelczynski}):

\begin{question}

\alaligne

\begin{enumerate}

\item Does every non-Hilbertian space have a non-Hilbertian subspace with a Schauder basis?

\item Does every non-Hilbertian space with unconditional FDD have a non-Hilbertian subspace with an unconditional basis?

\end{enumerate}

\end{question}

\begin{proof}[Proof of \prettyref{thm:BasesMNH}]

\alaligne

\begin{enumerate}
    
    \item Let $X$ be a \MNH{} space, and fix $\lambda_0$ as given by \prettyref{lem:MNHandPropH} for $X$. Build an FDD $(F_n)_{n \in \N}$ of a subspace of $X$, along with a decreasing sequence $(Y_n)_{n \in \N}$ of finite-codimensional subspaces of $X$, by induction as follows. Let $Y_0 = X$. The subspace $Y_n$ and all the $F_m$'s, for $m < n$, being built, we can find $F_n \subseteq Y_n$ with a normalized $\lambda_0$-unconditional basis such that $d_{BM}(F_n, \ell_2^{\dim(F_n)}) \geqslant n$. We then find a finite-codimensional subspace $Y_{n + 1} \subseteq Y_n$ with $Y_{n+1} \cap [F_m \mid m \leqslant n] = \{0\}$, such that the first projection $[F_m \mid m \leqslant n] \oplus Y_{n+1} \to [F_m \mid m \leqslant n]$ has norm at most $2$. This finishes the induction.
    
    \smallskip
    
    The sequence $(F_n)_{n \in \N}$ we just built is an FDD of a non-Hilbertian subspace $Y$ of $X$. It has constant at most $2$, and all the $F_n$'s have a basis with constant at most $\lambda_0$. Thus, concatenating these bases, we get a basis of $Y$ with constant at most $2+4\lambda_0$, as wanted.
    
    \smallskip
    
    \item Let $X$ be a \MNH{} space with a UFDD $(F_n)_{n \in \N}$. If $(F_n)$ has a normalized block-sequence spanning a non-Hilbertian subspace, then we are done. So from now on, we assume that every normalized block-sequence of $(F_n)$ spans a Hilbertian subspace. Since normalized block-sequences of $(F_n)$ are unconditional, we deduce that all of them are equivalent to the canonical basis of $\ell_2$. Our first step is to prove that this holds uniformly.
    
    \begin{claim}
    
    There exists a constant $C$ satisfying the following property: every normalized block-sequence of $(F_n)$ is $C$-equivalent to the canonical basis of $\ell_2$.
    
    \end{claim}

    \begin{proof}
    
    We prove the formally weaker, but actually equivalent, following statement: there exist $n_0 \in \N$ and a constant $C$ such that every normalized block-sequence of $(F_n)_{n \geqslant n_0}$ is $C$-equivalent to the canonical basis of $\ell_2$. Suppose that this does not hold. Then for every $n_0, N \in \N$ we can find a finite normalized block-sequence $(x_i)_{i < i_0}$ of $(F_n)_{n \geqslant n_0}$ which is not $N$-equivalent to the canonical basis of $\ell_2^{i_0}$. Applying this for successive values of $N$, we can build by induction a normalized block-sequence $(x_i)_{i \in \N}$ of $(F_n)$ and a sequence $0 = i_0 < i_1 < i_2 < \ldots$ such that for every $N \in \N$, the sequence $(x_i)_{i_N \leqslant i < i_{N+1}}$ is not $N$-equivalent to the canonical basis of $\ell_2^{i_{N+1} - i_N}$. In particular, $(x_i)_{i \in \N}$ is not equivalent to the canonical basis of $\ell_2$, a contradiction.
    
    \end{proof}
    
    We now finish the proof of \prettyref{thm:BasesMNH}, proceeding similarly as in 1. Fix $\lambda_0$ as given by \prettyref{lem:MNHandPropH} for $X$. Observe that for every $n_0 \in \N$, we can find a finite-dimensional subspace $G \subseteq [F_n \mid n \geqslant n_0]$, finitely supported on the FDD $(F_n)$, having a normalized $2\lambda_0$-unconditional basis and such that $d_{BM}(G, \ell_2^{\dim(G)})$ is arbitrarily large: indeed, it is enough to take a small perturbation of a (non-necessarily finitely supported) finite-dimensional subspace $G' \subseteq [F_n \mid n \geqslant n_0]$ with a normalized $\lambda_0$-unconditional basis and large $d_{BM}(G', \ell_2^{\dim(G')})$. Using this remark, we can build a better block-FDD $(G_k)_{k \in \N}$ of $(F_n)$ such that all of the $G_k$'s have a $2\lambda_0$-unconditional basis. Let $i_k = \sum_{l < k} \dim{G_l}$ for every $k \in \N$, and denote by $(x_i)_{i_k \leqslant i < i_{k + 1}}$ the unconditional basis of $G_k$. To conclude the proof, it is enough to prove that the sequence $(x_i)_{i \in \N}$ is unconditional.
    
    \smallskip
    
    So let $(a_i)_{i \in \N}$ be a finitely supported sequence of real numbers and $(\varepsilon_i)_{i \in \N}$ be a sequence of signs. For every $k \in \N$, let $b_k, c_k \geqslant 0$ and $y_k, z_k \in S_{G_k}$ be such that $b_ky_k=\sum_{i_k \leqslant i < i_{k + 1}} a_ix_i$ and $c_kz_k=\sum_{i_k \leqslant i < i_{k + 1}} \varepsilon_ia_ix_i$. In particular, we have $b_k = \left\|\sum_{i_k \leqslant i < i_{k + 1}} a_ix_i\right\|$ and $c_k = \left\|\sum_{i_k \leqslant i < i_{k + 1}} \varepsilon_ia_ix_i\right\|$, so since the sequence $(x_i)_{i_k \leqslant i < i_{k+1}}$ is $2\lambda_0$-unconditional, we have that $c_k \leqslant 2\lambda_0b_k$. Now, since $(y_k)_{k \in \N}$ and $(z_k)_{k \in \N}$ are normalized block-sequences of $(F_n)$, they are $C$-equivalent to the canonical basis of $\ell_2$. Thus, we have:
    \begin{eqnarray*}
    \left\|\sum_{i \in \N}\varepsilon_i a_i x_i\right\| & = & \left\|\sum_{k \in \N} c_k z_k\right\| \\
    & \leqslant & C \cdot \sqrt{\sum_{k \in \N} c_k^2} \\
    & \leqslant & 2\lambda_0C \cdot \sqrt{\sum_{k \in \N} b_k^2} \\
    & \leqslant & 2\lambda_0C^2 \cdot \left\|\sum_{i \in \N}a_i x_i\right\|,
    \end{eqnarray*}
    proving that the sequence $(x_i)_{i \in \N}$ is $2\lambda_0C^2$-unconditional.
    
\end{enumerate}

\end{proof}

\bigskip

\subsection{Properties of \HHP{} spaces}

Recall that \HHP{} spaces are non-Hilbertian Banach spaces that do not contain any direct sum of two non-Hilbertian subspaces. HI spaces are of course \HHP{}. We could only discover
two other examples of \HHP{} spaces. Before presenting them, we recall a basic result in operator theory. For its proof, see \cite{MaureyHI}, Proposition 3.2. The terminology of the next definition is from \cite{GowersMaurey}.

\begin{defin}
An operator $T \colon X \to Y$  between two Banach spaces is
\textit{infinitely singular} if there is no finite-codimensional subspace $X_0 \subseteq X$ such that $T_{\restriction X_0} \colon X_0 \to T(X_0)$ is an isomorphism.
\end{defin}

\begin{prop}[Folklore]\label{prop:InfSing}

An operator $T \colon X \to Y$ between two Banach spaces is infinitely singular if and only if for every $\varepsilon > 0$, there exists a subspace $X_\varepsilon \subseteq X$ such that $\left\|T_{\restriction X_\varepsilon}\right\| \leqslant \varepsilon$.

\end{prop}

\begin{exa}\label{exa:HIplusl2}

Let $Y$ be an HI space. Then $X = Y \oplus \ell_2$ is \HHP{}. Indeed, denote by $p_Y \colon X \to Y$ and $p_{\ell_2}\colon X \to \ell_2$ the two projections. Suppose that two non-Hilbertian subspaces $U, V \subseteq X$ are in direct sum. Then $\left(p_{\ell_2}\right)_{\restriction U}$ and $\left(p_{\ell_2}\right)_{\restriction V}$ are infinitely singular, so by \prettyref{prop:InfSing}, we can find subspaces $U' \subseteq U$ and $V' \subseteq V$ on which $p_{\ell_2}$ has arbitrarily small norm. In particular, $U'$ and $V'$ can be chosen in such a way that $\left\|\left(p_{\ell_2}\right)_{\restriction U' \oplus V'}\right\| \leqslant \frac{1}{2}$. Thus, $p_Y$ induces an isomorphism between $U' \oplus V'$ and $p_Y(U' \oplus V')$. In particular, $p_Y(U')$ and $p_Y(V')$
 are two subspaces of $Y$ that are in direct sum, contradicting the fact that $Y$ is HI.
\end{exa}

\begin{exa}\label{exa:ArgyrosRaikofstalis}

In \cite{ArgyrosRaikoftsalis}, Argyros and Raikoftsalis build, for every $1 \leqslant p < \infty$ (resp. for $p=\infty$) a space $\mathfrak{X}_p$ having the following properties: $\mathfrak{X}_p \cong \mathfrak{X}_p\oplus\ell_p$ (resp. $\mathfrak{X}_p \cong \mathfrak{X}_p\oplus c_0$), and for every decomposition as a direct sum $\mathfrak{X}_p = Y \oplus Z$, then $Y\cong \mathfrak{X}_p$ and $Z \cong \ell_p$ (resp. $Z \cong c_0$), or vice-versa. The space $\mathfrak{X}_p$ is built as an \emph{HI Schauder sum} of copies of $\ell_p$ (resp. $c_0$); the construction of such a sum is quite involved and is exposed in \cite{ArgyrosFelouzis}, Section 7. In \cite{ArgyrosRaikoftsalis}, the following results are proved for the space $\mathfrak{X}_p$:
\begin{enumerate}
\item $\mathfrak{X}_p$ does not contain any direct sum of two HI subspaces (see the proof of Lemma 1 in \cite{ArgyrosRaikoftsalis});
\item for every subspace $Y \subseteq \mathfrak{X}_p$ not containing any HI subspace, and for every $\varepsilon > 0$, there exists a projection $P$ of $\mathfrak{X}_p$ with image isomorphic to $\ell_p$ (resp. to $c_0$) such that $\left\|\left(\Id_{\mathfrak{X}_p} - P\right)_{\restriction Y}\right\| \leqslant \varepsilon$ (see Lemma 3 in \cite{ArgyrosRaikoftsalis}).
\end{enumerate}
This implies that $\mathfrak{X}_2$ is \HHP{}. Indeed, if two subspaces $Y, Z \subseteq \mathfrak{X}_2$ are in direct sum, then by 1., one of them does not contain any HI subspace, for example $Y$. Choosing a projection $P$ as given by 2. for $\varepsilon = \frac{1}{2}$, we get that $P_{\restriction Y}$ is an isomorphism onto its image, which is contained in an isomorphic copy of $\ell_2$; so $Y$ is Hilbertian.

\end{exa}

The interest of \HHP{} spaces in the study of ergodicity, and in particular of our conjectures \prettyref{conj:WeakJohnson} and \prettyref{conj:WeakErgodic}, comes from the following result:

\begin{thm}\label{thm:MNHHHP}

Let $X$ be a non-ergodic, non-Hilbertian separable Banach space. Then $X$ has a non-Hilbertian subspace $Y$ such that:
\begin{itemize}
    \item either $Y$ has an unconditional basis;
    \item or $Y$ is simultaneously \MNH{} and \HHP{}.
\end{itemize}

\end{thm}

\begin{proof}
By \prettyref{cor:HMinErgodic}, we can assume that $X$ is \MNH{}. By \prettyref{thm:FirstDichotofordegrees}, either $X$ has a subspace with a better UFDD, or a \HHP{} subspace. In first case, \prettyref{thm:BasesMNH} shows that we can find a further non-Hilbertian subspace having an unconditional basis.
\end{proof}

It would of course be interesting to remove the second alternative, thus reducing somehow the problem to spaces with unconditional bases. This motivates the following question:

\begin{question}\label{question:MNHHHPErg}
Does there exist a non-ergodic Banach space which is simultaneously \MNH{} and \HHP{}?
\end{question}


\smallskip

Both examples of \HHP{} spaces given above contain an HI subspace. In particular, they are ergodic, and they cannot be $d_2$-minimal. Thus, \prettyref{question:MNHHHPErg} reduces to the special case of \HHP{} spaces that do not contain any HI subspace. The latter spaces are exactly those \HHP{} spaces that are $\ell_2$-saturated. We know no examples of such spaces.

\begin{question}\label{ExistHHP}
Do there exist $\ell_2$-saturated \HHP{} spaces?
\end{question}

After this article was submitted, Argyros, Manoussakis and Motakis announced in \cite{ArgyrosManoussakisMotakis} that they were able to build an $\ell_2$-saturated \HHP{} space, thus giving a positive answer to Question \ref{ExistHHP}. The construction will be published in a forthcoming paper, and the preprint \cite{ArgyrosManoussakisMotakis} exposes the construction of analogues of that space.

\smallskip

We now come back to \prettyref{question:MNHHHPErg}. We conjecture that the answer to this question is negative, and we actually have the following stronger conjecture:

\begin{conj}\label{conj:MNHHHP}
A Banach space cannot be simultaneously \MNH{} and \HHP{}.
\end{conj}

This conjecture is motivated by the fact that the \HHP{} property is a weakening of the HI property, and it is known that HI spaces have many different subspaces, up to isomorphism. For example, Gowers--Maurey's \prettyref{thm:GowersMaurey} says that HI spaces cannot be isomorphic to any proper subspace of themselves. This implies, in particular, that they cannot be minimal. It would be tempting to adapt Gowers--Maurey's approach to \HHP{} spaces. Note that, however, in the case of \HHP{} spaces, we cannot hope to have a result as strong as Gowers--Maurey's one, since both spaces presented in \prettyref{exa:HIplusl2} and in \prettyref{exa:ArgyrosRaikofstalis} are isomorphic to their hyperplanes and even, to their direct sum with $\ell_2$. However, we can hope that these spaces cannot be isomorphic to ``too deep'' subspaces of themselves. This is at least the case for our first example, as shown by the following lemma:

\begin{lem}\label{lem:HIPlusl2}

Let $Y$ be an HI space and let $X = Y \oplus \ell_2$. Then every subspace of $X$ that is isomorphic to $X$ is complemented in $X$ by a (finite- or infinite-dimensional) Hilbertian subspace.

\end{lem}

\begin{proof}
Denote by $P_Y \colon X \to Y$ and $P_{\ell_2} \colon X \to \ell_2$ the projections. Let $U \subseteq X$ be an isomorphic copy of $X$; we can write $U = V \oplus W$, where $V \cong Y$ and $W \cong \ell_2$. Suppose that $(P_Y)_{\restriction V}$ is infinitely singular. Then by \prettyref{prop:InfSing}, we can find a subspace $V'' \subseteq V$ on which $P_Y$ has small norm. In particular, $P_{\ell_2}$ would induce an isomorphism between $V''$ and a subspace of $\ell_2$, a contradiction. Thus, $(P_Y)_{\restriction V}$ is not  infinitely singular: we can find a finite-codimensional subspace $V'$ of $V$ such that $P_Y$ induces an isomorphism between $V'$ and $P_Y(V')$.

\smallskip

Observe that $V' \cong P_Y(V')$, and that $V'$ and $P_Y(V')$ are respectively subspaces of $V$ and $Y$, that are HI and isomorphic. By Gowers--Maurey's \prettyref{thm:GowersMaurey}, we deduce that the codimension of $P_Y(V')$ in $Y$ is equal to the codimension of $V'$ in $V$, so is finite. So write $Y = P_Y(V') \oplus F$, where $F$ has finite dimension. We have $P_Y(V') \subseteq V + \ell_2$, so $Y \subseteq V + \ell_2 + F$, so $X = V + \ell_2 + F = U + \ell_2 + F$. Letting $Z$ be a complement of $U \cap (\ell_2 + F)$ in $\ell_2 + F$, we get that $Z$ is Hilbertian and that $X = U \oplus Z$.

\end{proof}


\begin{question}\label{question:SubHHP}
Let $X$ be \HHP{} and let $Y$ be a subspace of $X$ which is isomorphic to $X$. Does it follow that $Y$ is complemented by a (finite- or infinite-dimensional) Hilbertian subspace?
\end{question}


\bigskip
\paragraph{Acknowledgments.} The authors would like to thank F. Le Ma\^itre and S. Todorcevic for useful suggestions, and Ch. Rosendal and O. Kurka for comments
on drafts of the paper. They also thank the anonymous referees for comments leading to the improvement of this article.

\bigskip\bigskip

\bibliographystyle{plain}
\bibliography{main}

 \par 
  \bigskip
  \textsc{\footnotesize W. Cuellar Carrera, 
Departamento de Matem\'atica, Instituto de Matem\'atica e
Estat\'\i stica, Universidade de S\~ao Paulo, rua do Mat\~ao 1010,
05508-090 S\~ao Paulo SP, BRAZIL}

\textit{E-mail address}: \texttt{cuellar@ime.usp.br}

  \par
  \bigskip
    \textsc{\footnotesize N. de Rancourt, Universit\"at Wien, Institut f\"ur Mathematik,
    Kurt G\"odel Research Center, Augasse 2-6, UZA 1 -- Building 2, 1090 Wien, AUSTRIA}
    
    {\small \textit{Current adress:}
    Charles University,
    Faculty of Mathematics and Physics,
  Department of Mathematical Analysis,
  Sokolovsk\'a 49/83,
  186 75 Praha 8,
  CZECH REPUBLIC
    
    \textit{E-mail address}: \texttt{rancourt@karlin.mff.cuni.cz}
    
  \par 
  \bigskip
  \textsc{\footnotesize V. Ferenczi, 
Departamento de Matem\'atica, Instituto de Matem\'atica e
Estat\'\i stica, Universidade de S\~ao Paulo, rua do Mat\~ao 1010,
05508-090 S\~ao Paulo SP, BRAZIL,
and \\
Equipe d'Analyse Fonctionnelle,
Institut de Math\'ematiques de Jussieu,
Sorbonne Universit\'e - UPMC,
Case 247, 4 place Jussieu,
75252 Paris Cedex 05,
FRANCE}

\textit{E-mail address}: \texttt{ferenczi@ime.usp.br}}

\end{document}